\documentclass[reqno,final,oneside]{amsart}
\usepackage{amsaddr}
\usepackage[textwidth=16.1cm]{geometry}

\usepackage{mathrsfs}  
\usepackage{cancel} 
\usepackage[dvipsnames]{xcolor}
\usepackage{enumitem}
\usepackage{mathtools}
\usepackage[final]{hyperref}
\usepackage{bm}
\usepackage{subcaption}

\usepackage{comment}
\usepackage{xargs}
\usepackage{soul}

\newtheorem{theorem}{Theorem}[section]
\newtheorem{lemma}[theorem]{Lemma}
\newtheorem{proposition}[theorem]{Proposition}

\theoremstyle{definition}
  \newtheorem{definition}[theorem]{Definition}

\theoremstyle{remark}
  \newtheorem{remark}[theorem]{Remark}

\newcommand{\lp}{\varepsilon}
\newcommand{\N}{\mathbb{N}}

\newcommand{\R}{\mathbb{R}}
\newcommand{\C}{\mathbb{C}}
\newcommand{\Id}{\mathbf{Id}}
\newcommand{\id}{\mathbf{id}}
\newcommand{\eps}{\varepsilon}
\newcommand{\vphi}{\varphi}
\newcommand{\weakly}{\rightharpoonup}
\newcommand{\weaklystar}{\stackrel{*}{\rightharpoonup}}
\newcommand{\defas}{\coloneqq}
\newcommand{\sym}{\mathrm{sym}}
\newcommand{\elen}{\mathcal{W}^{\mathrm{el}}}
\newcommand{\cplen}{\mathcal{W}^{\mathrm{cpl}}}
\newcommand{\hyen}{\mathcal{H}}
\newcommand{\mechen}{\mathcal{M}}
\newcommand{\toten}{\mathcal{E}}
\newcommand{\inten}{W^{\mathrm{in}}}
\newcommand{\tempen}{\mathcal{T}}
\newcommand{\diss}{\mathcal{R}}
\newcommand{\elpot}{W^{\mathrm{el}}}
\newcommand{\cplpot}{W^{\mathrm{cpl}}}
\newcommand{\hypot}{H}
\newcommand{\felpot}{W}
\newcommand{\disspot}{R}
\newcommand{\Wid}{\mathcal{Y}_{\id}}
\newcommand{\Wzero}{W^{2, p}_{\Gamma_D}(\Omega; \R^d)}
\newcommand{\pl}{\partial}

\newcommand{\yst}[1]{y_{\lp,\tau}^{(#1)}}
\newcommand{\ysts}[1]{y_{\tau}^{(#1)}}
\newcommand{\tst}[1]{\theta_{\lp,\tau}^{(#1)}}
\newcommand{\tsts}[1]{\theta_{\tau}^{(#1)}}
\newcommand{\lst}[1]{\ell_{\tau}^{(#1)}}
\newcommand{\wst}[1]{w_{\lp,\tau}^{(#1)}}

\newcommand{\fst}[1]{f_\tau^{(#1)}}

\newcommand{\gst}[1]{g_\tau^{(#1)}}
\newcommand{\ust}[1]{u^{(#1)}}
\newcommand{\bt}{\theta_\flat}
\newcommand{\btst}[1]{\theta_{\flat, \tau}^{(#1)}}
\newcommand{\hc}{\mathbb{K}}
\newcommand{\hcm}{\mathcal{K}}
\newcommand{\Kst}[1]{\hcm_{\lp, \tau}^{(#1)}}
\newcommand{\drate}{\xi}

\newcommand{\haus}{\mathcal{H}}
\newcommand{\ddif}{\delta_\tau}
\newcommand{\aC}{C_0}
\newcommand{\ac}{c_0}

\newcommand{\ny}{\overline{y}_{\lp, \tau}}
\newcommand{\py}{\underline{y}_{\lp, \tau}}
\newcommand{\ay}{\hat{y}_{\lp, \tau}}
\newcommand{\dotay}{\dot{\hat{y}}_{\lp, \tau}}
\newcommand{\nyk}{\overline{y}_k}
\newcommand{\pyk}{\underline{y}_k}
\newcommand{\ayk}{\hat{y}_k}
\newcommand{\awk}{\hat{w}_k}
\newcommand{\dotayk}{\dot{\hat{y}}_k}
\newcommand{\nt}{\overline{\theta}_{\lp, \tau}}
\newcommand{\pt}{\underline{\theta}_{\lp, \tau}}
\newcommand{\at}{\hat{\theta}_{\lp, \tau}}

\newcommand{\ntk}{\overline{\theta}_k}
\newcommand{\ptk}{\underline{\theta}_k}

\newcommand{\nuk}{\overline{u}_k}
\newcommand{\puk}{\underline{u}_k}
\newcommand{\auk}{\hat{u}_k}
\newcommand{\dotauk}{\dot{\hat{u}}_k}
\newcommand{\nmuk}{\overline{\mu}_k}
\newcommand{\pmuk}{\underline{\mu}_k}
\newcommand{\amuk}{\hat\mu_k}

\newcommand{\nw}{\overline{w}_{\lp, \tau}}
\newcommand{\pw}{\underline{w}_{\lp, \tau}}
\newcommand{\aw}{\hat{w}_{\lp, \tau}}
\newcommand{\dotaw}{\dot{\hat{w}}_{\lp, \tau}}
\newcommand{\intQ}{\int_0^T\int_\Omega}

\newcommand{\intSN}{\int_0^T\int_{\Gamma_N}}

\newcommand{\CW}{\mathbb{C}_W}
\newcommand{\CD}{\mathbb{C}_D}
\newcommand{\mechenl}{\overline{\mechen}_0}
\newcommand{\rdrate}{\drate_\alpha^{\rm{reg}}}
\usepackage{stackengine}
\newcommand{\cdddot}{\mathrel{\Shortstack{{.} {.} {.}}}}
\newcommand*{\di}{\mathop{}\!\mathrm{d}}

\DeclareMathOperator{\dist}{dist}
\DeclareMathOperator*{\argmin}{argmin}
\DeclareMathOperator{\trace}{tr}
\DeclareMathOperator{\diver}{div}

\DeclarePairedDelimiterX\setof[1]\{\}{#1}
\DeclarePairedDelimiterX\abs[1]\lvert\rvert{#1}
\DeclarePairedDelimiterX\norm[1]\lVert\rVert{#1}
\DeclarePairedDelimiterX\sprod[2]\langle\rangle{#1, #2}

\newcommand{\NNN}{\color{black}}

\newcommand{\CCC}{\color{black}}
\newcommand{\BBB}{\color{black}}
\newcommand{\MMM}{\color{black}}

\numberwithin{equation}{section}
\begin{document}
\title[Nonlinear and linearized models in thermoviscoelasticity]{Nonlinear and linearized models in thermoviscoelasticity}


\author[R.~Badal]{Rufat Badal}
\address[Rufat Badal]{
  Department of Mathematics, \\
  Friedrich-Alexander Universit\"at Erlangen-N\"urnberg, \\
  Cauerstr.~11, D-91058 Erlangen, Germany
}
\email{rufat.badal@fau.de}

\author[M.~Friedrich]{Manuel Friedrich} 
\address[Manuel Friedrich]{%
  Department of Mathematics, \\
  Friedrich-Alexander Universit\"at Erlangen-N\"urnberg, \\
  Cauerstr.~11, D-91058 Erlangen, Germany, \\
  \& Mathematics M\"{u}nster, \\
  University of M\"{u}nster, \\
  Einsteinstr.~62, D-48149 M\"{u}nster, Germany
}
\email{manuel.friedrich@fau.de}

\author[M.~Kru\v{z}\'ik]{Martin Kru\v{z}\'ik}
\address[Martin Kru\v{z}\'ik]{
  Czech Academy of Sciences, \\
  Institute of Information Theory and Automation, \\
  Pod vod\'arenskou v\v{e}\v{z}\'i 4, CZ-182 08 Praha 8, Czechia (corresponding address) \\
  \& Faculty of Civil Engineering, \\
  Czech Technical University, \\
  Th\'akurova 7, CZ-166 29 Praha 6, Czechia}
\email{kruzik@utia.cas.cz}

\begin{abstract}
We consider a quasistatic nonlinear  model  in thermoviscoelasticity at a finite-strain setting in the Kelvin-Voigt rheology where both the elastic and viscous stress tensors comply with the principle of frame indifference under rotations. The force balance is formulated in the reference configuration by resorting to the concept of nonsimple materials whereas the heat transfer equation is \CCC governed \BBB by the Fourier law in the deformed configurations. Weak solutions are obtained by means of a staggered in-time discretization where the deformation and the temperature are updated alternatingly. Our result refines a recent work by {\sc Mielke \& Roub\'{\i}\v{c}ek} \cite{MielkeRoubicek20Thermoviscoelasticity} since our approximation does not require any regularization of the viscosity term. Afterwards, we focus on the case of deformations  near the identity and small temperatures,  and we show by a rigorous linearization procedure that weak solutions of the nonlinear system converge in a suitable sense to solutions of a system in linearized thermoviscoelasticity.  The  same property holds for time-discrete approximations and  we provide a corresponding commutativity result.   
\end{abstract}

\maketitle

\section{Introduction}

Nonlinear and large strain continuum mechanics has become a thriving field of research over the last decades which is still subject of important advancements and, at the same time, offers many challenging open questions. For instance, rigorous studies on large strain viscoelastic materials \cite{FiredrichKruzik18Onthepassage,  kroemrou, Lewick, MielkeOrtnerSenguel14Anapproach}   or nonlinear models in thermoviscoelasticity \cite{MielkeRoubicek20Thermoviscoelasticity} have been initiated only recently. Besides analytical intricacies, the usage of large strain models in engineering practice is often impeded due to nonconvex behavior that complicates numerical implementations. At many occasions, however, linearized models are still sufficient to describe observed phenomena and are significantly easier to treat, both analytically and numerically. Roughly speaking, heuristic calculations suggest that, if the deformation of the body is ``close'' to the identity, nonlinear models can be replaced by  linear ones with a negligible error. Clearly, the  reliability of such predictions depends on \NNN the \BBB rigorous derivation of  simplified linearized models, e.g., via $\Gamma$-convergence \cite{Braides:02, DalMaso93AnIntroduction}. This is an intensive research program that has been initiated in the context of linearized elastostatics in \cite{DalMasoNegriPercivale02Linearized}. Subsequently, this work was extended in various directions, among others, models for   incompressible materials \cite{JesenkoSchmidt21Geometric, MaininiPercivale20Variational, MaininiPercivale21Linearization},  atomistic models \cite{Braides-Solci-Vitali:07, Schmidt:2009}, or problems without Dirichlet boundary conditions \cite{MaorMora21Reference} have been considered. For  multiwell energies allowing for phase transitions we refer to \cite{alicandro.dalmaso.lazzaroni.palombaro, DavoliFriedrich20Two-well, Schmidt08Linear}, and we mention also settings beyond elasticity such as plasticity \cite{Ulisse} or fracture \cite{Friedrich:15-2, higherordergriffith}. As to evolutionary  models, we refer to \cite{FiredrichKruzik18Onthepassage} where viscoelasticity in the Kelvin-Voigt rheology and its linearized version are treated.

The goal of this contribution is to couple the nonlinear equations of viscoelasticity with a heat transfer equation. We first analyze a corresponding frame-indifferent and thermodynamically-consistent model of thermoviscoelasticity at large strains, and  refine the results obtained recently by {\sc Mielke \& Roub\'{\i}\v{c}ek} \cite{MielkeRoubicek20Thermoviscoelasticity}. Afterwards, in the spirit of the isothermal result \cite{FiredrichKruzik18Onthepassage}, we pass to a linearized limit in terms of rescaled displacement fields and different regimes of rescaled temperatures.

We start by introducing the large strain model. Neglecting inertial effects, a nonlinear viscoelastic material in \CCC Kelvin-Voigt \BBB rheology obeys the following system of equations
\begin{equation}\label{viscoel}
 - \diver\big(
      \partial_F W(\nabla y, \theta)
      + \partial_{\dot F} R(\nabla y, \nabla \dot y, \theta)
    \big) 
  = f \qquad \text{in $[0,T]\times \Omega$.}
\end{equation}
Here, $[0, T]$ is a process time interval with $T > 0$, $\Omega \subset \R^d$ is a bounded domain representing the reference configuration, $y \colon [0, T] \times \Omega \to \R^d$ is a deformation mapping, $\nabla y$ is the deformation gradient, $\theta$ denotes the temperature,  $W \colon \R^{d\times d} \times [0, \infty) \to \R \cup \setof{+\infty}$ is a stored energy density, which represents a potential of the first Piola-Kirchhoff stress tensor $\partial_F W$, and $F \in \R^{d\times d}$ is the placeholder of $\nabla y$. Finally, $R \NNN \colon \R^{d \times d} \times \R^{d \times d} \times [0, \infty) \to \R\BBB$ denotes a (pseudo)potential of dissipative forces, where $\dot F$ is the time derivative of $F$, and $f \colon [0, T] \times \Omega \to \R^d$ is a volume density of external forces acting on $\Omega$.

The density $W$ respects frame indifference under rotations  and  positivity of the determinant of the deformation gradient, i.e., local non-self-penetration is realized. (In contrast to \cite{kroemrou}, we do not consider conditions implying global
non-self-penetration.) At the same time, we focus on physically correct viscous stresses, i.e., as observed by {\sc Antman} \cite{Antmann04Nonlinear}, $R$ must comply with \NNN a \BBB time-continuous frame indifference principle meaning that for all $F$ it holds that 
\begin{equation*}
  R(F,\dot F,\theta) = \hat R(C,\dot C,\theta),
\end{equation*}
for some nonnegative function $\hat R$, where $C \defas F^\top F$ and $\dot C \defas \dot F^T F + F^T \dot F$.

In contrast to the rapidly developed static theory at large strains, already in the isothermal case  existence of solutions to \eqref{viscoel}  remains a challenging problem and results for models respecting the physically relevant frame indifference for both $W$ and $R$ are scarce. We refer, e.g., to \cite{Lewick}  for  local in-time existence or to \cite{demoulini} for the existence of  measure-valued solutions. To date, weak solutions in  finite strain isothermal viscoelasticity  \cite{FiredrichKruzik18Onthepassage, kroemrou, MielkeRoubicek20Thermoviscoelasticity} can only be guaranteed by using the concept of  second-grade nonsimple materials  where   the stored energy density (and consequently the first Piola-Kirchhoff stress tensor) additionally depends on the second gradient of the deformation. This idea was first introduced by {\sc Toupin} \cite{Toupin62Elastic,Toupin64Theory} and proved to be useful in mathematical continuum mechanics, see e.g.~\cite{BallCurrieOlver81Null, Batra76Thermodynamics, MielkeRoubicek16Rateindependent, Podio02Contact}. In this spirit, we consider  a version of \eqref{viscoel} for nonsimple materials where the \NNN stored \BBB energy density depends also on the second gradient of $y$, and \eqref{viscoel} is replaced by 
\begin{equation}\label{viscoel_nonsimple}
 - \diver\Big(
      \partial_F W(\nabla y, \theta)
      - {\rm div}(\partial_G H(\nabla^2 y))
      + \partial_{\dot{F}}R(\nabla y, \nabla \dot y, \theta)
    \Big)
  = f \qquad \text{in $[0,T]\times \Omega$,}
\end{equation}
which corresponds to an additional convex term $ \int_\Omega H(\nabla^2 y)  \di x$ in the stored energy. Let us mention  that a main justification of this model lies in the observation that,  in the small strain limit and under suitable scaling,  the problem leads to the standard system of linear viscoelasticity without second gradient  \cite{FiredrichKruzik18Onthepassage}.

In the present contribution, we focus on a nonlinear coupling of the system \eqref{viscoel_nonsimple} with a heat transfer equation of the form
\begin{equation}\label{heat}
  c_V(\nabla y,\theta) \, \dot\theta =
    \diver(\mathcal{K}(\nabla y, \theta) \nabla\theta)
    + \partial_{\dot F} R(\nabla y, \nabla \dot y, \theta) : \nabla \dot y
    + \theta \partial_{F \theta} W^{\rm cpl}(\nabla y, \theta) : \nabla \dot y \qquad \NNN \text{in $[0,T]\times \Omega$}, \BBB
\end{equation}
where  $W^{\rm cpl}$ denotes a  thermo-mechanical coupling potential, $c_V(F,\theta) = -\theta \partial^2_\theta W^{\rm cpl}(F, \theta)$ is the  heat capacity, $\mathcal{K}$ denotes the matrix of the heat-conductivity coefficients, and the last term plays the role of an {adiabatic heat} source. This corresponds to heat transfer modeled by the Fourier law in the deformed configuration which is however pulled back to the reference configurations, whence $\mathcal{K}$ depends on the deformation gradient.   
 Here,  following \cite{MielkeRoubicek20Thermoviscoelasticity}, we assume a rather weak thermal coupling by using the splitting of the free energy $W$ via the explicit ansatz 
\begin{align}\label{decomposition}
W(F,\theta) = W^{\rm el}(F) + W^{\rm cpl}(F,\theta)
\end{align}
implying $\partial_\theta W = \partial_\theta W^{\rm cpl}$. The coupled system \eqref{viscoel_nonsimple}--\eqref{heat} is equipped with suitable initial and boundary conditions, see \eqref{main_bc}--\eqref{initial_cond} below.

Thermoviscoelasticity is a notoriously difficult problem already at small strains, e.g.,  there is no obvious variational structure of the thermal part due to  the low regularity of data. New developments in the the $L^1$-theory for the nonlinear heat equation \cite{Boccardoetal, BoccardoGallouet89Nonlinear} paved the way to advancements in small strain thermoviscoelasticity (for example, see \cite{Blanchard, Bonetti, Roubicek98Nonlinear}) which eventually culminated in the analysis of a physically sound large-strain model by {\sc Mielke \& Roub\'{\i}\v{c}ek} \cite{MielkeRoubicek20Thermoviscoelasticity}. We refer to \cite[Introduction; items $(\alpha)$--$(\lp)$]{MielkeRoubicek20Thermoviscoelasticity} for the main properties and challenges for this model which coincides with ours up to minor points, see  Remark \ref{rem: relation}.  Their existence result is based on a time-incremental approach for a regularized system which does not comply with the above mentioned frame indifferent principles, e.g., in \eqref{viscoel_nonsimple} a term $\NNN \lambda \BBB \nabla \dot{y}$ is added for $\NNN \lambda \BBB > 0$. Then, they first pass to the time-continuous limit in the regularized problem and eventually recover the original system in the limit of the vanishing regularization parameter $\NNN \lambda\BBB$. 

The first result of our work (Theorem \ref{thm:van_tau}) revisits their study by proposing a slightly different semidiscretization in time which  directly approximates the PDE system in the  limit for vanishing time steps and comes along without any regularization. Although establishing the same existence result on weak solutions, our approach sheds new light on the issue  as we propose a time-discrete approximation scheme complying with frame indifference. This combined with a spatial discretization, see e.g.\ \cite[Section 9.3]{KruzikRoubicek19Rate-independent},  could be the basis for a numerical implementation. As in \cite{MielkeRoubicek20Thermoviscoelasticity}, our scheme is staggered, i.e., first the deformation is updated at fixed temperature from the previous time \NNN step \BBB and then the temperature is updated. Our scheme differs in the usage of explicit or implicit steps, i.e., whether in certain terms the `old' or the `new' temperature is used, see Remark \ref{rem:difference_to_mielke_in_steps}. By means of delicate estimates on the coupling potential, we are hereby able to establish the necessary a priori bounds without any regularization. At this point, we \CCC derive \BBB a priori estimates for different scalings of the elastic strains and the temperature which is at the basis of our subsequent analysis on small-strain limits.

In the second part of our work, we are interested in the case of small strains and temperatures, i.e., when  $\nabla u \defas \nabla y - \Id $ is of order $\lp$ for some small $\lp >0$ and $\theta$ is of order $\lp^\alpha$ for any exponent $\alpha >0$. Here, $u \defas y-\id$ is the displacement corresponding to $y$ with  $\id$ and $\Id$ standing for the identity map and identity matrix, respectively.  Such  properties are certainly reasonable if initial values and boundary values for the deformation and the temperature are close to the identity or zero, respectively. 
Therefore, it is convenient to introduce the rescaled displacement $u_\lp = \lp^{-1}(y - \id)$ and rescaled temperature $\mu_\lp = \lp^{-\alpha} \theta$, and to replace $f$ by $\lp f$. We write \eqref{viscoel_nonsimple}--\eqref{heat} in terms of the rescaled quantities and  multiply \eqref{viscoel_nonsimple} with $\lp^{-1}$ and \eqref{heat} with $\lp^{-\alpha}$. Then, formally, we can pass to the limit and obtain the system
\begin{equation}\label{viscoel-small-intro}
\begin{aligned}
	-\diver \big( \C_W e(u) + \C_D e(\dot u) + \mathbb{B}^{(\alpha)} \mu \big) &= f, \\
	\bar c_V \dot\mu - \diver(\mathbb{K}_0 \nabla \mu)
  &= \C_D^{(\alpha)} e(\dot u): e(\dot u),
\end{aligned}
\end{equation}
where $\C_W \defas \partial^2_F \elpot(\Id)$ is the tensor of elastic constants ($\elpot$ is defined in \eqref{decomposition}), $\C_D\defas\partial_{\dot F^2}R(\Id, \CCC \dot F \BBB, 0)$ is the tensor of viscosity coefficients, $\mathbb{B}^{(\alpha)}$ represents a thermal expansion matrix,  $\bar c_V$ is the heat capacity   at zero temperature and the stress free material state, and $\mathbb{K}_0 \defas \mathcal{K}(\Id, 0)$. Finally, $e(u)\defas(\nabla u+(\nabla u)^\top)/2$ denotes the linear strain tensor and $e(\dot u)$ the strain rate. By different scaling properties of the two equations, it turns out that the limit is $\alpha$-dependent and, as we point out later, only meaningful in the regime $\alpha \in [1,2]$. The matrix $\mathbb{B}^{(\alpha)}$ is only active for $\alpha = 1$ and in this case it is related to the coupling potential, namely $\mathbb{B}^{(\alpha)} =  \partial_{F\theta} \cplpot(\Id,0)$. On the other hand, $\C_D^{(\alpha)} $ is nonzero only for $\alpha=2$ and then it coincides with $\C_D$.  \NNN Interestingly, \BBB although the nonlinear thermoviscoelasticity system is written for a nonsimple material, in the limit we obtain  linear equations without spatial gradients of $e(u)$.

\NNN Formal derivations of such  \BBB PDE systems is not new and can be found, e.g., in \cite[Section 59]{GurtinFriedAnand10Themechanics}. The second  main contribution of our work (Theorem \ref{thm:linearization_right_diag}) is to make this limit passage rigorous, i.e., we show that solutions to the nonlinear system \eqref{viscoel_nonsimple}--\eqref{heat} converge in a suitable sense to  the unique \BBB solution of the linear system \eqref{viscoel-small-intro} as $\lp\to 0$. Besides this convergence result, we also get analogous convergences for time-discretized problems,   and we confirm that  convergences for vanishing time step and $\lp\to 0$ commute, see Theorem \ref{thm:linearization_left_bottom}.

To our best knowledge, it is the first linearization result of a mechanical model coupled with heat transfer in the material. We perform linearization near the natural (i.e., stress free) state and zero temperature. Without further details, let us however mention that by a shifting argument our techniques would allow to  linearize about a fixed, positive temperature $\theta_c$, whenever the initial and boundary data lie above $\theta_c$ and the coupling potential $W^{\rm cpl}(F,\theta)$ vanishes for $\theta \le \theta_c$.

We now give an  outline of the paper and present some fundamental ingredients of the proof. After some basic notation, we introduce the nonlinear setting in Subsection \ref{sec:setting}. In Subsection \ref{sec:nonlinear_scheme}, we formulate our semi-discrete approximation result in the nonlinear setting and briefly highlight the differences to the scheme in  \cite{MielkeRoubicek20Thermoviscoelasticity}, see Remark \ref{rem:difference_to_mielke_in_steps}. In  Subsection \ref{sec: linear case}, we introduce the linearized setting and present our results on convergence of solutions in the nonlinear-to-linear passage. 

In Subsections \ref{sec:single_step_well_defined}--\ref{sec: welldef}, we address the well-definedness of the staggered time-incremental scheme. The core of our approach is an inductive bound on the total energy, see Lemma \ref{lem:initial_toten_bound}: this is achieved by suitably testing the momentum balance and the heat-transfer equation,  adding the two equations, and exploiting cancellation of the dissipation. In contrast to \cite{MielkeRoubicek20Thermoviscoelasticity}, see particularly \cite[Remark 6.1]{MielkeRoubicek20Thermoviscoelasticity}, this cancellation is \CCC already \BBB possible \CCC in the time-discrete setting \BBB as we use a simpler, explicit, thermo-mechanical coupling term in the scheme allowing us to proceed without the necessity of regularizing terms. This, however, comes at the expense of the fact that the argument to guarantee nonnegativity of the temperature in the thermal step is more sophisticated. For this, we need a delicate estimate for the coupling potential, see Proposition \ref{prop:existence_thermal_step}. 

As a preparation for the passage to the linearized system, we need an adaption of the bound on the total energy, see Subsection \ref{sec: adaptions}. In fact, due to the different scaling $\lp$ and $\lp^\alpha$ of the mechanical and the heat-transfer equation, the above mentioned cancellation cannot be used in general for small $\lp$. \NNN Thus, \BBB   novel techniques are required to tame the contribution  of the dissipation including higher integrability of the temperature variable, see Lemma \ref{lem:initial_toten_bound_lin} for details. Section \ref{sec:staggered_scheme} is closed with  a priori bounds derived from the energy bound, see Subsection \ref{sec: a priori}. As in \cite{MielkeRoubicek20Thermoviscoelasticity}, the main ingredients here are Gagliardo-Nirenberg interpolation inequalities and special test functions developed by {\sc Boccardo and Gallou\"et} \cite{BoccardoGallouet89Nonlinear} for parabolic equations with measure-valued right-hand side. For convenience of the reader, almost complete proofs are provided since in addition to \cite{MielkeRoubicek20Thermoviscoelasticity} we need  scaling invariant estimates in terms of the small parameter $\lp$.
 
In Section \ref{sec:tau_to_zero_delta_fixed} we then address the passage to vanishing time steps in the nonlinear model. At this point,  having settled  the  a priori estimates,  we can essentially follow \cite{MielkeRoubicek20Thermoviscoelasticity}. Since we work without regularization terms, however, we need to combine and adapt the  techniques from Sections 5--6 of \cite{MielkeRoubicek20Thermoviscoelasticity}, and therefore we elaborate the proofs to some extent. Eventually, Section \ref{sec: linearization} is devoted to the linearization. In Subsection~\ref{sec: 5.1} we first deal with the passage to the time-continuous problem. The strategy in the proof is similar to the one in the nonlinear setting in Section \ref{sec:tau_to_zero_delta_fixed}, with the additional challenge that in each term we need to ensure that higher order terms in Taylor expansions are asymptotically negligible. In particular, we show that contributions of the second gradient vanish in the limit. As in the nonlinear setting, strong convergence of the strains and the strain rates is necessary to pass to the limit, see Lemma \ref{lem:strong_strain_rates_conv} and Lemma \ref{lemma: strong ratistrain}. Due to rescaling of the equations, however, this is more demanding  in the passage to the linerized setting  as higher integrability of the temperature is needed to control the coupling term, cf.\ Remark \ref{rem: after temp}. Eventually, Subsection \ref{sec: 5.2} is devoted to time-discrete problems which particularly involves a $\Gamma$-convergence result   for the mechanical part, see Proposition \ref{prop:gamma_conv}.

\section{The model and main results}\label{sec:model_and_main_results}

\subsection{The setting and modeling assumptions}\label{sec:setting}

In what follows, we \NNN use \BBB standard notation for Lebesgue and Sobolev spaces. The lower index $_+$ means nonnegative elements, i.e., $L^2_+(\Omega)$ denotes the convex cone of nonnegative functions belonging to $L^2(\Omega)$ and a similar notation is used for $H^1_+(\Omega)$.  We also set $\R_+\defas [0,+\infty)$. Let \NNN  $a \wedge b \defas \min\setof{a, b}$ for $a, \, b \in \R$.  Denoting by $d \NNN \ge 2\BBB$ the dimension, \NNN we let $\Id \in \R^{d \times d}$  be the identity matrix, and \BBB $\id(x) \defas x$ stands for the identity map on $\R^d$. We define \NNN the subsets \BBB $SO(d) \defas \setof{A \in \R^{d \times d} \colon A^T A = \Id, \, \NNN \det A = 1 \BBB }$, $GL^+(d) \defas \setof{F \in \R^{d \times d} \colon \det(F) > 0}$, \NNN and \BBB $\R^{d \times d}_\sym \defas \setof{A \in \R^{d \times d} \colon A^T = A}$.
Furthermore, $F^{-T} \defas (F^{-1})^T=(F^T)^{-1}$, and  given a tensor (of arbitrary dimension), $\abs{F}$ will denote its Frobenius norm.
We denote the scalar product between vectors, matrices, or 3rd-order tensors by $\cdot$, $:$, and $\cdddot$, respectively.  As usual, in the proofs generic constants $C$ may vary from line to line. If not stated otherwise, constants depend only on $d$, $p >d$, $\Omega$, $\alpha >0$, and the potentials introduced in the sequel. \NNN We frequently use a scaled version of Young's inequality with constant $\lambda \in (0,1)$ by which we mean $ab \le \lambda a^p + Cb^q/\lambda$ for $a,b \ge 0$, exponents $p,q \ge 1$ with $1/p + 1/q = 1$, and $C>0$ large enough.     \BBB

Consider an open set $\Omega \subset \R^d$ with Lipschitz boundary $\Gamma \defas \partial \Omega$.
Let $\Gamma_D, \, \Gamma_N$ be disjoint Borel subsets of $\Gamma$ such that $\haus^{d-1}(\Gamma_D) > 0$, $\haus^{d-1}(\Gamma_N) > 0$, and $\Gamma = \Gamma_D \cup \Gamma_N$, representing Dirichlet and Neumann parts of the boundary, respectively.
For $p > d$, we introduce the set of \emph{admissible deformations} by
\begin{equation}\label{eq: randwerde}
  \Wid \defas \setof*{
    y \in W^{2, p}(\Omega; \R^d) \colon
    y = \id \text{ on } \Gamma_D, \,
    \det(\nabla y) > 0 \text{ in } \Omega
  }
\end{equation}
and we say that the \emph{absolute temperature} $\theta$ is admissible if $\theta \in L^1_+(\Omega)$. We also introduce the space
\begin{equation}\label{eq: 000}
 \NNN   \Wzero \BBB \defas \setof{y \in W^{2, p}(\Omega; \R^d) \colon y = 0 \text{ on } \Gamma_D}.
\end{equation}
\NNN Next, we \BBB discuss our variational setting.
In this regard, let $\ac, \, \aC$ with $0 < \ac < \aC \CCC < \infty \BBB$ be some  fixed \BBB  constants.

\noindent \textbf{Mechanical energy and coupling energy:}
The \emph{elastic energy} $\elen \colon \Wid \to \R_+$ is given by
\begin{equation}\label{purely_elastic}
  \elen(y) \defas \int_\Omega \elpot(\nabla y) \di x,
\end{equation}
where $\elpot \colon GL^+(d) \to \R_+$ is a frame indifferent elastic energy potential with the usual assumptions in nonlinear elasticity.
More precisely, we require that
\begin{enumerate}[label=(W.\arabic*)]
  \item \label{W_regularity} $\elpot$  is \NNN continuous and $C^3$ in a neighborhood of $SO(d)$; \BBB
  \item \label{W_frame_invariace} Frame indifference: $\elpot(QF) = \elpot(F)$ for all $F \in GL^+(d)$ and $Q \in SO(d)$;
  \item \label{W_lower_bound} Lower bound: $W^{\rm el}(F) \ge \ac \big(|F|^2 + \det(F)^{-q}\big) - \aC$ for all $F \in GL^+(d)$, where $q \ge \frac{pd}{p-d}$.
\end{enumerate}
Adopting the concept of 2nd-grade nonsimple materials, see \cite{Toupin62Elastic, Toupin64Theory}, we also consider a \emph{strain gradient \BBB energy term} $\hyen \colon \Wid \to \R_+$, defined as
\begin{equation}\label{hyperelastic}
  \hyen(y) \defas \int_\Omega \hypot(\nabla^2 y) \di x,
\end{equation}
where its potential $\hypot \colon \R^{d \times d \times d} \to \R_+$ satisfies
\begin{enumerate}[label=(H.\arabic*)]
  \item \label{H_regularity} $\hypot$ is convex and $C^1$;
  \item \label{H_frame_indifference} Frame indifference: $\hypot(QG) = \hypot(G)$ for all $G \in \R^{d \times d \times d}$ and $Q \in SO(d)$;
  \item \label{H_bounds} $\ac \abs{G}^p \leq H(G) \leq \aC (1+ \abs{G}^p)$ and $\abs{\pl_G H(G)} \leq \aC \abs{G}^{p-1}$ for all $G \in \R^{d \times d \times d}$.
\end{enumerate}
The \emph{mechanical energy} $\mechen \colon \Wid \to \R_+$ is then defined as the sum
\begin{equation}\label{mechanical}
  \mechen(y) \defas \elen(y) + \hyen(y).
\end{equation}
Besides the mechanical energy, we introduce a \emph{coupling energy} $\cplen \colon \Wid \times \NNN L^1_+(\Omega) \BBB  \to \R$ given by
\begin{equation*}
  \cplen(y, \theta) \defas \int_\Omega \cplpot(\nabla y, \theta) \di x,
\end{equation*}
where $\cplpot \colon GL^+(d) \times \R_+ \to \R$ describes mutual interactions of mechanical and thermal effects (see e.g.~\cite{GurtinFriedAnand10Themechanics}), and satisfies
\begin{enumerate}[label=(C.\arabic*)]
  \item \label{C_regularity} $\cplpot$ is continuous and $C^2$ in $GL^+(d) \times (0, \infty)$;
  \item \label{C_frame_indifference} $\cplpot(QF, \theta) = \cplpot(F, \theta)$ for all $F \in GL^+(d)$, $\theta \geq 0$, and $Q \in SO(d)$;
  \item \label{C_zero_temperature} $\cplpot(F, 0) = 0$ for all $F \in GL^+(d)$;
  \item \label{C_lipschitz} $|\cplpot(F,\theta) - \cplpot(\tilde{F}, \theta)| \le \aC(1 + |F| + |\tilde{F}|)|F - \tilde{F}|$ for all $F, \, \tilde F \in GL^+(d)$\CCC, \BBB and $\theta \geq 0$;
  \item \label{C_bounds} For all $F \in \NNN GL^+(d)\BBB$ and $\theta > 0$ it holds that
  \begin{align*}
    \abs{\partial_F^2 W^{\rm cpl}(F,\theta)} &\le \aC, &
    \abs{\pl_{F \theta} \cplpot(F, \theta)} & \leq \frac{\aC(1+|F|)}{\max\lbrace \theta,1\rbrace}, &
    \ac & \leq -\theta \pl_\theta^2 \cplpot(F, \theta) \leq \aC.
  \end{align*}
\end{enumerate}
\NNN Notice \BBB that, by \ref{C_zero_temperature} and the second bound in \ref{C_bounds}, $\pl_F \cplpot$ can be continuously extended to zero temperatures with $\pl_F W^{\rm cpl}(F, 0) = 0$. For $F \in GL^+(d)$ and $\theta \geq 0$, we define the \emph{total free energy potential} 
\begin{align}\label{eq: free energy}
\felpot(F, \theta) \defas \elpot(F) + \cplpot(F, \theta).
\end{align}

\noindent\textbf{Dissipation potential:}
The \emph{dissipation functional} $\diss \colon \Wid \times \Wzero \times  L^1_+(\Omega) \BBB \to \R_+$ is defined as
\begin{equation}\label{dissipation}
  \diss(\tilde y, y - \tilde y, \theta)
  \defas \int_\Omega \disspot(\nabla \tilde y, \nabla y - \nabla \tilde y, \theta) \di x,
\end{equation}
where $\disspot \colon \R^{d \times d} \times \R^{d \times d} \times \R_+ \to \R_+$ is the \emph{potential of dissipative forces} satisfying
\begin{enumerate}[label=(D.\arabic*)]
  \item \label{D_quadratic} $\disspot(F, \dot F, \theta) \defas \frac{1}{2} D(C, \theta)[\dot C, \dot C] \defas \frac{1}{2} \dot C : D(C, \theta) \dot C$, where $C \defas F^T F$, $\dot C \defas \dot F^T F + F^T \dot F$, and $D \in C(\R^{d \times d}_\sym \times \R_+; \R^{d \times d \times d \times d})$ with $D_{ijkl} = D_{jikl}= D_{klij}$ for $1 \le i,j,k,l \le d$;
  \item \label{D_bounds} $\ac \abs{\dot C}^2 \leq \dot C : D(C, \theta) \dot C \leq \aC \abs{\dot C}^2$ for all $C, \, \dot C \in \R^{d \times d}_\sym$, and $\theta \geq 0$.
\end{enumerate}
Notice that the fact that $\disspot$ can be written as a function depending on the right Cauchy-Green tensor $C = F^T F$ and its time derivative $\dot C$ is equivalent to \emph{dynamic frame indifference} (see also \cite{Antmann98Physically}).
Condition \ref{D_quadratic} also implies that the viscous stress $\pl_{\dot F} \disspot(F, \dot F, \theta)$ is linear in the time derivative $\dot C$ as indeed a simple calculation shows
\begin{equation}\label{chain_rule_Fderiv}
  \pl_{\dot F} \disspot(F, \dot F, \theta) = 2 F (D(C, \theta) \dot C).
\end{equation}
The choice of a linear material viscosity is crucial in our approach and is a relevant modeling assumption for
non-activated dissipative processes with rather moderate rates.
We emphasize, however, that the geometrical nonlinearity of finite elasticity is still present due to $\dot C$ in \eqref{chain_rule_Fderiv},   and that $\pl_{\dot F} \disspot$ necessarily also depends on $F$, even for constant functions $D$. We also define the associated \emph{dissipation rate} $\drate \colon \R^{d \times d} \times \R^{d \times d} \times \R_+ \to \R_+$ as
\begin{equation}\label{diss_rate}
  \drate(F, \dot F, \theta)
  \defas \pl_{\dot F} \disspot(F, \dot F, \theta) : \dot F
  = D(C, \theta) \dot C : \dot C = 2 R(F, \dot F,\theta),
\end{equation}
where the \CCC second \BBB identity follows from \eqref{chain_rule_Fderiv}, $\dot C = \dot F^T F + F^T \dot F$, and the symmetries in \ref{D_quadratic}.

\noindent\textbf{Heat conductivity and internal energy:}
The map $\hc \colon \Omega \times \R_+ \to \R^{d \times d}_\sym$ will denote the \emph{heat conductivity tensor} of the material in the deformed configuration.
We require that $\hc$ is continuous, symmetric, uniformly positive definite, and bounded.
More precisely, for all $x \in \Omega$ and $\theta \geq 0$ it holds that
\begin{equation}\label{spectrum_bound_K}
  \ac \leq \hc(x, \theta) \leq \aC,
\end{equation}
where the inequalities are meant in the eigenvalue sense.
We define the pull-back $\hcm \colon \Omega \times  GL^+(d) \BBB  \times \R_+ \to \R^{d \times d}_\sym$ of $\hc$ into the reference configuration by (see \cite[(2.24)]{MielkeRoubicek20Thermoviscoelasticity})
\begin{equation}\label{hcm}
  \hcm(x, F, \theta) \defas \det(F) F^{-1} \hc(x, \theta) F^{-T}.
\end{equation}
From this point on, we will usually omit stating the $x$-dependence of $\hcm$ and $\hc$ explicitly, shortly writing $\hc(\theta)$ for $\hc(x, \theta)$ and $\hcm(F, \theta)$ for $\hcm(x, F, \theta)$.

\noindent\textbf{Internal and total energy:}
The \textit{internal energy} $\inten \colon GL^+(d) \times (0, \infty) \to \R$ is defined as
\begin{equation}\label{Wint}
  \inten(F, \theta) \defas \cplpot(F, \theta) - \theta \pl_\theta \cplpot(F, \theta).
\end{equation}
Using \ref{C_zero_temperature} and the third bound in \ref{C_bounds}, we can easily see that $\inten$ can be continuously extended to zero temperatures by setting $\inten(F, 0) = 0$ for all $F \in GL^+(d)$.
Also by the third bound in \ref{C_bounds}, the internal energy is controlled by the temperature in the following sense: 
\begin{align}\label{sec_deriv}
  \partial_{\theta} \inten (F, \theta)
  = -\theta \pl_\theta^2 \cplpot(F, \theta) \in [\ac, \aC]
  \qquad \text{for all $F \in GL^+(d)$ and $\theta \NNN > \BBB 0$}
\end{align}
which along with \ref{C_zero_temperature} yields
\begin{equation}\label{inten_lipschitz_bounds}
  \ac \theta \leq \inten(F, \theta) \leq \aC \theta.
\end{equation}
\NNN Eventually, \BBB the \emph{total energy functional} $\toten \colon \Wid \times L^1_+(\Omega) \to \R_+$ is then given by
\begin{equation}\label{toten}
  \toten(y, \theta) \defas \mechen(y) +  \NNN \mathcal{W}^{\rm in}(y,\theta) \quad \text{ with }  \mathcal{W}^{\rm in}(y,\theta) \defas \BBB \int_\Omega \inten(\nabla y, \theta) \di x.
\end{equation}

\begin{remark}[Comparison to \cite{MielkeRoubicek20Thermoviscoelasticity}\BBB]\label{rem: relation}
  We close this part on modeling assumptions by highlighting the differences to the assumptions in \cite{MielkeRoubicek20Thermoviscoelasticity}:
  Our condition in \ref{W_lower_bound} is slightly more general than the corresponding one in \cite[(2.30a)]{MielkeRoubicek20Thermoviscoelasticity}, where the term $\abs{F}^2$ is replaced by $\abs{F}^s$ for $s > 2$.
  We do not assume that $W^{\rm cpl}$ is bounded from below.
  Condition \ref{C_zero_temperature} as well as bounds similar to \ref{C_lipschitz}--\ref{C_bounds} are also required in \cite{MielkeRoubicek20Thermoviscoelasticity}, see \cite[(2.15), (2.30)]{MielkeRoubicek20Thermoviscoelasticity}.
  There, the bound on $\pl_{F \theta} \cplpot$ is slightly more general for $\theta$ near zero, and only an upper bound on the eigenvalues of $\pl_F^2 \cplpot(F, \theta)$ is required, see \cite[(2.30c)]{MielkeRoubicek20Thermoviscoelasticity}. This similarity of the assumptions will in particular allow us to employ several intermediate steps proven in \cite{MielkeRoubicek20Thermoviscoelasticity}.   For models complying with the above assumptions we refer to \cite[Examples 2.4, 2.5]{MielkeRoubicek20Thermoviscoelasticity}.
\end{remark}

\noindent\textbf{Equations of nonlinear thermoviscoelasticity:} Fixing a finite time horizon $T > 0$, let us from now on shortly write $I \defas [0, T]$.
We fix a constant $\lp \in (0, 1]$ which represents the \emph{magnitude of the elastic strain}.
In the first part of the paper, we are mainly interested in the large strain setting, where $\lp = 1$.
However, later we perform the passage to the small strain limit $\lp \to 0$.
To allow for a consistent notation, we include the parameter $\lp$ throughout  the entire paper.
\NNN Let \BBB $\lp f$ with $f \in W^{1, 1}(I; L^2(\Omega; \R^d))$ be a time-dependent \emph{dead force}, $\lp g$ with $g \in W^{1, 1}(I; L^2(\Gamma_N; \R^d))$  be \BBB a \emph{boundary traction}, and let $\lp^{ \alpha \BBB} \bt$ with $\bt \in  W^{1, 1}(I; L^2_+(\Gamma))$  and $\alpha >0$ \BBB   be an external temperature.
 We study the coupled system \BBB
\begin{subequations}\label{main_evol_eq}
\begin{align}
  \lp f &=
    -\diver\big(
      \pl_F \felpot(\nabla y, \theta)
      + \pl_{\dot F} \disspot(\nabla y, \nabla \dot y, \theta)
      - \diver(\pl_G \hypot(\nabla^2 y))
    \big), \label{main_mechanical_eq} \\
  -\theta \pl_\theta^2 \cplpot(\nabla y, \theta) \, \dot{\theta} &=
    \diver(\hcm(\nabla y, \theta) \nabla \theta)
    + \drate(\nabla y, \nabla \dot y, \theta)
    + \theta \pl_{F \theta} \cplpot(\nabla y, \theta) : \nabla \dot y, \label{main_thermal_eq}
\end{align}
\end{subequations}
which, as in \cite{MielkeRoubicek20Thermoviscoelasticity}, is complemented with the boundary conditions
\begin{subequations}\label{main_bc}
\begin{align}
  \big(
    \pl_F \felpot(\nabla y, \theta)
    + \pl_{\dot F} \disspot(\nabla y, \nabla \dot y, \theta)
  \big) \nu
  - \diver_S \big( \pl_G \hypot(\nabla^2 y) \nu \big)
    &= \lp g & &\text{ on } \NNN I \times \BBB \Gamma_N, \\
  y &= \id & &\text{ on } I \times \Gamma_D, \\
  \pl_G \hypot(\nabla^2 y) : (\nu \otimes \nu)
    &= 0 & &\text{ on }I \times  \Gamma, \\
  \hcm(\nabla y, \theta) \nabla \theta \cdot  \nu \BBB + \kappa \theta
    &= \kappa \lp^\alpha \bt & &\text{ on } I \times \Gamma.
\end{align}
\end{subequations}
Here, $\nu$ denotes the outward pointing unit normal on $\Gamma$ and $\kappa \ge 0$ is a \emph{phenomenological heat-transfer coefficient} on $\Gamma$. Moreover, $\diver_S$   represents the \emph{surface divergence}, defined by $\diver_S(\cdot) = \trace(\nabla_S(\cdot))$, where $\trace$ denotes the trace and $\nabla_S \defas (\Id - \nu \otimes \nu) \nabla$ denotes the surface gradient. We refer to \cite[(2.28)--(2.29)]{MielkeRoubicek20Thermoviscoelasticity} for an explanation and derivation of the boundary conditions.  Note that by \eqref{diss_rate} the system \eqref{main_evol_eq} indeed coincides with \eqref{viscoel_nonsimple}--\eqref{heat}. \BBB

The mechanical evolution \eqref{main_mechanical_eq} is the quasistatic version of the Kelvin-Voigt rheological model (neglecting inertia), corresponding to the sum of the conservative and the dissipative forces.
The equation \eqref{main_thermal_eq} follows from the entropy equation $\theta \dot s = \xi - \diver  q \BBB$, where the \emph{entropy} $s$ is expressed in terms of the free energy by $s = - \partial_\theta W = - \partial_\theta \cplpot$.
Furthermore, the dissipation rate $\xi$ is defined in \eqref{diss_rate} and the \emph{heat flux} $q$ is modeled by the \emph{Fourier law} in the deformed configuration, pulled back to the reference configuration, i.e., $q = - \mathcal{K}(F,\theta) \nabla \theta$.
The term $-\theta \pl_\theta^2 \cplpot(\nabla y, \theta)$ corresponds to the \emph{heat capacity} \CCC at constant volume \BBB and the last term in \eqref{main_thermal_eq} is an \emph{adiabatic heat} source.
We again refer to \cite{MielkeRoubicek20Thermoviscoelasticity} or to \cite[Section 8.1]{KruzikRoubicek19Rate-independent} for details.
Notice that the the purely mechanical stored energy \NNN $W^{\rm el}$, see \eqref{purely_elastic}, \BBB does not influence the heat production and transfer in \eqref{main_thermal_eq}.

We consider a corresponding initial-value problem, by imposing the initial conditions
\begin{equation}\label{initial_cond}
  y(0, \cdot) =  y_{0,\lp} \BBB \defas \id + \lp u_0 \qquad \text{and} \qquad \theta(0, \cdot) = \theta_{0,\lp} \defas \lp^\alpha \mu_0
\end{equation}
for some  $\mu_0 \in L^2_+(\Omega)$ and some $u_0 \CCC \in W^{2, p}_{\Gamma_D}(\Omega; \R^d)\BBB$.
We now define weak solutions associated to the initial-boundary-value problem \eqref{main_evol_eq}--\eqref{initial_cond}.

\begin{definition}[Weak solution of the nonlinear system]\label{def:weak_formulation}
A couple $(y, \theta) \colon I \times \Omega \to \R^d \times \R$ is called a \emph{weak solution} to the initial-boundary-value problem \eqref{main_evol_eq}--\eqref{initial_cond} if $y \in L^\infty(I; \Wid) \cap H^1(I; H^1(\Omega; \R^d))$ with $y(0, \cdot)=  y_{0,\lp}\BBB$, $\theta \in L^1(I; W^{1,1}(\Omega))$ with $\theta \ge 0$ a.e., and if it satisfies the identities
\begin{equation}\label{weak_limit_mechanical_equation}
\begin{aligned}
  &\intQ \pl_G
    \hypot(\nabla^2 y) \cdddot \nabla^2 z
    + \Big(
      \pl_F \felpot(\nabla y, \theta)
      + \pl_{\dot F} \disspot(\nabla y, \nabla \dot y, \theta)
    \Big) : \nabla z \di x \di t \\
  &\quad= \lp \intQ f \cdot z \di x \di t
    + \lp \int_0^T \int_{\Gamma_N} g \cdot z \di \haus^{d-1} \di t
\end{aligned}
\end{equation}
for any test function $z \in C^\infty(I \times \overline{\Omega}; \R^d)$ with $z = 0$ on $I \times \Gamma_D$, as well as
\begin{equation}\label{weak_limit_heat_equation}
\begin{aligned}
  &\intQ \hcm(\nabla y, \theta) \nabla \theta \cdot \nabla \vphi
    -\big(
      \drate(\nabla y, \nabla \dot y, \theta)
      + \pl_F \cplpot(\nabla y, \theta) : \nabla \dot y
    \big) \vphi
    - \inten(\nabla y, \theta) \dot \vphi \di x \di t \\
  &\quad + \kappa \int_0^T \int_{\NNN \Gamma} \theta \vphi \di \haus^{d-1} \di t
    = \kappa \lp^\alpha \int_0^T \int_{\NNN \Gamma} \bt \vphi \di \haus^{d-1} \di t
    + \int_\Omega \inten(\nabla  y_{0,\lp}, \BBB \theta_{0,\lp}) \, \vphi(0) \di x
\end{aligned}
\end{equation}
for any test function $\vphi \in C^\infty(I \times \overline \Omega)$ with $\varphi(T) = 0$.
\end{definition}

One can indeed show that sufficiently smooth weak solutions lead to the classical formulation \eqref{main_evol_eq} along with the boundary conditions \eqref{main_bc}, see \cite{MielkeRoubicek20Thermoviscoelasticity}.
We refer to \cite[(2.28)--(2.29)]{MielkeRoubicek20Thermoviscoelasticity} for details on the derivation of \eqref{main_mechanical_eq}, particularly how to treat the boundary terms.
For the derivation of \eqref{main_thermal_eq}, one uses standard integration by parts and the fact that by the definition in \eqref{Wint} we have
\begin{equation*}
  \frac{\di}{\di t} (\inten(\nabla y, \theta))
  = \partial_F \cplpot(\nabla y, \theta) : \nabla \dot y
    - \theta \pl_{F \theta} \cplpot (\nabla y, \theta) : \nabla \dot y
    - \theta \pl_\theta^2 \cplpot (\nabla y, \theta) \dot \theta.
\end{equation*}
Moreover, using test functions with $\varphi(0) \neq 0$ we obtain $\inten(\nabla y(0), \theta(0)) =  \inten(\nabla y_{0,\lp}, \theta_{0,\lp})\BBB$, and by the strict monotonicity in \eqref{sec_deriv} along with $y(0) = y_{0,\lp}$ we conclude $\theta(0) = \theta_{0,\lp}$.
We emphasize  that one can only expect the regularity $\nabla \dot y \in L^2(I \times \Omega;  \R^{d\times d})\BBB$ and thus $\drate(\nabla y, \nabla \dot y, \theta) \in L^1(I \times \Omega)$ by \eqref{diss_rate}.
Therefore, \eqref{main_thermal_eq} can be understood as a heat equation with $L^1$-data.
For this, \eqref{weak_limit_heat_equation} is a standard weak formulation, see \NNN e.g.~\cite{Roubicek98Nonlinear}. \BBB

\subsection{Approximation of solutions in the nonlinear setting}\label{sec:nonlinear_scheme}
In this subsection, we  study the nonlinear system and therefore we fix $\lp =1$. (In the notation, $\lp$ is still included, as before.)
The existence of energy-conserving weak solutions to \eqref{main_evol_eq} in the sense of Definition \ref{def:weak_formulation} has been proven in \cite[Theorem~2.2]{MielkeRoubicek20Thermoviscoelasticity}.
In contrast to this work, we show here that the solutions can be obtained directly as limits of a staggered time-incremental scheme without using  any \BBB additional regularization.

 We  fix a discrete time step size $\tau \in (0, 1]$.
For the sake of notational clarity, we assume without a further mention that any $\tau$ we encounter evenly divides the time interval $[0,T]$.
Given any sequence \NNN $(a_l)_{l \ge 0}$, \BBB it will be useful to introduce the following notation for discrete differences
\begin{equation*}
  \ddif a_l \defas \frac{a_l - a_{l-1}}{\tau}, \quad l \in  \N. \BBB
\end{equation*}
Our time-discrete staggered scheme is initialized by setting
\begin{equation}\label{def_0_step}
 \ysts 0 \defas y_{0,\lp} \qquad \text{and} \qquad \tsts 0 \defas \theta_{0,\lp},
\end{equation}
where $y_{0,\lp}$ and $\theta_{0,\lp}$ are as in \eqref{initial_cond}.  
We then alternate between a \textit{mechanical step}, deforming the material while keeping the temperature fixed, and a \textit{thermal step}, adjusting the temperature distribution inside the material without changing the deformation, \CCC see Figure \ref{fig:minmoves}. \BBB
\begin{figure}[ht]
  \begin{subfigure}{.245\textwidth}
    \centering
    \includegraphics[height=3.1cm]{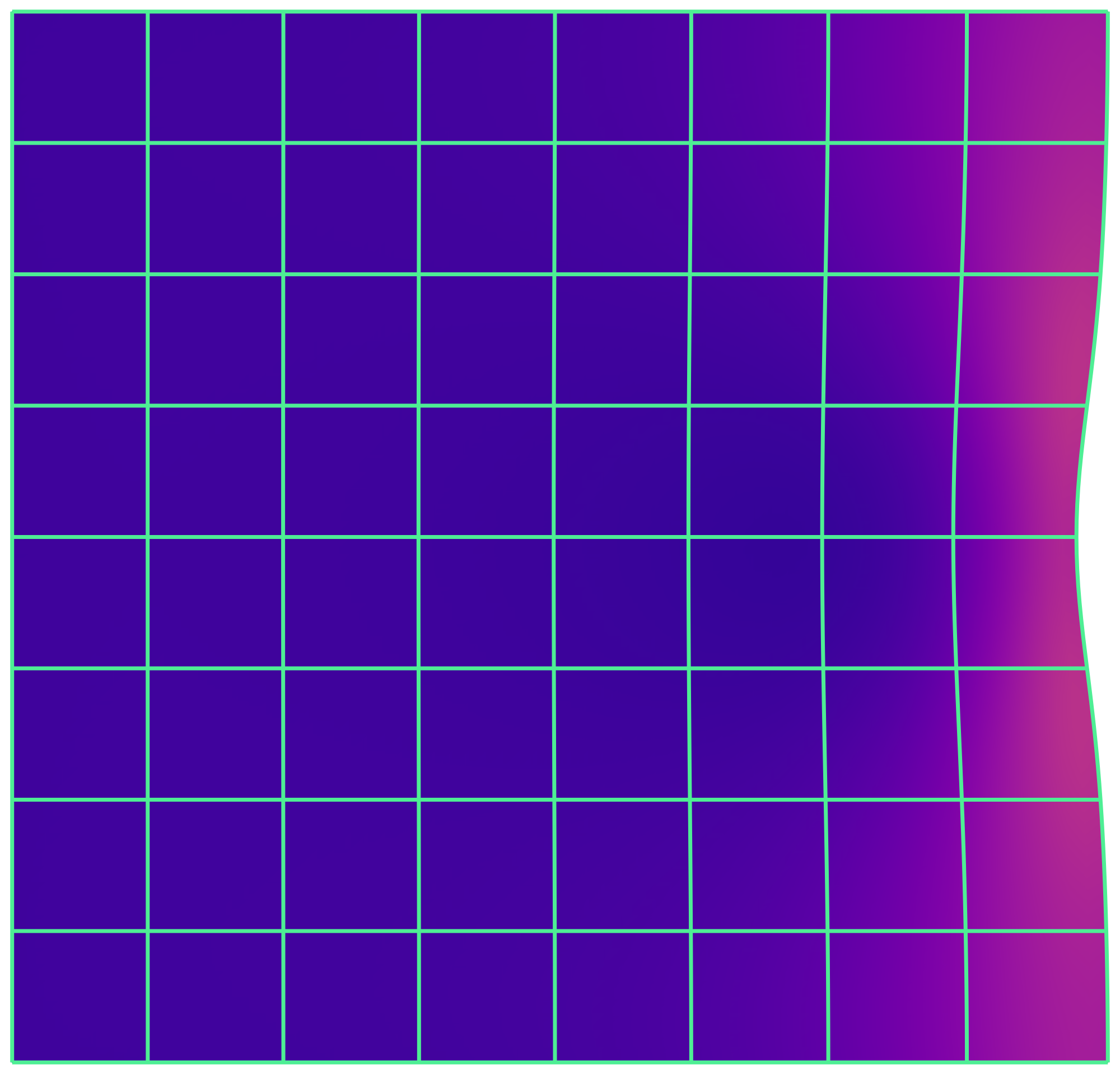}  
    \caption{previous thermal step}
    \label{fig:prev}
  \end{subfigure}
  \begin{subfigure}{.245\textwidth}
    \centering
    \includegraphics[height=3.1cm]{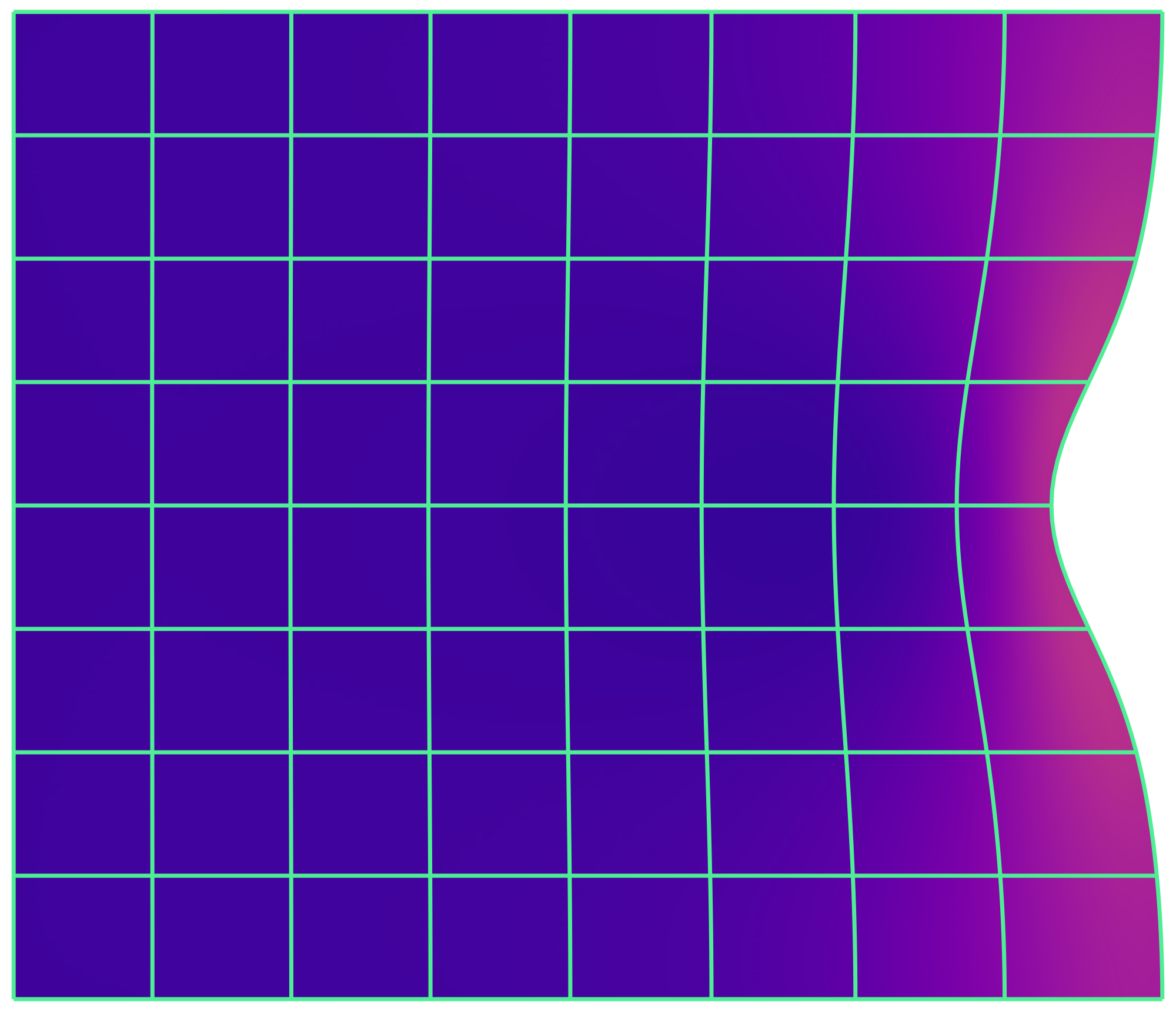}
    \caption{current mechanical step}
    \label{fig:cur_mech_step}
  \end{subfigure}
  \begin{subfigure}{.245\textwidth}
    \centering
    \includegraphics[height=3.1cm]{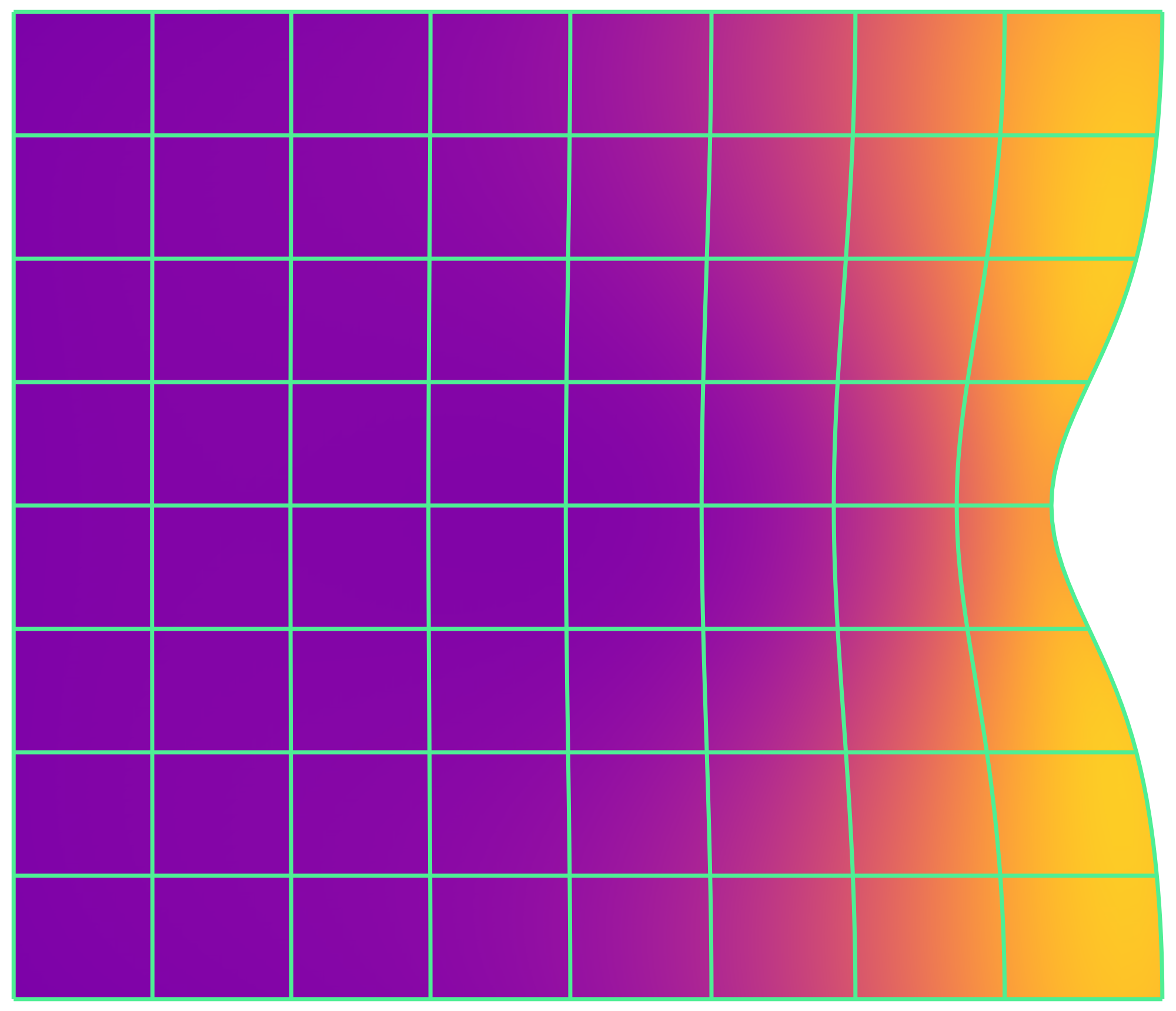}  
    \caption{current thermal step}
    \label{fig:cur_temp_step}
  \end{subfigure}
  \begin{subfigure}{.245\textwidth}
    \centering
    \includegraphics[height=3.1cm]{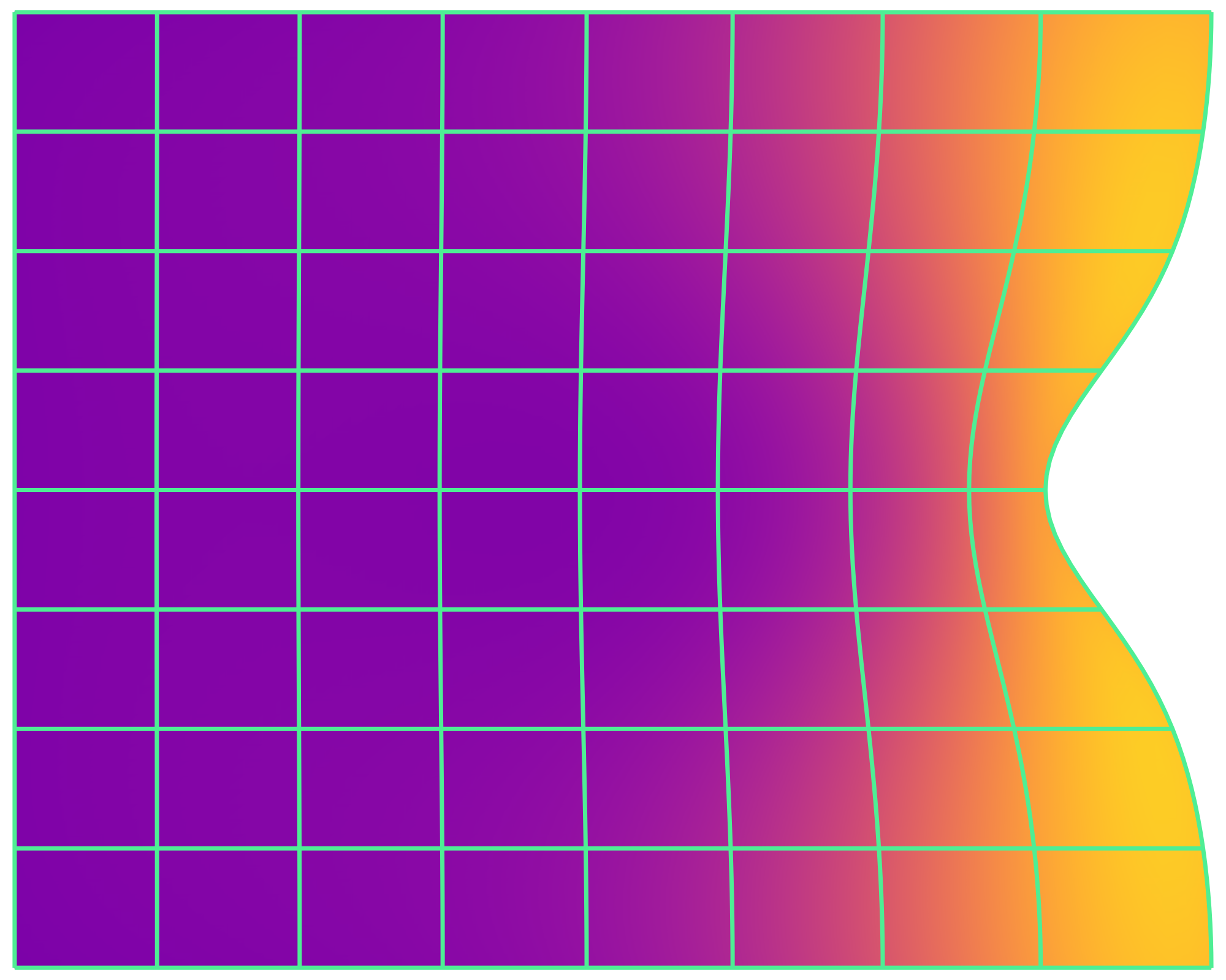}  
    \caption{next mechanical step}
    \label{fig:next_mech_step}
  \end{subfigure}
  \caption{4 consecutive steps of the staggered scheme in $\R^2$}
  \label{fig:minmoves}
\end{figure}
More precisely, suppose that we have already constructed $\ysts 0, \ldots,  \ysts{k-1} \in \Wid$, and $\tsts 0, \ldots,  \tsts{k-1} \in L^2_+(\Omega)$ for some $k \in \setof{1, \ldots,  T / \tau}$.
The next deformation $\ysts k$ is a solution of the minimization problem
\begin{equation}\label{mechanical_step}
  \min_{y \in \Wid} \Big\{
    \mechen(y) + \cplen\big(y, \tsts{k-1})
    + \frac{1}{\tau} \diss(\ysts{k-1}, y - \ysts{k-1}, \tsts{k-1})
    - \lp \langle \lst{k}, y \rangle
  \Big\},
\end{equation}
where
\begin{equation}\label{forces_mech_step}
  \langle \lst{k}, y \rangle
    \defas \int_\Omega \fst k \cdot y \di x
    + \int_{\Gamma_N} \gst k \cdot y \di \haus^{d-1}
\end{equation}
for $\fst k \defas \CCC \tau^{-1} \BBB \int_{(k-1)\tau}^{k\tau} f(t) \di t$  and \BBB $g^{(k)}_\tau \defas \CCC \tau^{-1} \BBB \int_{(k-1)\tau}^{k\tau} g(t) \di t$.
We define the $k$-th temperature step $\tsts k$ as a solution of the minimization problem
\begin{align}\label{thermal_step}
  \min_{\theta \in H^1_+(\Omega)}
  \Bigg\{
    &\int_\Omega \int_0^\theta \frac{1}{\tau}\big(
      \inten(\nabla \ysts{k},s)
      - \inten(\nabla \ysts{k-1},\tsts{k-1})
    \big) \di s \di x
    + \int_\Omega \frac{1}{2} \nabla \theta
      \cdot \hcm(\nabla \ysts{k-1}, \tsts{k-1}) \nabla \theta \di x \notag \\
    &- \int_\Omega h_\tau(\ysts{k}, \ysts{k-1}, \tsts{k-1}) \, \theta \di x
    + \frac{\kappa}{2} \int_\Gamma (\theta - \lp^\alpha\btst k )^2 \di \haus^{d-1}
  \Bigg\},
\end{align}
where $h_\tau$ plays the role of a heat source given by
\begin{equation}\label{def:h_tau}
  h_\tau(\ysts k,\ysts {k-1},\tsts {k-1})
  \defas \pl_F \cplpot(\nabla \ysts{k-1}, \tsts{k-1})
    : \ddif \nabla \ysts k
  + \drate(\nabla \ysts{k-1}, \ddif \nabla \ysts k, \tsts{k-1})
\end{equation}
and $\btst k \defas \frac{1}{\tau} \int_{(k-1)\tau}^{k\tau} \bt(t) \di t$.
The underlying idea is that the Euler-Lagrange equations associated to \eqref{mechanical_step} and \eqref{thermal_step} lead to time-discretized variants of the equations \eqref{main_evol_eq},  see \eqref{mechanical_step_single} and \eqref{el_thermal_step} below. \BBB
Supposing that the steps $\ysts 0, \ldots,  \ysts{T / \tau}$ and $\tsts 0, \ldots,  \tsts{T / \tau}$ as described above exist, we define \NNN interpolations as follows: for $k \in \setof{0, \ldots,  T / \tau}$, \NNN we let $\overline{y}_\tau(k\tau) =   \underline{y}_\tau(k\tau) = \hat{y}_\tau(k\tau)  \defas  \ysts k$   and for  $t \in ((k-1)\tau, k\tau)$ \BBB 
\begin{align}\label{y_interpolations}
  \overline{y}_\tau(t) \BBB &\defas \ysts k, &
    \underline{y}_\tau(t) \BBB & \defas \ysts{k-1}, &
    \hat{y}_\tau(t) \BBB \defas \frac{k \tau - t}{\tau} \ysts{k-1} + \frac{t - (k-1)\tau}{\tau} \ysts k.
\end{align}
A similar notation is employed for $\overline{y}_\tau$, $\underline{y}_\tau$, and $\hat{y}_\tau$. 
We now formulate our first main result concerning the convergence of solutions to the staggered scheme towards a weak solution of \eqref{main_evol_eq}--\eqref{initial_cond}.

\begin{theorem}[Staggered time-incremental scheme and convergence to solutions]\label{thm:van_tau}
  Given any $T > 0$ there exists $\tau_0 \in (0, 1]$ such that for any $\tau \in (0, \tau_0)$  the \BBB following holds: \\
{\rm (i)} (Existence of the scheme) The sequences $\ysts 0, \ldots,  \ysts{T / \tau}$ and $\tsts 0, \ldots,  \tsts{T / \tau}$ satisfying \eqref{def_0_step}, \eqref{mechanical_step}, and \eqref{thermal_step} exist.\\
{\rm (ii)} (Convergence to solutions) There exist $y \in L^\infty(I; \Wid)  \cap H^1(I; H^1(\Omega; \R^d))\BBB$ and $\theta \in L^1(I; W^{1,1}(\Omega))$ such that  the couple $(y, \theta)$ \BBB is a weak solution to \eqref{main_evol_eq}--\eqref{initial_cond} in the sense of Definition~\ref{def:weak_formulation},  and \BBB up to selecting a subsequence, it holds  that \BBB 
    \begin{align}
         \NNN \hat{y}_\tau \BBB &\to y \text{ in } L^\infty(I; W^{1,\infty}(\Omega;\R^d))
      & &\text{and} &
    \dot{\hat{y}}_\tau  &\to \dot y_\lp \text{ strongly in } L^2(I; H^1(\Omega; \R^d)),
      \label{van_tau_y_conv} \\
     \hat{\theta}_\tau  &\to \theta \text{ in } \CCC L^s(I \times \Omega) \BBB \BBB
      & &\text{and} &
    \hat{\theta}_\tau &\weakly \theta \text{ weakly in } L^r(I; W^{1,r}(\Omega))\BBB
      \label{van_tau_theta_conv}
    \end{align}
  as $\tau \to 0$  for any $r \in [1, \tfrac{d+2}{d+1})$ and $s \in [1, \frac{d+2}{d})$.  
     The same holds true if we replace $\hat{y}_\tau$ by $\overline{y}_\tau$ or $\underline{y}_\tau$ \NNN in the first part of \eqref{van_tau_y_conv},  and $\hat{\theta}_\tau$ by $\overline{\theta}_\tau$ or $\underline{\theta}_\tau$ \NNN in \eqref{van_tau_theta_conv}.    \BBB
   \end{theorem}

\NNN Let us mention that the proof shows   that weak solutions satisfy \CCC a \BBB total energy balance of the form
$$\frac{d}{\di t} \toten(y,\theta) = \lp \int_\Omega f \cdot \dot y \di x
    +  \lp \int_{\Gamma_N} g \cdot \dot  y \di \haus^{d-1}  - \kappa  \int_\Gamma (\theta - \lp^\alpha\theta_\flat)  \di \haus^{d-1}, $$
i.e., the total energy is conserved up to the work of the external loadings and the heat flux through $\Gamma$. \BBB 
\begin{remark}[Difference to scheme in \cite{MielkeRoubicek20Thermoviscoelasticity}]\label{rem:difference_to_mielke_in_steps}
  The scheme has several differences to the one considered in \cite[(4.5)--(4.7)]{MielkeRoubicek20Thermoviscoelasticity}.
  On the one hand, both steps in \cite{MielkeRoubicek20Thermoviscoelasticity} are suitably regularized.
  More precisely, in \eqref{mechanical_step} an additional dissipative term $\frac{\lambda}{2\tau} \Vert \nabla y - \nabla \ysts{k-1} \Vert_{L^2(\Omega)}^2$  is \BBB considered, where $\lambda > 0$ is a regularization parameter \NNN (called $\eps$ there), \BBB and in \eqref{thermal_step} the dissipation rate $\xi$ is replaced by a smoothly truncated version $\frac{\xi}{1 + \lambda \xi}$.
  On the other hand, the term $\pl_F \cplpot(\nabla \ysts{k-1}, \tsts{k-1}) \, \theta$ in  \eqref{thermal_step}--\eqref{def:h_tau} \BBB is replaced  by the more involved term $\int_0^\theta \pl_F \cplpot(\nabla \ysts k, s) \di s$.
  One of the main novelties in the present work is that the same result  on  existence and time-discrete approximations is achieved for the simpler, explicit, thermo-mechanical coupling term $\pl_F \cplpot(\nabla \ysts{k-1}, \tsts{k-1})$ and without  regularizing terms.
\end{remark}

\subsection{Passage to linearized thermoviscoelasticity}\label{sec: linear case}

We are now interested in the passage to a small strain regime $\lp \to 0$.
This is induced by small external loading, boundary traction, and external temperature as $\lp \to 0$, see \eqref{main_mechanical_eq} and the boundary conditions in \eqref{main_bc}.
In a similar fashion, we suppose that the initial values are small when $\lp$ is small, see \eqref{initial_cond}.
 At this point, we additionally  need to require \BBB
\begin{enumerate}[label=(W.\arabic*)]
  \setcounter{enumi}{3}
  \item \label{W_lower_bound_spec} $\elpot(F) \geq \ac \dist^2(F, SO(d))$ for all $F \in GL^+(d)$, \NNN and \BBB $\elpot(F) = 0$ if $F \in SO(d)$\CCC;\BBB
\end{enumerate}
\begin{enumerate}[label=(H.\arabic*)]
  \setcounter{enumi}{3}
  \item \label{H_bounds2} $H(0) = 0$\CCC;\BBB
\end{enumerate}
\begin{enumerate}[label=(C.\arabic*)]
  \setcounter{enumi}{5}
  \item \label{C_heatcap_cont} The \emph{heat capacity} $c_V(F, \theta) \defas - \theta \pl_\theta^2 \cplpot(F, \theta)$ for $F \in GL^+(d)$ and $\theta > 0$ as well as $\pl_{F \theta} \cplpot$ can be continuously extended to $GL^+(d) \times \R_+$\CCC;\BBB
  \item \label{C_thrid_order} \CCC For all $F \in GL^+(d)$ and $\theta > 0$ it holds that $\partial_{FF\theta} \cplpot(F, \theta) \leq \frac{C_0}{\max\{\theta, 1\}}$.\BBB
\end{enumerate}
In order to ensure the compatibility of \ref{W_lower_bound_spec} with \ref{W_lower_bound}, we assume $\aC \geq \ac(d + 1)$ from now on.
We write the equations \eqref{main_evol_eq} and the boundary conditions \eqref{main_bc} equivalently in terms of the \emph{rescaled displacement field}  $u= \lp^{-1} (y - \id)$ \BBB and the \emph{rescaled temperature}  $\mu =\lp^{-\alpha} \theta$. \BBB
Then,  for $\alpha \in [1,2]$, \BBB rescaling the equations by $\lp^{-1}$ and  $\lp^{-\alpha}$, \BBB respectively, and letting $\lp \to 0$ we obtain, at least formally, the system
\begin{equation}\label{viscoel_small}
\begin{aligned}
  - \BBB	\diver \big( \CW e(u) + \C_D e(\dot u) + \mathbb{B}^{(\alpha)} \mu \big) &= f, \\
	\bar c_V \dot\mu - \diver(\mathbb{K}_0 \nabla \mu) 
  &=\CD^{(\alpha)} e(\dot u): e(\dot u),
\end{aligned}
\end{equation}
along with the boundary conditions
\begin{align}\label{viscoel_small_bdy}
	u &= 0 \text{ on } I \times \Gamma_D, &
  \big( \CW e(u) + \CD e(\dot u) + \mathbb{B}^{(\alpha)} \mu \big) \nu &= g \text{ on } I \times \Gamma_N, &
  \mathbb{K}_0 \nabla \mu \cdot  \nu \BBB + \kappa \mu
  &= \kappa \bt \text{ on } I \times \Gamma
\end{align}
and initial conditions
\begin{equation}\label{initial_conds_lin}
  u(0) = u_0, \quad \mu(0) = \mu_0.
\end{equation}
Here, $e(u) \defas \frac{1}{2} (\nabla u + (\nabla u)^T)$ denotes the linear strain tensor,  and \BBB the \MMM tensors \BBB of elasticity and viscosity coefficients are defined by
\begin{align}\label{def_WD_tensors}
  \C_W &\defas \partial^2_{F}\elpot(\Id), &
  \C_D &\defas \partial^2_{\dot F} R(\Id,  0 , \BBB  0) = 4D(\Id,0).
\end{align}
Moreover, the heat conductivity tensor and the heat capacity (see also \ref{C_heatcap_cont}) at \CCC zero temperature \BBB and the natural material state \BBB are given by
\begin{align}\label{lin_heat_cond_cap}
  \mathbb{K}_0(x) &\defas \mathbb{K}(x,0), &
  \bar c_V & \defas c_V(\Id, 0).
\end{align}
\NNN Eventually, \BBB we have the $\alpha$-dependent quantities
\begin{align}\label{alpha_dep}
  \mathbb{B}^{(\alpha)} &= \begin{cases}
    \partial_{F\theta} \cplpot(\Id,0) & \text{if } \alpha=1 \\
    0  & \text{if } \alpha \in (1,2]
  \end{cases}, &
  \CD^{(\alpha)} &= \begin{cases}
    0 & \text{if } \alpha \in [1,2) \\
    \CD  & \text{if } \alpha =2,
  \end{cases}
\end{align}
where $\mathbb{B}^{(\alpha)}$ plays the role of a \emph{thermal expansion matrix}.
Notice that in the formal analysis above the  elasticity \BBB  tensor does not depend on the coupling potential.
This is due to the fact that $\pl_F^2 \cplpot(\Id, 0) = 0$, see \ref{C_zero_temperature}.

\NNN Although \BBB the nonlinear system is given for a nonsimple material, in the limit we obtain equations without spatial gradients of $e(u)$. \NNN This is a consequence of the growth conditions in  \BBB \ref{H_bounds}.   Moreover, there is an interesting decoupling effect due to the different scaling \CCC of coupling terms in \BBB the mechanical and the heat-transfer equation, expressed in terms of the $\alpha$-dependent quantities in \eqref{alpha_dep}.
This computation also shows why we restrict to the range $\alpha \in [1, 2]$. Indeed, formally, we would have $\mathbb{B}^{(\alpha)} = + \infty$ for $\alpha < 1$ while $\mathbb{C}_D^{(\alpha)} = +\infty$ for $\alpha > 2$.

The second main goal of this article is to show that the above formal linearization can be made rigorous.
In the case $\alpha \in [1, 2)$, our analysis requires a regularization of the thermal evolution.
More precisely, we define the $k$-th thermal step through
\begin{equation}\label{thermal_step_reg}\tag{\ref*{thermal_step}$_\lp$}
   \eqref{thermal_step} \BBB \text{ with $\drate$ replaced by $\rdrate$,}
\end{equation}
where
\begin{equation}\label{def_rdrate}
  \rdrate \defas \begin{cases}
    \drate &\text{if } \drate \leq 1, \\
    \drate^{\alpha/2} &\text{else.}
  \end{cases}
\end{equation}
 \NNN Due to the different scaling of the mechanical and the heat-transfer equation, the existence \MMM of a solution to \BBB \NNN the scheme is more delicate for $\eps$ small and $\alpha \neq 2$. More specifically, we need higher integrability of $\inten$ defined in \eqref{Wint}  in $L^{2/\alpha}$  which can be guaranteed by the choice in \eqref{def_rdrate}. We refer to Subsection~\ref{sec: adaptions} below for details. We \CCC also \NNN emphasize that for  $\alpha = 2$ no regularization is applied as $\drate_\alpha^{\rm{reg}} = \drate$. A similar result  as \BBB Theorem \ref{thm:van_tau} holds true in the regularized setting.
\begin{proposition}[Vanishing time-discretization in the regularized  \MMM nonlinear \BBB setting]\label{cor:van_tau_reg}
  Given any $T > 0$ there exists  $\lp_0,\tau_0 \in (0, 1]$ \BBB such that for any  $\tau \in (0, \tau_0)$  \BBB and $\lp \in (0, \lp_0)$ the following holds: \\
{\rm (i)} (Existence of the scheme) The sequences $\yst 0, \ldots,  \yst{T / \tau}$ and $\tst 0, \ldots,  \tst{T / \tau}$ satisfying \eqref{def_0_step}, \eqref{mechanical_step}, and \eqref{thermal_step_reg} exist.\\
{\rm (ii)} (Convergence to solutions) The convergences \eqref{van_tau_y_conv}--\eqref{van_tau_theta_conv} \NNN towards a limit $(y_\lp, \theta_\lp)$ \BBB hold true for the interpolations of the steps from {\rm (i)}.  Here $(y_\lp, \theta_\lp)$ is a weak solution to the system  in a sense similar to Definition \ref{def:weak_formulation}, namely  \eqref{weak_limit_mechanical_equation} is satisfied and \eqref{weak_limit_heat_equation} holds with $\drate$ replaced by $\rdrate$. \BBB 
\end{proposition}

We will prove that \eqref{viscoel_small} admits a unique weak solution and that solutions of the above described regularization  guaranteed by  Proposition  \ref{cor:van_tau_reg}(ii) \BBB converge to the solution of \eqref{viscoel_small} in a suitable sense.
Setting
\begin{equation}\label{eq: H1 strange}
  H^1_{\Gamma_D}(\Omega; \R^d) \defas \setof{u \in H^1(\Omega; \R^d) \colon u = 0 \text{ on } \Gamma_D}
\end{equation}
we have the following definition of weak solutions  for the linearized system. \BBB
\begin{definition}[Weak solution of the linearized system]\label{def:weak_form_linear_evol}
A couple $(u, \mu) \colon I \times \Omega \to \R^d \times \R$ is called a \emph{weak solution} to the initial-boundary-value problem \eqref{viscoel_small}--\eqref{initial_conds_lin} if $u \in H^1(I; H^1_{\Gamma_D}(\Omega; \R^d))$ with $u(0, \cdot) = u_0$, $\mu \in L^1(I; W^{1,1}(\Omega))$ with $\mu \ge 0$ a.e., and if it satisfies the identities
\begin{equation}\label{linear_evol_mech}
  \intQ \big( \CW e(u) + \CD e(\dot u) + \mu \mathbb{B}^{(\alpha)} \big) : \nabla z \di x \di t
    = \intQ f \cdot z \di x \di t + \intSN g \cdot z \di \haus^{d-1} \di t
\end{equation}
for any $z \in C^\infty(I \times \overline\Omega; \R^d)$ with $z = 0$ on $I \times \Gamma_D$, as well as
\begin{align}\label{linear_evol_temp}
  &\intQ
      \hc \nabla \mu \cdot \nabla \vphi
      - \CD^{(\alpha)} e(\dot u) : e(\dot u) \vphi
      - \bar c_V \mu \dot \vphi \di x \di t
    + \kappa \int_\Gamma \mu \vphi \di \haus^{d-1} \di t \notag \\
  &= \kappa \int_\Gamma   \theta_\flat \BBB \vphi \di \haus^{d-1} \di t
    + \bar c_V \int_\Omega \mu_0 \vphi(0) \di x
\end{align}
for any $\vphi \in C^\infty(I \times \overline\Omega)$   with $\vphi(T) = 0$. \BBB
\end{definition}

Indeed, it is a standard matter to check that sufficiently smooth weak solutions lead to the classical formulation \eqref{viscoel_small}. Next, we state the relation between time-continuous or time-discrete solutions of  the nonlinear system \BBB and solutions to \NNN  \eqref{viscoel_small}--\eqref{initial_conds_lin}. \BBB
\begin{theorem}[Passage to linearized thermoviscoelasticity]\label{thm:linearization_right_diag}
  Under the above assumptions we have:\\
{\rm (i)} There exists a unique weak solution $(u, \mu)$ to  \eqref{viscoel_small}--\eqref{initial_conds_lin} \BBB in the sense of Definition \ref{def:weak_form_linear_evol}.\\
{\rm (ii)} Given any sequence $(\lp_k)_k$ converging to zero and any sequence of weak solutions $(y_{\lp_k}, \theta_{\lp_k})$  given \BBB by Proposition \ref{cor:van_tau_reg}\,{\rm(ii)}, the functions $u_{\lp_k} \defas \lp_k^{-1} (y_{\lp_k} - \id)$ and $\mu_k = \lp_k^{-\alpha} \theta_{\lp_k}$ satisfy
    \begin{align*}
   \NNN  u_{\lp_k} &\to u \text{ in } L^\infty(I;H^1(\Omega; \R^d)),  &  \NNN \dotauk &\to   \dot u \text{  in } L^2(I; H^1(\Omega; \R^d)), \BBB \\ 
      \mu_{\lp_k} &\to \mu \text{ in } \CCC L^s(I \times \Omega)\BBB,  & 
      \mu_{\lp_k} &\weakly \mu \text{ weakly in } L^r(I; W^{1, r}(\Omega))
    \end{align*}
    for any  $s \in [1,  \frac{d+2}{d})$ and $r \in [1, \tfrac{d+2}{d+1})$. \\
    {\rm (iii)} Given sequences $(\lp_k)_k$, $(\tau_k)_k$ converging to zero and any sequence $(\overline{y}_{\lp_k\CCC,\BBB \tau_k}, \overline{\theta}_{\lp_k\CCC,\BBB \tau_k})$ of time-discrete solutions   given by Proposition \ref{cor:van_tau_reg}\,{\rm (i)}, $\overline{u}_k \defas \lp_k^{-1} (\overline{y}_{\lp_k\CCC,\BBB \tau_k} - \id)$ and $\overline{\mu}_k = \lp_k^{-\alpha} \overline{\theta}_{\lp_k\CCC,\BBB \tau_k}$ satisfy
    \begin{align*}
         \NNN \hat{u}_k &\to u \text{ in } L^\infty(I;H^1(\Omega; \R^d)),  &  \NNN   \dot{\hat{u}}_k &\to u \text{ in } L^2(I;H^1(\Omega; \R^d)), \BBB \\
     \hat{\mu}_k &\to \mu \text{ in } \CCC L^s(I \times \Omega)\BBB, & 
      \hat{\mu}_k &\weakly \mu \text{ weakly in } L^r(I; W^{1, r}(\Omega))
    \end{align*}
    for any $s \in [1,  \frac{d+2}{d})$ and $r \in [1, \tfrac{d+2}{d+1})$. \NNN Apart from the convergence of $\dot{\hat{u}}_k$, \BBB the same holds true if we replace $\hat{y}_{\lp_k\CCC,\BBB \tau_k}$ by $\overline{y}_{\lp_k\CCC,\BBB \tau_k}$ or $\underline{y}_{\lp_k\CCC,\BBB \tau_k}$ and $\hat{\theta}_{\lp_k\CCC,\BBB \tau_k}$ by $\overline{\theta}_{\lp_k\CCC,\BBB \tau_k}$ or $\underline{\theta}_{\lp_k\CCC,\BBB \tau_k}$, and consider the corresponding rescaled quantities.
\end{theorem}

\NNN Note particularly that we obtain strong convergence of strains and strain rates. \BBB Finally, we study the relation between the time-discrete solutions in the nonlinear and the linear setting,  as well as the convergence of time-discrete solutions in the linearized setting  under vanishing time-discretization.  
\begin{theorem}[Passage to linearized thermoviscoelasticity, time-discrete solutions] \label{thm:linearization_left_bottom} The following properties hold: \\ \BBB 
{\rm (i)} \NNN Let $\tau$ be sufficiently small. \BBB For every $k \in \{1, \ldots,  T / \tau \}$ we have \CCC as $\lp \to 0$ \BBB
    \begin{align}\label{eq: again a convergence}
      \frac{1}{\lp} (\yst k - \id) &\to u_\tau^{(k)} \text{ strongly in } H^1(\Omega; \R^d), &
      \frac{1}{ \lp^\alpha } \tst k &\weakly \mu_\tau^{(k)} \text{ weakly in } W^{1, r}(\Omega)
    \end{align}
 for any $r \in [1, \tfrac{d+2}{d+1})$,     where $u_\tau^{(k)}$ is uniquely determined by 
    \begin{equation}\label{el_mech_step_lin}
      \int_\Omega \big(\CW e(u_\tau^{(k)}) \CCC + \CD e(\ddif u_\tau^{(k)})\BBB + \mu_\tau^{(k-1)} \mathbb{B}^{(\alpha)}\big) : \nabla z \di x
      - \langle \lst k , z \rangle = 0
    \end{equation}
    for all $z \in H^1_{\Gamma_D}(\Omega; \R^d)$ and $\mu_\tau^{(k)}$ is uniquely determined by
    \begin{equation}\label{el_thermal_step_lin}
      \int_\Omega
        \Big(
          \bar c_V \ddif \mu_\tau^{(k)}
          \NNN - \BBB \CD^{(\alpha)} e(\ddif u_\tau^{(k)}) : e(\ddif u_\tau^{(k)})
        \Big) \vphi \di x
      + \int_\Omega \mathbb{K}_0 \nabla \mu_\tau^{(k)} \cdot \nabla \vphi \di x
      + \kappa \int_\Gamma (\mu_\tau^{(k)} - \btst k) \vphi \di \haus^{d-1} = 0
    \end{equation}
    for all $\vphi \in C^\infty(\overline \Omega)$, where $\ddif u_\tau^{(k)} \defas (u_\tau^{(k)} - u_\tau^{(k-1)})/\tau$ and $\ddif \mu_\tau^{(k)} \defas (\mu_\tau^{(k)} - \mu_\tau^{(k-1)})/\tau$.\\
{\rm (ii)} Given $(u_\tau^{(k)})_k$ and $(\mu_\tau^{(k)})_k$ from {\rm (i)}, define $\hat{u}_\tau$ and $\hat{\mu}_\tau$ similarly to \eqref{y_interpolations}.
    Then,
    \begin{equation}\label{linearized_tau_zero_convs}
    \begin{aligned}
      \NNN  \hat{u}_\tau &\to u \text{ in } L^\infty(I;H^1(\Omega; \R^d)), &  \NNN \dot{\hat{u}}_\tau &\to \dot{\hat{u}} \text{ in } L^2(I;H^1(\Omega; \R^d)), \BBB \\
       \hat{\mu}_\tau &\to \mu \text{ in } \CCC L^s(I \times \Omega)\BBB, &  
      \hat{\mu}_\tau &\weakly \mu \text{ weakly in } L^r(I; W^{1, r}(\Omega))
    \end{aligned}
    \end{equation}
     as $\tau \to 0$     for any $s \in [1, \frac{d+2}{d})$ and $r \in [1, \frac{d+2}{d+1})$, where $(u, \mu)$ is the  unique \BBB weak solution of \NNN \eqref{viscoel_small}--\eqref{initial_conds_lin} \BBB in the sense of Definition \ref{def:weak_form_linear_evol}. \NNN Apart from the time derivative, \BBB the convergences in \eqref{linearized_tau_zero_convs} also hold for the other interpolations. \BBB
\end{theorem}

\begin{remark}[Variational  structure \BBB in the time-discrete linear setting]
  With regard to Theorem \ref{thm:linearization_left_bottom}, we can in fact show that $u_\tau^{(k)}$ is the unique solution of the minimization problem
  \begin{equation*}
    \argmin_{u \in H^1_{\Gamma_D}(\Omega; \R^d)}
      \Bigg\{
        \frac{1}{2} \int_\Omega (\CW e(u) + \mu_\tau^{(k-1)} \mathbb{B}^{(\alpha)}) : e(u) \di x
        + \frac{1}{2\tau} \int_\Omega \CD e(u - u_\tau^{(k-1)}) : e(u - u_\tau^{(k-1)}) \di x
        - \sprod{\lst k}{u}
      \Bigg\}
  \end{equation*}
  and for $\alpha \in [1, 2)$ that \NNN the nonnegative function \BBB $\mu_\tau^{(k)}$ is the unique solution of the minimization problem
  \begin{equation}\label{min_problem_temp_lin}
    \argmin_{\NNN \mu \in H^1(\Omega)\BBB}
      \Bigg\{
        \frac{\bar c_V}{2\tau} \int_\Omega (\mu - \mu_\tau^{(k-1)})^2 \di x
        \NNN - \BBB \int_\Omega \CD^{(\alpha)} e(\ddif u_\tau^{(k)}) : e(\ddif u_\tau^{(k)}) \mu
          + \frac{1}{2} \mathbb{K}_0 \nabla \mu \cdot \nabla \mu \di x
        + \frac{\kappa}{2} \int_{\Gamma} (\mu - \btst{k})^2 \di x
      \Bigg\}.
  \end{equation}
  From \NNN the a priori bounds \BBB in the nonlinear setting, we will be only able to prove that $\CD e(\ddif u_\tau^{(k)}) : e(\ddif u_\tau^{(k)}) \in L^1(\Omega)$.
  Consequently, the functional in \eqref{min_problem_temp_lin} might not be well-defined on $H^1(\Omega)$ \NNN for $\alpha =2$. \BBB
  Nevertheless,   for sufficiently smooth $\Gamma$, smooth functions $f$ and $\theta_\flat$, \NNN and $\Gamma_D=\Gamma$,  \BBB it follows by elliptic regularity theory that $\CD e(\ddif u_\tau^{(k)}) : e(\ddif u_\tau^{(k)}) \in L^2(\Omega)$.
  In this case, $\mu_\tau^{(k)}$ is a minimizer of \eqref{min_problem_temp_lin} also for $\alpha = 2$. 
\end{remark}
\NNN
Section \ref{sec:staggered_scheme} is devoted to existence of the staggered time-incremental scheme leading to Theorem \ref{thm:van_tau}(i) and Proposition \ref{cor:van_tau_reg}(i). Then, in Section \ref{sec:tau_to_zero_delta_fixed} we pass to the limit $\tau \to 0$ and show Theorem \ref{thm:van_tau}(ii) and Proposition \ref{cor:van_tau_reg}(ii).  Eventually, in Section \ref{sec: linearization} we address the limit $\eps \to 0$ and prove Theorems \ref{thm:linearization_right_diag} \BBB and   \ref{thm:linearization_left_bottom}.

\section{Staggered time-incremental scheme}\label{sec:staggered_scheme}
This section is devoted to the analysis of the staggered time-incremental scheme described in the previous section. Let us start with some fundamental auxiliary results.

\begin{lemma}[A \NNN priori estimates, \BBB positivity of determinant]\label{lem:pos_det}
  Given $M > 0$ there exists a constant $C_M > 0$ such that for all $y \in \Wid$ with $\mechen(y) \leq M$ (where $\mechen$ is defined in \eqref{mechanical}) it holds that
  \begin{align}\label{pos_det}
    \norm{y}_{W^{2, p}(\Omega)} &\leq C_M, &
    \norm{y}_{C^{1, 1-d/p}(\Omega)} &\leq C_M, &
    \norm{(\nabla y)^{-1}}_{C^{1 - d/p}(\Omega)} &\leq C_M, &
    \det(\nabla y) \geq \frac{1}{ C_M} \MMM \text{ in } \Omega.\BBB
  \end{align}
  If $W$ additionally satisfies  {\rm \ref{W_lower_bound_spec}}, there exists a universal constant $C$ and \CCC a constant \BBB $C_M^* >0$ with $C_M^* \to 0$ as $M \to 0$ such that 
  \begin{align}
    \norm{y - \id}_{H^1(\Omega)} &\leq C \norm{\dist(y, SO(d))}_{L^2(\Omega)}, \label{H1_dist_to_id} \\
    \norm{y - \id}_{W^{1, \infty}(\Omega)} &\leq C^*_M. \label{W1infty_dist_to_id}
  \end{align} \BBB
\end{lemma}

\begin{proof}
  For a proof  of the first part \BBB we refer to \cite[Theorem~3.1]{MielkeRoubicek20Thermoviscoelasticity} relying on a result in \cite{HealeyKroemer09Injective}.  The second part can be found in \cite[Lemma 4.2]{FiredrichKruzik18Onthepassage}, \NNN where $\mathscr{S}^M_\delta$ therein simply corresponds to $\mechen(y) \leq M\delta^2$. \BBB
\end{proof}

\begin{lemma}[Generalized Korn's inequality]\label{lem:gen_korn}
  Given $M > 0$ there exists a constant $c_M > 0$ such that for all $v \in \NNN   H^1_{\Gamma_D}(\Omega; \R^d) \BBB$ and $y \in \Wid$ with $\mechen(y) \leq M$ it holds that
  \begin{equation*}
    \int_\Omega \abs*{(\nabla v)^T \nabla y + (\nabla y)^T \nabla v}^2 \di x
    \geq c_M  \NNN \Vert v \Vert^2_{H^1(\Omega)}. \BBB
  \end{equation*}
\end{lemma}

\begin{proof}
  The statement can be found in \cite[Corollary~3.4]{MielkeRoubicek20Thermoviscoelasticity}, relying on the result in \cite{Pompe03Korns}.
\end{proof}

\begin{lemma}[Heat conductivity]\label{lem:bound_hcm}
  For any $M > 0$ there exist constants $c_M, \, C_M > 0$ such that \MMM for \BBB  $y \in \Wid$ satisfying \BBB $\mechen(y) \leq M$ and $\theta \in L^{\CCC 1\BBB}(\Omega)$ we have that $\hcm(\nabla y, \theta)$ is well-defined   and
  \begin{equation}\label{bound_hcm}
    c_M \leq \hcm(\nabla y, \theta) \leq C_M.
  \end{equation}
\end{lemma}

\begin{proof}
  By Lemma \ref{lem:pos_det} we see that $(\nabla y(x))^{-1}$ exists for every $x \in \Omega$ which shows the well-definedness of $\hcm(\nabla y, \theta)$,  see \eqref{hcm}. \BBB   The bound in \eqref{bound_hcm} is a direct consequence of the latter three estimates in \eqref{pos_det} combined with \eqref{spectrum_bound_K}.
\end{proof}

\begin{lemma}[Estimate on coupling potential]
  For all $F \in GL^+(d)$ and $\theta \geq 0$ it holds that
  \begin{equation}\label{C_locally_lipschitz}
    \abs{\pl_F \cplpot(F, \theta)} \leq  2 \BBB \aC (\theta \wedge 1) (1 + \abs{F}).
  \end{equation}
\end{lemma}

\begin{proof}
  We start by proving \eqref{C_locally_lipschitz} for $\theta \le 1$.
  To this end, we use that $\partial_F\cplpot(F, 0) = 0$ (see \ref{C_zero_temperature}  and comments thereafter), \BBB \ref{C_bounds}, and apply the Fundamental Theorem of Calculus to get
  \begin{align*}
    |\partial_F\cplpot(F,\theta)|
    &= \Big|
        \partial_F\cplpot(F,0)
        + \int_0^\theta\partial_{F\theta}\cplpot(F,s)\di s
      \Big|
      \le \int_0^\theta |\partial_{F\theta}\cplpot(F,s)|\di s \\
    & \le \aC(1+|F|)
      \int_0^\theta \max\lbrace s,1\rbrace^{-1} \di s
      = \aC \theta (1 + \abs{F}).
  \end{align*}
  On the other hand, for $\theta \ge 1$, we use \ref{C_lipschitz} in the limit $\tilde{F} \to F$ to find $|\partial_F\cplpot(F,\theta)| \le \aC(1 +  2 \BBB |F|)$ for every $F \in GL^+(d)$.
\end{proof}

\subsection{\MMM Existence  of solutions to \CCC time-discretized \BBB \MMM schemes\BBB}\label{sec:single_step_well_defined}

In this subsection, we \NNN show \BBB that for sufficiently small $\tau \in (0, 1]$, depending on a bound of the mechanical energy of the previous deformation, a single time step of the staggered time-discretization scheme introduced in \eqref{mechanical_step}--\eqref{thermal_step} is well-defined. \NNN Here, we treat the case $\alpha = 2$ and $\eps \in (0,1]$ postponing necessary adaptions for $\alpha <2$ to Subsection~\ref{sec: adaptions} below.  \CCC We assume the same set-up of Subsection~\ref{sec:nonlinear_scheme}. More precisely, consider \BBB
initial steps $\yst 0 \defas y_{0, \CCC \lp \BBB} \in \Wid$ and $\tst 0 \defas \theta_{0, \CCC \lp \BBB} \in L^2_+(\Omega)$ \CCC with $y_{0, \lp}$ and $\theta_{0, \lp}$ as in \eqref{initial_cond}. \BBB
Further, let $f \in W^{1, 1}(I; L^2(\Omega; \R^d))$, $g \in W^{1, 1}(I; L^2(\Gamma_N; \R^d))$, $\theta_\flat \in \NNN W^{\CCC 1,1 \BBB}\BBB (I; L^2_+(\Gamma))$, and for each $k \in \{1, \ldots,  T / \tau\}$ let $\lst k$ be as in \eqref{forces_mech_step}.
Suppose that we have already constructed $\yst 0, \ldots,  \yst{k-1} \in \Wid$ and $\tst 0, \ldots,  \tst{k-1} \in L^2_+(\Omega)$ for some $k \in \CCC \{1, \ldots,  T / \tau \}\NNN$. \NNN (We always add an index $\eps$ for clarification.) \BBB  We first investigate the existence of the $k$-th mechanical step.
\begin{proposition}[Mechanical step]\label{prop:existence_mechanical_step}
  For any $M > 0$ there exists $\tau_0 \in (0, 1]$ such that if \CCC $k \in \{1, \ldots, T /\tau\}$, \BBB $\tau \in (0, \tau_0)$\CCC, \BBB and $\mechen(\yst{k-1}) \leq M$ the minimization problem \eqref{mechanical_step} is well-posed, i.e.,
  \begin{equation}\label{mechanical_step2}
    \min_{y \in \Wid} \Big\{
      \mechen(y) + \cplen\big(y, \tst{k-1})
      + \frac{1}{\tau} \diss(\yst{k-1}, y - \yst{k-1}, \tst{k-1})
      - \CCC \lp \BBB \langle \lst k , y \rangle
    \Big\}
  \end{equation}
  attains a solution.
  Furthermore, such a minimizer $\yst{k}$ solves the corresponding Euler-Lagrange equation, i.e., it holds for all $z \in \Wzero$ \NNN (see \eqref{eq: 000}) \BBB that
  \begin{equation}\label{mechanical_step_single}
    \int_\Omega
      \big(
        \pl_F \felpot(\nabla \yst{k}, \tst{k-1})
        + \pl_{\dot F} \disspot(\nabla \yst{k-1}, \ddif \nabla\yst{k}, \tst{k-1})
      \big) : \nabla z
      + \pl_G \hypot(\nabla^2 \yst{k}) \cdddot \nabla^2 z \di x
      - \NNN \eps \BBB \langle \ell^{(k)}_\tau , z \rangle = 0.
  \end{equation}
\end{proposition}

\begin{proof}
  We provide  the proof for \BBB  the coercivity \NNN  in \BBB $W^{2, p}(\Omega; \R^d)$. The remaining argument   coincides with the one in \cite[Proposition 4.1]{MielkeRoubicek20Thermoviscoelasticity},  and  we only include a brief sketch for convenience of the reader. \BBB   Let us shortly write $\tilde y \defas \yst{k-1}$ and $\tilde \theta \defas \tst{k-1}$.
  Let $(y_n)_n \subset \Wid$ be a minimizing sequence for the problem in \eqref{mechanical_step2}.
  Using $ \tilde{y}\BBB$ as a competitor we can, without loss of generality, assume that for all $n \in \N$
  \begin{equation*}
    \mechen(y_n)
      + \cplen(y_n, \tilde \theta)
      + \frac{1}{\tau} \diss(\tilde y, y_n - \tilde y,  \tilde \theta) \BBB
      - \NNN \eps \BBB \sprod{\lst k}{y_n}
    \leq \mechen(\tilde y) + \cplen(\tilde y, \tilde \theta)
      - \CCC \lp \BBB \sprod{\lst k}{\tilde y}
  \end{equation*}
  and therefore
  \begin{equation}\label{comparison_yn}
    \mechen(y_n)
      + \frac{1}{\tau} \diss(\tilde y, y_n - \tilde y, \MMM\tilde \theta) \BBB
    \leq \mechen(\tilde y)
      + \abs{\cplen(y_n, \tilde \theta) - \cplen(\tilde y, \tilde \theta)}
      + \CCC \lp \BBB \abs{\sprod{\lst k}{y_n - \tilde y}}.
  \end{equation}
  By Lemma \NNN \ref{lem:gen_korn} \BBB and \ref{D_bounds} there exists $c_M > 0$ (only depending on $M$) such that
  \begin{equation*}
    \frac{1}{\tau} \diss(\tilde y, y_n - \tilde y, \CCC \tilde \theta \BBB)
    \geq \frac{c_M}{\tau} \int_\Omega \abs{\nabla y_n - \nabla \tilde y}^2 \di x.
  \end{equation*}
  By \eqref{C_locally_lipschitz}, the Fundamental Theorem of Calculus, Young's inequality with constant $c_M/(2\tau)$, \BBB and \ref{W_lower_bound} we derive
  \begin{align*}
    &\big|
      \cplen(  y_n, \tilde \theta)
      - \cplen(  \tilde y, \tilde \theta)
    \big|
    \leq 2\aC \int_\Omega
        (\tilde{\theta} \wedge 1) \BBB (1 + \abs{\nabla y_n} + \abs{\nabla \tilde y})
      \abs{\nabla y_n - \nabla \tilde y} \di x \\
    &\quad\leq  C_M \BBB \tau \int_\Omega
      \CCC (\tilde{\theta} \wedge 1)^2 \BBB \Big(1 + \NNN 2 C_0c_0^{-1}  + c_0^{-1}\elpot(\nabla y_n) + c_0^{-1}\elpot(\nabla \tilde y) \BBB\Big) \BBB \di x + \frac{c_M}{4\tau} \int_\Omega \abs{\nabla y_n - \nabla \tilde y}^2 \di x \\
    &\quad\leq   C_M \BBB  \tau \big(\Vert\tilde{\theta} \wedge 1 \Vert^2_{L^2(\Omega)} \BBB + \elen(y_n) + \elen(\tilde y)\big)
      + \frac{c_M}{ 4\BBB \tau} \int_\Omega \abs{\nabla y_n - \nabla \tilde y}^2 \di x
  \end{align*}
   for $C_M$ sufficiently large depending on $M$ and $c_0$.    By using Poincar\'e's inequality, \NNN the trace estimate on the bulk and surface term, respectively, \BBB and Young's inequality with constant \CCC $c_M / (4C\tau\eps)$ \BBB we derive that
  \begin{align*}
   \NNN \eps \BBB  \abs{\sprod{\lst k}{y_n - \tilde y}} &= \NNN \eps \MMM\Big| \BBB\int_\Omega f^k_\tau \cdot (y_n - \tilde y) \di x
      + \int_{\Gamma_N} g^k_\tau \cdot (y_n - \tilde y) \di \haus^{d-1} \MMM\Big| \BBB\\
    & \le C \NNN \eps \BBB \big(
        \norm{f^k_\tau}_{L^2(\Omega)} + \norm{g^k_\tau}_{L^2(\Gamma_N)}
      \big) \Vert \nabla y_n - \nabla \tilde y \Vert_{L^2(\Omega)} \\
    &\leq C_M \tau  \NNN \eps^2 \BBB (
        \norm{f^k_\tau}_{L^2(\Omega)}^2 + \norm{g^k_\tau}_{L^2(\Gamma_N)}^2
      )
      + \frac{c_M}{4 \tau} \norm{\nabla y_n -  \nabla \tilde y}_{L^2(\Omega)}^2.
  \end{align*}
Combining the aforementioned estimates with \eqref{comparison_yn}, \NNN and using $\mathcal{M} \ge \mathcal{W}^{\rm el}$  \BBB we get  \BBB
  \begin{equation}\label{eq: safe for later}
  \begin{aligned}
    &(1 - C_M \tau) \mechen(y_n) +   \frac{c_M}{2\tau} \BBB \norm{\nabla y_n -\nabla  \tilde  y}_{L^2(\Omega)}^2 \\
    &\quad\leq (1 + C_M \tau) \mechen(\tilde{y})
      + C_M \tau (
        \Vert\tilde{\theta} \wedge 1 \Vert^2_{L^2(\Omega)} \BBB + \CCC \lp^2 \BBB \norm{f^k_\tau}_{L^2(\Omega)}^2
        + \CCC \lp^2 \BBB \norm{g^k_\tau}_{L^2(\Gamma_N)}^2
      ).
  \end{aligned} 
  \end{equation}  
  For $\tau_0$ sufficiently small such that $C_M\tau_0 \le \CCC 1/2 \BBB$, \BBB   Lemma \ref{lem:pos_det} then shows the desired coercivity in $W^{2,p}(\Omega;\R^d)$.
  The functional is weakly lower semicontinuous on $W^{2,p}(\Omega;\R^d)$ by the convexity of $\hypot$, see \ref{H_regularity}, the compact embedding $W^{2,p}(\Omega; \R^d) \subset W^{1,\infty}(\Omega; \R^d)$, and the continuity of $\elpot$, \CCC $\cplpot$, and $\disspot$. \BBB   This proves the existence of a minimizer.

  For the derivation of the Euler-Lagrange equation,  we recall the definitions in \eqref{hyperelastic}--\eqref{eq: free energy}. The \BBB treatment of the convex term $\hyen$ is standard by \ref{H_bounds} and \ref{H_regularity}.
  The G\^ateaux differentiability of the other terms relies on the uniform bound on gradients  and the control on the determinant, \BBB see \eqref{pos_det}.
  We refer also to \cite[Proposition 3.2]{MielkeRoubicek20Thermoviscoelasticity}.
\end{proof}

 From the previous proof, we directly deduce  the following. \BBB

\begin{lemma}[Bound on mechanical energy and dissipation]\label{lem:bad_mechen_bound}
  For any $M > 0$ there \NNN exist \BBB constants $c_M, \, C_M > 0$ and $\tau_0 \in (0, 1]$ such that if \CCC $k \in \{1, \ldots,  T / \tau\}$, \BBB $\tau \in (0, \tau_0)$, and $\mechen(\yst{k-1}) \leq M$ it holds that
  \begin{equation}\label{bad_mechen_bound}
  \begin{aligned}
    &\mechen(\yst k) + c_M  \tau \BBB \norm{  \ddif \BBB \nabla \yst k}_{L^2(\Omega)}^2 \\
    &\quad\leq (1+C_M\tau)\mechen(\yst {k-1})
      + C_M \tau \big(
          \Vert \tst{k-1} \wedge 1 \Vert^2_{L^2(\Omega)} \BBB + \CCC \lp^2 \BBB \norm{\fst k}_{L^2(\Omega)}^2
        + \CCC \lp^2 \BBB \norm{\gst k}_{L^2(\Gamma_N)}^2
      \big).
  \end{aligned} 
  \end{equation}
\end{lemma}

\begin{proof}
   Let $C_M$ as in \eqref{eq: safe for later}.   For $\tau_0$ sufficiently small with respect to $C_M$ we derive $\frac{1}{1 - C_M \tau} \le 1 + 2 C_M \tau$ \NNN for all $\tau \in (0,\tau_0)$. \BBB    \CCC Dividing \BBB \eqref{eq: safe for later} (for $\yst k$ in place of $y_n$) by $1 - C_M \tau$ \CCC \NNN we \BBB get the desired estimate,  up to changing the constant\CCC s \BBB $C_M$ and $c_M$.
\end{proof}

\begin{remark}\label{rem:bad_mechen_bound}
  \NNN By $1 \wedge s \le \sqrt{s}$ for $s \ge 0$, \eqref{inten_lipschitz_bounds}, \eqref{toten}, by \BBB  the definition below \eqref{forces_mech_step}, and by a standard application of H\"older's inequality, we deduce from \eqref{bad_mechen_bound} that
  \begin{equation*}
    \mechen(\yst k)  + c_M  \tau \BBB \norm{  \ddif \BBB \nabla \yst k}_{L^2(\Omega)}^2
    \leq \mechen(\yst {k-1})
      + C_M ( \NNN \tau  \toten(\yst {k-1}, \tst {k-1})
 \BBB         + \NNN\lp^2 \BBB\norm{f}_{L^2(I \times \Omega)}^2
        + \NNN \lp^2 \BBB \norm{g}_{L^2(I \times \Gamma_N)}^2
      ).
  \end{equation*}
  In fact, we have $\norm{\fst k}_{L^2(\Omega)}^2
    = \tau^{-2} \int_\Omega \big| \int_{(k-1)\tau}^{k\tau} f(t, x) \di t \big|^2 \di x
    \leq \tau^{-1} \int_0^T \norm{f(t)}_{L^2(\Omega)}^2 \di t$
  and a similar computation holds for $g$.
\end{remark}

In the next lemma we discuss the well-definedness of the thermal step.

\begin{proposition}[Thermal step]\label{prop:existence_thermal_step}
  For any $M > 0$ there exists $\tau_0 \in (0, 1]$ such that  if the minimizer given by Proposition \ref{prop:existence_mechanical_step} exists, $\tau \in (0, \tau_0)$, and \BBB $\mechen(\yst{k-1}) \leq M$ the minimization problem \eqref{thermal_step} is well-posed on $H^1_+(\Omega)$.
  More precisely,
  \begin{align*}
    \mathcal{T}(\theta)
    \defas& \int_\Omega \int_0^\theta \frac{1}{\tau} \big(
        \inten(\nabla \yst k, s)
        - \inten(\nabla \yst{k-1}, \tst{k-1})
      \big) \di s \di x
      + \frac{1}{2} \int_\Omega \nabla \theta \cdot
        \hcm(\nabla \yst{k-1}, \tst{k-1}) \nabla \theta \di x \\
    & -\int_\Omega h_\tau(\yst{k}, \yst{k-1}, \tst{k-1}) \theta \di x
      + \frac{\kappa}{2} \int_\Gamma (\theta - \CCC \lp^2 \BBB \btst k)^2 \di \haus^{d-1}
  \end{align*}
  is finite on $H^1(\Omega)$ and attains a unique minimizer $\tst k$ on $H^1_+(\Omega)$.
  Moreover, $\tst k$ satisfies 
  \begin{align}\label{el_thermal_step}
    &\int_\Omega \Bigg(
        \frac{\wst k - \wst{k-1}}{\tau}
        - \pl_F \cplpot(\nabla \yst{k-1}, \tst{k-1}) : \ddif \nabla \yst k
        - \drate(\nabla \yst{k-1}, \ddif \nabla \yst k, \tst{k-1})
      \Bigg) \vphi \di x \notag \\
    &\quad + \int_\Omega
        \hcm(\nabla \yst{k-1}, \tst{k-1}) \nabla \tst k \cdot \nabla \vphi \di x
      + \kappa \int_\Gamma (\tst k - \CCC \lp^2 \BBB \btst k) \vphi \di \haus^{d-1} = 0
  \end{align}
  for any $\vphi \in H^1(\Omega)$,
  where for brevity  $\wst{k-1} \defas \inten(\nabla \yst{k-1}, \tst{k-1})$ and $\wst k \defas \inten(\nabla \yst k, \tst k)$.
\end{proposition}

 Remarkably,  the nonnegativity constraint in the minimization problem \eqref{thermal_step} does \emph{not} influence the stationarity condition \eqref{el_thermal_step}. We also \BBB emphasize that in contrast to \cite{MielkeRoubicek20Thermoviscoelasticity} we can ensure uniqueness of the minimizer.
This is due to the fact that we use a simpler (explicit) thermo-mechanical coupling term in the scheme, see Remark \ref{rem:difference_to_mielke_in_steps} for details.

\begin{proof}
  \textit{Step 1 (Finiteness):}
  We start by showing that all terms of $\mathcal{T}$ are well-defined and integrable.
  First, by \eqref{inten_lipschitz_bounds} we find that
  \begin{equation}\label{ul}
    \int_0^\theta \inten(\nabla \yst k, s) \di s
    \in [\tfrac{\ac}{2} \theta^2, \tfrac{\aC}{2} \theta^2]
  \end{equation}
  and $\int_0^\theta \wst{k-1} \di s \le \aC \theta \tst{k-1}$   a.e.~on $\Omega$ \BBB which both lie in $L^1(\Omega)$ by H\"older's inequality.   By Lemma~\ref{lem:bound_hcm}, $\hcm(\nabla \yst{k-1}, \tst{k-1})$ is well-defined in $\Omega$, and the corresponding term \CCC in \BBB $\mathcal{T}$ is \NNN integrable. \BBB
  Finally, by \eqref{C_locally_lipschitz}, \ref{D_bounds}, \eqref{diss_rate}, and the second estimate in \eqref{pos_det} we get that the term $h_\tau$ defined in \eqref{def:h_tau} satisfies $h_\tau(\yst k, \yst {k-1}, \tst {k-1}) \in L^\infty(\Omega)$, i.e., the third term is also well-defined.
  This completes the proof of the well-definedness of $\tempen$.

  \textit{Step 2 (Existence):}
  The functional is coercive on $H^1_+(\Omega)$ due to $  \int_0^\theta \inten(\nabla \yst k, s) \di s \BBB \ge \frac{\ac}{2} \theta^2$ by \eqref{ul}, the estimate $\nabla \theta \cdot \hcm(\nabla \yst{k-1}, \tst{k-1}) \nabla \theta \ge c_M |\nabla \theta|^2$ by \eqref{bound_hcm}, \NNN and \BBB the fact that all other terms are either nonnegative or linear in $\theta$.
  Moreover, the functional is weakly lower semicontinuous on $H^1_+(\Omega)$.
  To see this, we again use \eqref{bound_hcm}, the weak continuity of the trace operator in $H^1(\Omega)$, and the fact that all other bulk terms are continuous in $L^2(\Omega)$ by the reasoning in Step 1.
  This shows that a minimizer \NNN  $\tst k$ \BBB exists.

  \textit{Step 3 (Euler-Lagrange equation):}
  In order to prove (\ref{el_thermal_step}) for test functions $\vphi \in H^1(\Omega)$ which are not constrained to be nonnegative, we extend the minimization problem \eqref{thermal_step} to possibly negative functions $\theta \in H^1(\Omega)$ and we show that $\tst k$ minimizes $\mathcal{T}$ on $H^1(\Omega)$.
  To this end, recalling that $\inten(F, 0) = 0$ for $F \in GL^+(d)$  (see below \eqref{Wint}), \BBB we continuously extend $\inten$ to negative temperatures by setting $\inten(F, \theta) = 0$ for $\theta < 0$.
  It now suffices to check that there exists a constant $c_M>0$ such that for all $\theta \in H^1(\Omega)$ it holds
  \begin{equation}\label{posiposi}
    \mathcal{T}(\theta)
    \ge \mathcal{T}(\theta^+)
      + \frac{c_M}{2} \Vert \nabla \theta^- \Vert^2_{L^2(\Omega)},
  \end{equation}
  where $\theta^- \defas \max\setof{-\theta, 0}$ and $\theta^+ \defas \max\setof{\theta, 0}$, i.e., $\theta = \theta^+ - \theta^-$.
  This \CCC guarantees \BBB that minimizers of  $\mathcal{T}$ are nonnegative, and because \MMM $\mathcal{T}$ is strictly  convex \NNN (to see this, use \eqref{sec_deriv}),  \BBB  $\tst k$  is its unique minimizer  on $H^1(\Omega)$.
  Once this is achieved,  in view of \NNN \eqref{def:h_tau} and \BBB \eqref{ul}, \BBB by taking first variations it is a standard matter to check that \eqref{el_thermal_step} holds true.

  Hence, it remains to prove \eqref{posiposi}.
  First, as $\btst k \geq 0$ $\haus^{d-1}$-a.e.~on $\Gamma$, \BBB we find
  \begin{equation}\label{posiposi1}
    \int_\Gamma (\theta - \CCC \lp^2 \BBB \btst k)^2 \di \haus^{d-1}
    \geq \int_\Gamma (\theta^+ - \CCC \lp^2 \BBB \btst k)^2 \di \haus^{d-1}.
  \end{equation}
  Next, by using (\ref{bound_hcm}) we see that
  \begin{align}\label{posiposi2}
    \frac{1}{2} \int_\Omega \nabla \theta \cdot \CCC \Kst{k-1} \BBB \nabla \theta \di x
    &= \frac{1}{2} \int_\Omega
      \nabla \theta^+ \cdot \CCC \Kst{k-1} \BBB \nabla \theta^+ \di x
      + \frac{1}{2} \int_\Omega
        \nabla \theta^- \cdot \CCC \Kst{k-1} \BBB\nabla \theta^- \di x \notag \\
    &\geq \frac{1}{2} \int_\Omega
      \nabla \theta^+ \cdot \CCC \Kst{k-1} \BBB \nabla \theta^+ \di x
      + \frac{c_M}{2} \int_\Omega \abs{\nabla \theta^-}^2 \di x,
  \end{align}
  where for brevity we have set $\CCC \Kst{k-1} \BBB \defas \hcm(\nabla \yst {k-1}, \tst {k-1})$.
  Moreover, for a.e.~$x \in \Omega$ we have
  \begin{equation}\label{posiposi3}
    \int_0^{\theta(x)} \inten(\nabla \yst{k},s) \di s
    \ge \int_0^{\theta^+(x)} \inten(\nabla \yst{k},s) \di s.
  \end{equation}
  This follows from $\inten(F, s) = 0$ for all $(F, s) \in  GL^+(d) \BBB \times (-\infty, 0)$.  Eventually, we consider the terms involving $h_\tau$ and $\wst{k-1}$.
  At this point, our argument for proving nonnegativity of the temperature is more delicate compared to \cite{MielkeRoubicek20Thermoviscoelasticity} as we use the backward approximation $\tst{k-1}$, see Remark \ref{rem:difference_to_mielke_in_steps}.  By \ref{C_frame_indifference} there exists a function $\hat{W}^{\rm cpl}$ such that ${\cplpot}(F,\MMM\theta\BBB) = \hat{W}^{\rm cpl}(C,\MMM\theta\BBB)$ with $C = F^TF$. Clearly, $\partial_C \hat{W}^{\rm cpl}$ is symmetric which implies \CCC with  the chain rule \BBB \MMM that \BBB
\begin{align}\label{eq: strange deriv}  
  \partial_F {\cplpot}(F,\MMM\theta \BBB) =  \CCC F \big(\partial_C\hat{W}^{\rm cpl}(C,\MMM\theta \BBB) + (\partial_C\hat{W}^{\rm cpl}(C,\MMM\theta \BBB))^T\big) \BBB  = 2F\partial_C\hat{W}^{\rm cpl}(C, \MMM\theta
	\BBB).
  \end{align} \BBB  
  By Lemma \ref{lem:pos_det}, $\nabla \yst{k-1}$ is invertible at every point in $\Omega$.  
  Hence, setting $C^{(k-1)}_{\CCC \lp \BBB, \tau} \defas (\nabla \yst{k-1})^T \nabla \yst{k-1}$, we derive by the second and third bound in \eqref{pos_det},   \eqref{C_locally_lipschitz},   \eqref{eq: strange deriv}, \BBB and the fact that $t \wedge 1 \leq \sqrt{t}$ for all $t \geq 0$  that \BBB
  \begin{align}
    \big|\pl_C \hat W^{\rm cpl}(C^{(k-1)}_{\CCC \lp \BBB, \tau}, \tst{k-1})\big|
    &= \CCC \frac{1}{2} \BBB \Big|
      (
        \nabla \yst{k-1})^{-1}
        \pl_F \cplpot(\nabla \yst{k-1}, \tst{k-1}
      )\Big| \nonumber \\
    &\leq  2 \BBB\aC \CCC | \BBB(\nabla \yst{k-1})^{-1}\CCC | \BBB
      (\tst{k-1} \wedge 1)
      (1 + \CCC | \BBB \nabla \yst{k-1}\CCC | \BBB)
    \leq C_M \sqrt{\tst{k-1}} \label{partial_C_cplpot_bound}
  \end{align}
  \NNN for $C_M>0$ sufficiently large. \BBB   Let us further define $\dot C^{(k)}_{\CCC \lp \BBB, \tau} \defas (\ddif \nabla \yst k)^T \nabla \yst{k-1} + (\nabla \yst{k-1})^T \ddif \nabla \yst k$.   By  the symmetry of $\partial_C \hat W^{\rm{cpl}}$ we have for all $F \in GL^+(d)$, $G \in \R^{d \times d}$, and $\theta \geq 0$
  \begin{equation*}
    F \partial_C \hat W^{\rm{cpl}}(C, \theta) : G
    = \partial_C \hat W^{\rm{cpl}}(C, \theta) : F^T G
    = \partial_C \hat W^{\rm{cpl}}(C, \theta) : G^T F,
  \end{equation*}
  where, again, $C \defas F^T F$. \MMM We now use  this identity with  $F=\nabla \yst{k-1}$ and $G=\ddif \nabla \yst k$. \BBB
  \BBB By  \eqref{inten_lipschitz_bounds}, \eqref{eq: strange deriv}, \eqref{partial_C_cplpot_bound}, \BBB and Young's inequality \NNN with constant $\tau$ \BBB \CCC it follows that \BBB 
  \begin{align*}
    &\abs{\pl_F \cplpot (\nabla \yst{k\CCC -1\BBB}, \tst {k-1}) : \ddif \nabla \yst k}
    =  \CCC 2 \BBB \big|\CCC \nabla \yst{k-1} \BBB \pl_C \hat W^{\rm cpl}(C^{(k-1)}_{\CCC \lp \BBB, \tau}, \tst{k-1})
      : \ddif \nabla \yst k\big| \\
    &\quad= \abs{\pl_C \hat W^{\rm cpl}(C^{(k-1)}_{\CCC \lp \BBB, \tau}, \tst{k-1}) : \dot C^{(k)}_{\CCC \lp \BBB, \tau}}
    \leq C_M \sqrt{\wst{k-1}} \abs{\dot C^{(k)}_{\CCC \lp \BBB, \tau}} \leq \frac{\wst{k-1}}{\tau} + C_M^{\MMM 2 \BBB} \tau \abs{\dot C^{(k)}_{\CCC \lp \BBB, \tau}}^2.
  \end{align*}
  Choosing $\tau_0$ sufficiently small such that $C_M^{\MMM 2 \BBB} \tau_0 \leq \ac$, we derive by \ref{D_quadratic}, \ref{D_bounds},  and \eqref{diss_rate} \BBB for all $\tau \in (0, \tau_0)$ that
  \begin{align}\label{eq: drate for late}
    \pl_F\cplpot (\nabla \yst{k \CCC -1 \BBB}, \tst {k-1}) : \ddif \nabla \yst{k} \geq - \frac{\wst{k-1}}{\tau}
      - \drate (\nabla \yst {k-1}, \ddif \nabla \yst k , \tst {k-1}).
  \end{align}
This shows $\tau^{-1}\wst{k-1} + h_\tau  (\yst{k}, \yst{k-1}, \tst{k-1}) \BBB \geq 0$  a.e.~\CCC on \BBB $\Omega$.
  From this we deduce
  \begin{equation}
    - \BBB \int_\Omega \Big(
        \frac{\wst {k-1}}{\tau}
        + h_\tau(\yst{k}, \yst{k-1}, \tst{k-1})
      \Big) \, \theta \di x
    \ge - \int_\Omega \Big(
      \frac{\wst {k-1}}{\tau}
      + h_\tau(\yst{k}, \yst{k-1}, \tst{k-1})
    \Big) \, \theta^+ \di x. \label{posiposi4}
  \end{equation}
  Combining the estimates \eqref{posiposi1}--\eqref{posiposi4} leads to \eqref{posiposi} which concludes the proof.
\end{proof}

\begin{remark}[Nonnegativity of temperature without dissipation rate]\label{eq: temp-rem}
{\normalfont

To derive estimate \eqref{eq: drate for late},  \CCC it was essential that $\drate(F, \dot F, \theta) \geq c |\dot F^T F + F^T \dot F|^q$ for some $q > 1$ \BBB. The pointwise nonnegativity can still be established \MMM only \BBB under the  assumption that $\drate \ge 0$, at the expense of assuming that $\mechen(\yst{k-1}) \leq \eta$  and $\mechen(\yst{k}) \leq \eta$  for some $\eta$ sufficiently small, and \CCC that \BBB $W$ \CCC additionally \BBB satisfies \ref{W_lower_bound_spec}. Indeed, in this case we can show
 \begin{equation}\label{nonnegativity_cpl_thermal_step}
  \partial_F \cplpot(\nabla \yst{k-1}, \tst{k-1}) : \ddif \nabla \yst{k} \geq -    \frac{\wst{k-1}}{\tau} 
  \end{equation}
a.e.~in $\Omega$ which along with $\drate \ge 0$ implies \eqref{posiposi4}. To see this, 
  by  \eqref{inten_lipschitz_bounds}, \eqref{W1infty_dist_to_id}, and  \eqref{C_locally_lipschitz},  we can estimate
  \begin{align*}
    \abs{\partial_F \cplpot(\nabla \yst{k-1}, \tst{k-1}) : \ddif \nabla \yst{k}}
    &\leq 2C_0 \tst{k-1} (1 + \abs{\nabla \yst{k-1}}) \abs{\ddif \nabla \yst{k}} \\
    &\leq \frac{2C_0}{c_0} \wst{k-1} (1 + \NNN |\Id| + \BBB C^*_\eta)
 \frac{2C^*_\eta}{\tau}. 
  \end{align*}
  Since \NNN $C^*_\eta \to 0$ \BBB as $\eta \to 0$, \eqref{nonnegativity_cpl_thermal_step} indeed follows for $\eta$ small enough.  \NNN This property will be exploited in the adaptions to the case $\alpha <2$ in Subsection \ref{sec: adaptions} below. \BBB
}
\end{remark}

  \BBB

For any  $\yst k$ and $\tst k$  as given in this subsection, \BBB we define from now on 
$$\wst k \defas \inten(\nabla \yst k, \tst k).$$

\subsection{Well-definedness of the scheme}\label{sec: welldef}
For fixed time \MMM  horizon \BBB $T > 0$ and time step $\tau \in (0, 1]$ small enough, we will now prove the well-definedness of the staggered time-discretization scheme described in the previous subsection. \NNN In this part, we are  interested in the large-strain setting, and treat the case  $\eps=1$ and  $\alpha=2$, where $\xi$ is not regularized. For later purposes, we again include $\eps$ in the estimates.  (The reader only interested in large strains, can readily set $\eps=1$.) \BBB As  before, we assume for the sake of simplicity that $T / \tau$ is an integer. \NNN Although not being necessary, for convenience we suppose that \ref{W_lower_bound_spec} holds. At the end of the subsection, we briefly indicate the changes  if \ref{W_lower_bound_spec} is not assumed. \BBB

We start with a bound on the total energy $\toten$ defined in \eqref{toten}.
 We \BBB also need to take the work of the external forces into account.
To this end, similar to the notation in \eqref{forces_mech_step}, we consider \CCC for each $t \in I$ \BBB the functionals $\ell(t)$ on $H^1(\Omega; \R^d)$ defined by
\begin{equation}\label{ell2}
  \langle \ell(t), v \rangle
  \defas \int_\Omega f(t) \cdot v \di x
    + \int_{\Gamma_N} g(t) \cdot v \di \haus^{d-1}
\end{equation}
for all  $v \in H^1(\Omega; \R^d)$.
Furthermore, we define
\begin{equation}\label{def_Cfg}
  C_{f,g}
  \defas \norm{f}_{W^{1,1}(I; L^2(\Omega))} + \norm{g}_{W^{1,1}(I;L^2(\Gamma_N))}.
\end{equation}
Note that  the trace estimate in $H^1(\Omega; \R^d)$ \BBB shows
\begin{equation*}
  \Vert \ell(t) \Vert_{H^{-1}} \le C \big(
      \Vert f(t) \Vert_{L^2(\Omega)} + \Vert g(t) \Vert_{L^2(\Gamma_N)}
    \big),
\end{equation*}
and, hence, by the   Fundamental Theorem of Calculus in $W^{1,1}(I; L^2(\Omega) \BBB)$ \BBB and \CCC $W^{1,1}(I; L^2(\Gamma_N))$ \BBB we get 
\begin{equation}\label{H-1_est}
  \Vert \ell(t) \Vert_{H^{-1}} \le C_T C_{f,g}
\end{equation}
for a constant $C_T$ only depending on $T$.
Given the sequences $\yst 0, \ldots,  \yst k$ and $\tst 0, \ldots\CCC, \BBB \tst k$ for some $k \in \setof{1, \ldots,  T / \tau}$, as described in Subsection \ref{sec:single_step_well_defined}, we define for $l \in \setof{0, \ldots, k}$
\begin{equation}\label{Fk}
  \mathcal{F}^{(l)}
  \defas \toten(\yst l, \tst l)
    - \NNN \lp \BBB \langle \ell(l \tau), \yst l \NNN - \id \BBB \rangle,
\end{equation}
and observe the following relation between $\mathcal{F}^{(l)}$ and the total energy \NNN $\toten(\yst l, \tst l)$. \BBB

\begin{lemma}\label{lem:EF}
  There exists a constant $C_T > 0$ only depending on $T$ such that for all $l \CCC \in \{ \BBB 0, \ldots,  k \CCC\} \BBB$ with $k \in \setof{1, \ldots,  T / \tau}$ it holds that
  \begin{equation*}
    \CCC \lp \BBB |\langle \ell(l\tau), \yst l  \NNN - \id \BBB  \rangle|
    \le \min \lbrace \mathcal{F}^{(l)}, \toten(\yst l, \tst l) \rbrace
      + \CCC \lp^2 C_T C_{f,g}^2 \BBB,
  \end{equation*}
  with $C_{f,g}$ as defined in \eqref{def_Cfg}.
\end{lemma}

\begin{proof}
  By $\yst l \in \Wid$, Poincaré's inequality, \CCC \eqref{H1_dist_to_id}, and \ref{W_lower_bound_spec} \BBB we derive 
  \begin{equation*}
    \norm{\yst l  \NNN - \id \BBB }_{H^1 (\Omega)}^2
    \leq C \norm{\nabla \yst l  \NNN - \Id \BBB }_{L^2(\Omega)}^2
    \leq C \elen(\yst l).
  \end{equation*}
  Hence, by \eqref{H-1_est} and Young's inequality  with \BBB constant \NNN $\lambda/\lp$ \BBB (to be chosen below) it follows that
  \begin{align*}
    |\langle \ell(l\tau), \yst l  \NNN - \id \BBB  \rangle|
    &\le \Vert \ell(l\tau) \Vert_{H^{-1}} \Vert \yst l  \NNN - \id \BBB \Vert_{H^1(\Omega)} \\
    &\le C_T C_{f,g} \Vert \yst l \NNN - \id \BBB  \Vert_{H^1(\Omega)} \\
    &\le \frac{C_T\lp}{\lambda} C_{f,g}^2
      + \NNN \frac{ \lambda}{\lp} \BBB \Vert \yst l  \NNN - \id \BBB  \Vert_{H^1(\Omega)}^2
    \le \frac{C_T\lp}{\lambda} C_{f,g}^2
      + C  \NNN\frac{\lambda}{\lp} \BBB \toten(\yst l, \tst l).
  \end{align*}
  Now, take $\lambda$ small enough such that $C\lambda \le \frac{1}{2}$.
  Then, by the definition of $\mathcal{F}^{(l)}$ we discover
  \begin{equation*}
    \mathcal{F}^{(l)}
    = \toten(\yst l, \tst l) - \NNN \lp \BBB \sprod{\ell(l\tau)}{\yst l \NNN - \id \BBB }
    \geq \frac{1}{2} \toten(\yst l, \tst l) - \CCC \lp^2 C_T C_{f, g}^2 \BBB,
  \end{equation*}
  and the statement follows.
\end{proof}

We now proceed with the bound on the total energy.
For definiteness, we   \CCC set \BBB $\ell(t) = 0$ for $t \notin I$.
\begin{lemma}[Inductive bound on the total energy]\label{lem:initial_toten_bound}
  For any $M > 0$ there exist $C_M$ such that, if the sequences $\yst 0, \ldots,  \yst k$ and $\tst 0, \ldots,  \tst k$, as described in Subsection \ref{sec:single_step_well_defined}, for some $k \in \setof{1, \ldots, T / \tau}$ exist satisfying $\mathcal{F}^{(l)} \leq M$ for all $l = 0, \ldots,  k-1$ with $\mathcal{F}^{(l)}$ defined in \eqref{Fk},  it holds that \BBB
  \begin{align}\label{initial_toten_bound}
    \mathcal{F}^{(k)}
    &\leq  \mathcal{F}^{(0)} \BBB + C_M \tau V_k +  \CCC \lp^2 C_T  \NNN (1 + C_{f,g}^3) \BBB
      + C\sum_{l=0}^k
        \mathcal{F}^{(l)} \int_{(l-1)\tau}^{l\tau} \big(
            \Vert \dot{\ell}(t) \Vert_{H^{-1}}
            + \Vert \dot{\ell}(t+\tau) \Vert_{H^{-1}} \big) \di t \notag \\
    &\phantom{\leq}\quad + \kappa \NNN \lp^2 \BBB\int_0^{k \tau} \int_\Gamma \theta_\flat \di \haus^{d-1} \di t,
  \end{align}
  where $C$ is a universal constant, $C_T$ a constant only depending on $T$, and
  \begin{equation}\label{def_Vk}
    V_k \defas \sum_{l = 1}^k \tau \int_\Omega \abs{\ddif \nabla \yst l}^2 \di x.
  \end{equation}
\end{lemma}

\begin{proof}
  \textit{Step 1:}
  Let us fix $l \in \setof{1, \ldots,  k}$.
  Using Proposition \ref{prop:existence_mechanical_step} for $l$ in place of $k$,   \eqref{diss_rate}, and   testing \NNN \eqref{mechanical_step_single} \BBB with $z = \ddif \yst l$ \CCC it follows \BBB that
  \begin{align}\label{el_mech_test_step_l}
    0
    =& \int_\Omega
      \pl_F \felpot(\nabla \yst l, \tst{l-1}) : \ddif \nabla \yst l
      + \pl_G \hypot(\nabla^2 \yst l) \cdddot \ddif \nabla^2 \yst l \di x \notag \\
    &\quad + \int_\Omega
      \drate(\nabla \yst{l-1}, \ddif \nabla \yst l, \tst{l-1}) \di x
      - \NNN \lp \BBB\langle \lst l, \ddif \yst l \rangle.
  \end{align}
  Similarly, using Proposition \ref{prop:existence_thermal_step} for $l$ in place of $k$  we test \eqref{el_thermal_step} with $\vphi = 1$ to obtain
  \begin{align}\label{el_temp_test_step_l}
    0
    =& \int_\Omega
      \ddif \wst l
      -\pl_F \cplpot(\nabla \yst {l-1}, \tst{l-1})
        : \ddif \nabla \yst l \notag \\
    & -\int_\Omega \drate(\nabla \yst{l-1}, \ddif \nabla\yst l, \tst{l-1}) \di x
      + \NNN \kappa \BBB  \int_\Gamma (\tst l - \NNN \lp^2 \BBB\btst l )\di \haus^{d-1}.
  \end{align}
  Adding \eqref{el_mech_test_step_l} to \eqref{el_temp_test_step_l}, multiplying by $\tau$, and eventually summing over $l = 1, \ldots,  k$ we discover that
  \begin{align}\label{test_all_steps}
    \int_\Omega \NNN w_{0,\eps} \BBB \di x
    & =
    \tau \sum_{l = 1}^k \Big(
      \int_\Omega \pl_F \elpot(\nabla \yst l) : \ddif\nabla \yst l \di x
      + \int_\Omega \pl_G \hypot(\nabla^2 \yst l) \cdddot \ddif \nabla^2 \yst l\di x
    \Big) \notag \\
    &\phantom{=}\quad +  \tau \BBB \sum_{l = 1}^k \int_\Omega
      \big(
        \pl_F \cplpot(\nabla \yst l, \tst{l-1})
        - \pl_F \cplpot(\nabla \yst {l-1}, \tst{l-1})
      \big)
      : \ddif \nabla \yst l \di x \notag \\
    &\phantom{=}\quad -\sum_{l = 1}^k \Big(
      \tau \kappa \int_\Gamma (\CCC \lp^2 \BBB \btst l -  \tst l) \di \haus^{d-1}
      + \tau \NNN \lp \BBB \langle \lst l, \ddif \yst l \rangle
    \Big) + \int_\Omega \wst k \di x,
  \end{align}
  where $\NNN w_{0,\eps} \BBB \defas \inten(\nabla y_{0,\eps}, \theta_{0,\eps})$.
  Here, we also used that $\felpot = \elpot + \cplpot$.

  \textit{Step 2:}
  We continue by bounding the first two sums on the right-hand side of (\ref{test_all_steps}) from below.
  By the convexity of $\hypot$ (see \ref{H_regularity}) it follows for \CCC $l \in \{ 1, \ldots,  k \}$ \BBB that
  \begin{equation*}
    H(\nabla^2 \yst{l-1}) \geq
    H(\nabla^2 \yst l)
      + \pl_G H(\nabla^2 \yst l) \cdddot (\nabla^2 \yst{l-1} - \nabla^2 \yst l).
  \end{equation*}
  Integrating the above inequality over $\Omega$ and summing over $l = 1, \ldots,  k$ leads to
  \begin{equation}\label{hypot_sum_lower_bound}
    \tau \sum_{l = 1}^k \int_\Omega
      \pl_G \hypot(\nabla^2 \yst l) \cdddot \ddif \nabla^2 \yst l \di x
    \geq \hyen(\yst k) - \hyen(y_{0,\eps}),
  \end{equation}
  where we recall the notation in \eqref{hyperelastic}.
  By using the piecewise affine function \CCC $\hat y_{\lp, \tau}$ \BBB introduced in \eqref{y_interpolations},  and   that $\mathcal{W}^{\rm el}$ is Gateaux differentiable (see \cite[Proposition 3.2]{MielkeRoubicek20Thermoviscoelasticity}) \NNN we get that \BBB
  \begin{align}\label{elpto_sum_lower_bound_interpolation}
    &\sum_{l = 1}^k \int_{(l-1)\tau}^{l\tau} \int_\Omega
      \pl_F \elpot(\nabla \CCC \ay\BBB(t)\BBB ) : \ddif \nabla \yst l \di x \di t \nonumber \\
    &\quad = \int_0^{k\tau} \int_\Omega \pl_F \elpot(\nabla \CCC \ay(t)\BBB ) : \nabla \CCC\dotay\BBB(t) \di x \di t
    = \elen(\yst k) - \elen(y_{0,\eps}).
  \end{align}
  For $\tau_0$ sufficiently small, we can apply Lemma \ref{lem:bad_mechen_bound} in the version of Remark \ref{rem:bad_mechen_bound}.  This along with \BBB Lemma~\ref{lem:EF}, $\mathcal{F}^{(l)}\leq M$ for $l \CCC\in \{ 0, \ldots,  k - 1 \}\BBB$,  and \eqref{def_Cfg} implies \BBB for all $l \CCC\in \{ 1, \ldots, k \}\BBB$ that 
  \begin{align*}
    \mechen(\yst l)
    &\leq \mechen(\yst {l-1})
      +  C_M \BBB \big(
        \NNN   \toten(\yst {l-1}, \tst {l-1}) \BBB  + \NNN \eps^2 \BBB \norm{f}_{L^2(I \times \Omega)}^2
        + \NNN \eps^2 \BBB \norm{g}_{L^2(I \times \Gamma_N)}^2
    \big) \\
    &\le 2\NNN (1+C_M) \BBB \mathcal{F}^{(l-1)} + C_M  \NNN \lp^2 \BBB C_T C_{f,g}^2
    \le 2 \NNN (1+C_M) \BBB M + C_M  \NNN \lp^2 \BBB C_T C_{f,g}^2.
  \end{align*}
 Together with Lemma \ref{lem:pos_det} we get that there exists a compact convex set $K$, only depending on $M$,  $T$, \BBB $f$, and $g$, such that $\nabla \yst l \in K$ a.e.~on $\Omega$ for all $l \CCC\in \{ 0, \ldots,  k \}\BBB$.
  Then, by the regularity of $\elpot$, setting $C_M \defas \sup_{F \in K} \abs{\pl_{FF} \elpot(F)}$, we can estimate for any $t \in [(l-1)\tau, l\tau]$ with $l \CCC \in \{1, \ldots,  k\}\BBB$ that
  \begin{align*}
    &\abs{\pl_F \elpot(\nabla  \CCC\ay\BBB(t)\BBB ) - \pl_F \elpot(\nabla \yst l)} \\
    &\quad\leq C_M \abs{\nabla  \CCC\ay\BBB(t) \BBB - \nabla \yst l} = C_M \frac{l\tau - t}{\tau} \abs{\nabla \yst l - \nabla \yst{l-1}}
    \leq C_M \abs{\nabla \yst l - \nabla \yst{l-1}}.
  \end{align*}
  Consequently, we get
  \begin{align*}
    &\sum_{l = 1}^k \, \Bigg|
      \int_{(l-1)\tau}^{l\tau}
        \int_\Omega
          \pl_F \elpot(\nabla \CCC\ay\BBB(t)) : \ddif \nabla \yst l \di x \di t
        -\tau \int_\Omega
          \pl_F \elpot(\nabla \yst l) : \ddif \nabla \yst l \di x
      \Bigg| \\
    &\quad \leq C_M \sum_{l = 1}^k \tau \int_\Omega
      \abs{\nabla \yst l - \nabla \yst{l-1}} \abs{\ddif \nabla \yst l} \di x
    = C_M \tau \sum_{l = 1}^k \tau \int_\Omega
      \abs{\ddif \nabla \yst l}^2 \di x
    = C_M \tau V_k.
  \end{align*}
  Combined with \eqref{elpto_sum_lower_bound_interpolation} this leads to
  \begin{equation}\label{elpot_sum_lower_bound}
    \tau \BBB \sum_{l = 1}^k \int_\Omega
      \pl_F \elpot(\nabla \yst l) : \ddif\nabla \yst l \di x
    \geq \elen(\yst k) - \elen(y_{0,\eps}) - C_M \tau V_k.
  \end{equation}
  In \NNN a \BBB similar fashion, using the first bound in \ref{C_bounds}, we can estimate
  \begin{equation}\label{cpldiff_bound}
    \tau \sum_{l = 1}^k \int_\Omega
        \big(
          \pl_F \cplpot(\nabla \yst l, \tst{l-1})
          - \pl_F \cplpot(\nabla \yst {l-1}, \tst{l-1})
        \big) : \ddif \nabla \yst l \di x
      \geq - \aC \tau V_k.
  \end{equation}
  Now, employing \eqref{hypot_sum_lower_bound}, \eqref{elpot_sum_lower_bound}, and \eqref{cpldiff_bound} in \eqref{test_all_steps}, and using the definition of the total energy $\toten$ we conclude that
  \begin{equation}\label{test_all_steps_v2}
    \toten(\yst k, \tst k)
    \leq \toten(y_{0,\eps}, \theta_{0,\eps})
      + C_M \tau V_k
      + \sum_{l = 1}^k \tau \NNN \lp \BBB \langle \lst l, \ddif \yst l \rangle
      + \sum_{l = 1}^k \tau\kappa \int_\Gamma  (\NNN \lp^2 \BBB  \btst l - \tst l )\di \haus^{d-1}.
  \end{equation}

  \textit{Step 3:}
  It remains to estimate the last two terms on the right-hand side of (\ref{test_all_steps_v2}).
  By the nonnegativity of $\tst l$ and the definition of $\btst l$  below \eqref{def:h_tau} \BBB we can  \CCC bound \BBB
  \begin{equation}\label{b_surface}
    \sum_{l = 1}^k \tau\kappa \int_\Gamma ( \NNN \lp^2 \BBB \btst l - \tst l) \di \haus^{d-1}
    \leq \sum_{l = 1}^k \tau \kappa \NNN \lp^2 \BBB  \int_\Gamma \btst l \di \haus^{d-1}
    = \kappa \NNN \lp^2 \BBB  \int_0^{k \tau} \int_\Gamma \theta_\flat \di \haus^{d-1} \di t.
  \end{equation}
  Note that for any $l \CCC\in \{1, \ldots,  k \}$ and $t \in \CCC ( \BBB (l-1)\tau, l\tau\CCC ) \BBB$ we have that $\ddif   \yst l \BBB = \CCC\dotay\BBB(t)$.
  Consequently,  integration by parts yields
  \begin{align}\label{rufanew}
  &\sum_{l = 1}^k  \tau \langle \lst l, \ddif \yst l \rangle
  = \int_0^{k\tau} \hspace{-0.1cm}\langle \ell(t), \CCC\dotay\BBB(t) \rangle \di t  \\
  &\quad= \langle \ell (k\tau), \CCC\ay\BBB (k \tau)  \NNN - \id \BBB \rangle
    - \langle \ell (0), \CCC\ay\BBB(0) \NNN - \id \BBB \rangle
    - \int_0^{k\tau} \hspace{-0.1cm}\langle \dot{\ell} (t), \CCC\ay\BBB(t)  \NNN - \id \BBB \rangle \di t \nonumber \\
  &\quad\le \langle \ell (k\tau), \CCC\ay\BBB(k \tau) \NNN - \id \BBB \rangle
    - \langle \ell (0), \CCC\ay\BBB(0) \NNN - \id \BBB \rangle
    + \int_0^{k\tau} \Vert \dot{\ell}(t) \Vert_{H^{-1}}\Vert \CCC\ay\BBB(t) \NNN - \id \BBB \Vert_{H^{1}( \Omega)} \di t.
  \end{align}
  By Poincaré's inequality\CCC, \eqref{H1_dist_to_id}, and \ref{W_lower_bound_spec}, \BBB for $t \in [(l-1)\tau, l\tau]$, we have
  \begin{equation*}
    \Vert  \CCC\ay\BBB(t) \NNN - \id \BBB \Vert^2_{H^1(\Omega)}
    \le 2(
      \Vert \yst {l-1}  \NNN - \id \BBB \Vert^2_{H^1(\Omega)}
      + \Vert \yst l   \NNN - \id \BBB\Vert^2_{H^1(\Omega)}
    )
    \le C \big(\elen(\yst {l-1}) + \elen(\yst l) \big).
  \end{equation*}
  Therefore, by Lemma \ref{lem:EF},  \eqref{H-1_est}, \BBB  and $\sqrt{s} \le s \NNN /\lp \BBB$ for all \NNN $s \ge \lp^2$ \BBB we get
  \begin{align*}
    &\int_0^{k\tau} \hspace{-0.1cm}
      \Vert \dot{\ell}(t) \Vert_{H^{-1}}
      \Vert \CCC\ay\BBB(t) - \NNN \id \BBB \Vert_{H^{1}(\Omega)} \di t \\
    &\quad \le C \sum_{l=1}^k \Big(
        \NNN \lp + \lp^{-1} \BBB \toten(\yst {l-1}, \tst {l-1})
        + \NNN  \lp^{-1} \BBB \toten(\yst l, \tst l)
      \Big) \int_{(l-1)\tau}^{l\tau} \hspace{-0.1cm}
      \Vert \dot{\ell}(t) \Vert_{H^{-1}} \di t \\
    &\quad \le  \NNN \frac{C}{\lp} \BBB \sum_{l=1}^k \Big(
      \big( \mathcal{F}^{(l-1)} + \mathcal{F}^{(l)} \big)
      \int_{(l-1)\tau}^{l\tau} \Vert \dot{\ell}(t) \Vert_{H^{-1}} \di t
    \Big) + \CCC \lp C_T \NNN (C_{f,g}  + C_{f,g}^3). \BBB
  \end{align*}
  Then, using  an index shift \CCC and $C_{f, g} \leq \frac{2}{3} + \frac{1}{3}C_{f,g}^3$\BBB we get
  \begin{align}
     &\int_0^{k\tau} \hspace{-0.1cm} \Vert \dot{\ell}(t) \Vert_{H^{-1}}
      \Vert \hat{y}_\tau(t) \NNN - \id \BBB \Vert_{H^{1}(\Omega)} \BBB \di t \\  
      &\quad\le  \frac{C}{\lp} \sum_{l=0}^k \Big(
        \mathcal{F}^{(l)} \int_{(l-1)\tau}^{l\tau}\hspace{-0.1cm}
          \big(
            \Vert \dot{\ell}(t) \Vert_{H^{-1}}
            + \Vert \dot{\ell}(t+\tau) \Vert_{H^{-1}}
          \big) \di t
      \Big) + \CCC  \lp C_T \NNN (1  + C_{f,g}^3) \BBB \label{rufanew2}
  \end{align}
  \NNN for a possibly larger $C_T>0$. \BBB   We plug this into \eqref{rufanew} and use \eqref{b_surface} to estimate the terms on the right-hand side of (\ref{test_all_steps_v2}), which  by   \eqref{Fk} \BBB concludes the proof.
\end{proof}

We proceed with a bound on the (discrete) strain rates $V_k$ defined in (\ref{def_Vk}).  

\begin{lemma}[Inductive bound on the strain rates]\label{lem:Vk_bound}
  Given $M, \, T > 0$, there exist a constant $C_M$ and $\tau_0 \in (0, 1]$ only depending on $M$, and a constant $C_T$ only depending on $T$ such that for $\tau \in (0, \tau_0)$ the following holds: Suppose that there exist
  the sequences $\yst 0, \ldots,  \yst k$ and $\tst 0, \ldots,  \tst k$ for some $k \in \CCC \{1, \ldots,  T/\tau\}\BBB$, as described in Subsection \ref{sec:single_step_well_defined}, with $ \mechen(\yst l) \BBB \leq M$ for all $l   \CCC \in \{ \BBB 0, \ldots,  k-1 \CCC \} \BBB$.
  Then,
  \begin{equation}\label{Vk_bound}
    \sum_{l = 1}^k \tau \int_\Omega \abs{\ddif \nabla \yst l}^2 \di x
    \le  \NNN C_M \BBB  \mechen(y_{0,\eps})  + \NNN \lp^2 \BBB C_M C_T C_{f, g}^2  + \NNN C_M \tau \sum_{l = 0}^{k-1} \big(  \mechen(\yst l) +  \Vert \tst{l} \wedge 1 \Vert^2_{L^2(\Omega)}\big). \BBB
  \end{equation}
\end{lemma}

\begin{proof}
   By Lemma \ref{lem:bad_mechen_bound}  there exist constants $c_M, \, C_M > 0$ depending on $M$ such that we have for $l \CCC\in \{1, \ldots,  k \}\BBB$
  \begin{align*}
    &\mechen(\yst l) 
      + c_M \tau \int_\Omega \abs{\ddif \nabla \yst l}^2 \di x \\
    &\quad\le \NNN (1+ C_M\tau) \mechen(\yst{l-1}) + \BBB C_M \tau \big( \NNN  \Vert \tst{l} \wedge 1 \Vert^2_{L^2(\Omega)} \BBB  + \NNN \lp^2 \BBB \norm{\fst l}_{L^2(\Omega)}^2
      + \NNN \lp^2 \BBB \norm{\gst l}_{L^2(\Gamma_N)}^2
    \big).
  \end{align*}
  Summing the above inequality over $l = 1, \ldots,  k$ and recalling the definition of $\fst l$, $\gst l$ below \eqref{forces_mech_step}, we arrive at
  \begin{align*}
    \mechen(\yst k) - \mechen(y_{0,\eps})
    + c_M \sum_{l = 1}^k \tau \int_\Omega \abs{\ddif \nabla \yst l}^2 \di x
    &= C_M \NNN \lp^2 \BBB \int_0^{k\tau} \big(\norm{f(t)}_{L^2(\Omega)}^2 + \norm{g(t)}_{L^2(\Gamma_N)}^2 \big) \di t  \\
    &\phantom{=}\quad +  C_M\tau \NNN \sum_{l = 0}^{k-1} \big(  \mechen(\yst l) +  \Vert \tst{l} \wedge 1 \Vert^2_{L^2(\Omega)}\big). \BBB
  \end{align*}
  As $\mechen(\yst k) \ge 0$,   we conclude the proof by \NNN \eqref{H-1_est}. \BBB  
\end{proof}

We are ready to prove the well-definedness of our time-discretization scheme,  i.e., Theorem \ref{thm:van_tau}(i). \BBB 
At the same time, we will also derive two useful a priori bounds, namely on the total energy and on the (discrete) strain rate, respectively.

\begin{theorem}[Well-definedness of the scheme]\label{thm:apriori_toten_velo_bound}
  For any $T > 0$ there  exist  a constant \BBB $\bar{C}_T>0$,  corresponding constants
  \begin{equation}\label{def_tildeE0}
    M' \defas
    2 e^{ \bar{C}_T C_{f,g}} \Big(
      \NNN \eps^{-2} \BBB \mathcal{F}^{(0)}
      +  \bar{C}_T \NNN (1+ C_{f,g}^3) \BBB
      + \kappa   \int_0^T \int_\Gamma \theta_\flat \di \haus^{d-1} \di t
    \Big), \quad  M\defas 2M' + \NNN   \bar{C}_T C_{f,g}^2, \BBB
  \end{equation}
  as well as constants $C_M>0$ and $\tau_0 \in (0, 1]$ depending on $M$   
   such that  the following holds true: \BBB  For each $\tau \in (0, \tau_0)$ such that $T / \tau \in \N$ the sequences $\yst 0, \ldots\CCC, \BBB \yst{T / \tau}$ and $\tst 0, \ldots,  \tst{T / \tau}$ as described in Subsection \ref{sec:single_step_well_defined} exist,  and \BBB for all $k \CCC \in \{ \BBB 0, \ldots,   T/\tau \CCC \} \BBB$ we  have that 
  \begin{equation}\label{mechen_apriori_bound}
    \toten(\yst k, \tst k)\leq \NNN \eps^2 \BBB M,
  \end{equation}
  \begin{equation}\label{velocity_apriori_bound}
    \NNN \sum_{l=1}^{k} \BBB \tau \int_\Omega \abs{\ddif \nabla \yst l}^2 \di x
    \leq \NNN \eps^2 C_M M(1+T) \BBB  +  \NNN \lp^2 C_M \bar{C}_T \BBB C_{f,g}^2.
  \end{equation}
\end{theorem}

\begin{proof}
  \textit{Step 1:}  Let \BBB $C_T$ be the maximum of the constants $C_T$ from Lemmas \ref{lem:EF}, \ref{lem:initial_toten_bound}, \ref{lem:Vk_bound}, and equation \eqref{H-1_est}, and let $C$ be  the universal constant of \BBB Lemma \ref{lem:initial_toten_bound}.   Define $\bar{C}_T = \NNN \max  \lbrace 2CTC_T, 2C_T , 2 \rbrace\BBB$, and let $M'$ and $M$ \CCC be \BBB as in \eqref{def_tildeE0}. Then, let \BBB  $C_M > 0$ be the maximum of the constants $C_M$ from Lemmas \ref{lem:initial_toten_bound} and \ref{lem:Vk_bound}.   Moreover, let $\tau_0 \in (0, 1]$ be chosen sufficiently small so that  Lemma \ref{prop:existence_mechanical_step}, \BBB Lemma \ref{lem:bad_mechen_bound}, Proposition~\ref{prop:existence_thermal_step}, Lemma \ref{lem:initial_toten_bound}, and Lemma~\ref{lem:Vk_bound} hold  true (all applied for $M$ from \eqref{def_tildeE0}). \BBB   In place of \eqref{mechen_apriori_bound}, we focus on showing
  \begin{equation}\label{mechen_apriori_bound*}
    \mathcal{F}^{(k)}
    = \toten(\yst k, \tst k) - \NNN \lp \BBB \langle \ell(k\tau), \yst k \rangle
    \leq \NNN \eps^2 \BBB M', \BBB
  \end{equation}
  as then \eqref{mechen_apriori_bound} follows directly by Lemma \ref{lem:EF}.

  We will prove the statement by induction  over \BBB $K$.
  In the base case $K = 0$, \eqref{mechen_apriori_bound*} is satisfied by our choice of  $M'$, \BBB and the fact that $\yst 0 = y_{0, \CCC \lp \BBB}$ and $\tst 0 = \theta_{0, \CCC \lp \BBB}$.   Given $K \in \CCC \{1, \ldots,  T / \tau\} \BBB$, let us assume that the statement as well as \eqref{mechen_apriori_bound*} hold true for $K-1$.
  We now show that the statement holds true for $K$.
  Applying first Proposition \ref{prop:existence_mechanical_step} and then Proposition \ref{prop:existence_thermal_step} we see that $\yst k$ and $\tst k$ exist, where for  both propositions \BBB we use the induction hypothesis \eqref{mechen_apriori_bound} for $K-1$.

  \textit{Step 2:}
  In this step, we prove that for $\tau_0$ small enough we have that
  \begin{equation}\label{tau_Vk_bound}
    C_M \tau \NNN  V_K \BBB \leq \NNN \eps^2, \BBB
  \end{equation}
  where $V_K$ is defined in \eqref{def_Vk}.
  By  \NNN $\eps^2 M \le M$, \BBB  Remark~\ref{rem:bad_mechen_bound}, and \eqref{def_Cfg} there exists a constant  $\tilde{C}_M$ only depending on $M$  such that
  \begin{equation*}
    \frac{1}{\tilde{C}_M\tau} \norm{\nabla \yst K - \nabla \yst{K-1}}_{L^2(\Omega)}^2
    \leq \NNN \eps^2 \BBB  M  + \CCC\lp^2\BBB\tilde C_M( \NNN M \BBB + T C_T^2 C_{f, g}^2),
  \end{equation*}
where we  again \BBB used the hypothesis \eqref{mechen_apriori_bound} for $K-1$.
  Hence, by possibly further decreasing $\tau_0$ (depending  only \BBB on $M$, $f$, $g$,  $T$, \NNN and the initial values) \BBB we can ensure that
  \begin{equation*}
     C_M \BBB \norm{\nabla \yst K - \nabla \yst{K-1}}_{L^2(\Omega)}^2
    \leq \CCC\frac{\lp^2}{2}\BBB.
  \end{equation*}
  Furthermore, by possibly decreasing $\tau_0$  (depending only on $M$, \NNN $u_0$, $\mu_0$, \BBB $f$, $g$, and $T$) \BBB and using the hypothesis (\ref{velocity_apriori_bound}) for $K-1$ in place of $K$ we get $C_M \tau V_{K-1} \leq \NNN  \eps^2/2 \BBB $.
  Consequently, combining the previous estimates and using $\tau V_K = \tau V_{K - 1} + \norm{\nabla \yst K - \nabla \yst{K-1}}_{L^2(\Omega) }^2$, the desired bound (\ref{tau_Vk_bound}) follows.

  \textit{Step 3:}
  By   hypothesis the energy bound in \eqref{mechen_apriori_bound*} is satisfied for $k \CCC\in \{0, \ldots,  K-1\}\BBB$.
  Consequently, Lemma \ref{lem:initial_toten_bound} applies for any $k \CCC\in \{ 0, \ldots,  K\}\BBB$.
  By (\ref{tau_Vk_bound}) we have that
  \begin{equation*}
    \mathcal{F}^{(k)}
    \leq  \mathcal{F}^{(0)}
     \NNN + \eps^2
      +   \lp^2 C_T (1+ C_{f,g}^3) \BBB
      + C \sum_{l=0}^k \mathcal{F}^{(l)}
        \int_{(l-1)\tau}^{l\tau}
        (  \Vert \dot{\ell}(t) \Vert_{H^{-1}}
          + \Vert \dot{\ell}(t+\tau) \Vert_{H^{-1}})
        \di t
      + \kappa \CCC \lp^2 \BBB \int_0^T \int_\Gamma \theta_\flat \di \haus^{d-1} \di t.
  \end{equation*}
  We now use the following discrete version of Gronwall's Lemma: if $\beta > 0$, $(a_l)_l$ is a nonnegative sequence, $(b_l)_l \subset (0, \CCC 1/2\BBB)$, and
  \begin{equation*}
    a_k \le \beta + \sum_{l=0}^{k} b_l a_l \quad \text{ for } k \ge 0,
  \end{equation*}
  then
  \begin{equation*}
    a_k \le 2\beta \exp\Big( \sum_{l=0}^{k-1} 2b_l \Big) \quad \text{ for } k \ge 0.
  \end{equation*}
  Indeed, as $b_l \le 1/2$, we get $a_k \le 2\beta + \sum_{l=0}^{k-1} 2b_l a_l$, and then the statement follows from the elementary discrete Gronwall inequality.
  We apply this result for
  \begin{align}\label{eq: bbeta}
    \beta &\defas  \mathcal{F}^{(0)} \BBB \NNN + \eps^2
      +   \lp^2 C_T (1+ C_{f,g}^3) \BBB
      + \kappa \CCC \lp^2 \BBB \int_0^T \int_\Gamma \theta_\flat \di \haus^{d-1} \di t, \notag\\
    a_l &\defas \mathcal{F}^{(l)}, \qquad
    b_l \defas C \int_{(l-1)\tau}^{l\tau} \big(\Vert \dot{\ell}(t) \Vert_{H^{-1}} + \Vert \dot{\ell}(t + \tau) \Vert_{H^{-1}} \big) \di t,
  \end{align}
  where we note that $b_l \le 1/2$ for all $l$, provided that $\tau_0$ is chosen small enough depending on $f$ and $g$.
  In view of \eqref{H-1_est}, \BBB  we then see that (\ref{mechen_apriori_bound*}) for $K$ is true.  
  \NNN Finally, \BBB (\ref{velocity_apriori_bound}) directly follows from the application of Lemma \ref{lem:Vk_bound} \NNN and the fact that the last term in \eqref{Vk_bound} can be controlled by $\mathcal{E}$, see e.g.\ Remark \ref{rem:bad_mechen_bound}. \BBB 
\end{proof}

\NNN Eventually, if \ref{W_lower_bound_spec} is not assumed, we get additional additive constants in Lemma \ref{lem:EF} and in the derivation of \eqref{rufanew2}, leading to an additional constant in \eqref{initial_toten_bound} which however does not scale as $\eps^2$. This does influence the proof of the well-definedness, only the scaling of the energy in terms of $\eps$.    \BBB

\subsection{Adaptions for exponents $\alpha < 2$}\label{sec: adaptions}
In this subsection, we prove Proposition \ref{cor:van_tau_reg}(i). \NNN This part can \CCC be \BBB skipped by a reader only interested in the proof of Theorem \ref{thm:van_tau}. \BBB
In the previous subsection, we have  already established the well-definedness of the scheme in the large-strain setting, as well as the energy bound \eqref{mechen_apriori_bound}.
The latter will be essential to obtain a priori bounds for the limit passage $\tau \to 0$ in Section \ref{sec:tau_to_zero_delta_fixed}.
\NNN In the case $\alpha <2$, \BBB for the passage to the linearized setting $\lp \to 0$, however, the bound \eqref{mechen_apriori_bound} and the induced a priori bounds are not expedient. \NNN This is  \BBB due to the different scaling of the internal and mechanical energy, being of order $\lp^\alpha$ and $\lp^2$, respectively.
To this end, it is necessary to establish energy bounds for rescaled versions of the energy functionals from Subsection \ref{sec:setting}, namely $\mechen_\lp \defas \frac{1}{\lp^2} \mechen$, $\cplen_\lp \defas \frac{1}{\lp^2} \cplen$, and \NNN for $\alpha \in [1,2]$ \BBB
\begin{align}\label{eq: rescaled toten}
  \toten_\lp(y, \theta) \defas \mechen_\lp(y) + \frac{\alpha}{2\lp^2} \int_\Omega \inten(\nabla y, \theta)^{\frac{2}{\alpha}} \di x,
\end{align} 
where both `types of energy' are of the same order. Controlling this energy is more delicate compared to Proposition \ref{lem:initial_toten_bound}, as the mechanical and thermal equation \eqref{el_mech_test_step_l}--\eqref{el_temp_test_step_l} scale with different powers of $\lp$ and cannot  simply be added up. Therefore, novel ideas are required to control the contributions of $W^{\rm cpl}$ and $\xi$. To achieve this, higher integrability of $\inten$ in $L^{2/\alpha}$ is needed which can be guaranteed by using the regularization of $\rdrate$ introduced in \eqref{def_rdrate}. This in turn induces new challenges for the analysis of the time-discrete scheme since showing the nonnegativity  of the temperature in the thermal step, see Proposition \ref{prop:existence_thermal_step}, is more delicate. For this, it will be essential to assume that strains are small, i.e., we suppose that  the parameter $\lp \in (0, 1]$ is sufficiently small.

Note that for the entire subsection we can assume that $\alpha \in [1,2)$ since in the case $\alpha =2$ there is no regularization of the dissipation rate,  the existence of the scheme is already guaranteed by Theorem~\ref{thm:van_tau}(i), and also an energy bound for $ \toten_\lp$ follows already from \eqref{mechen_apriori_bound}. The mechanical step is not affected by the regularization, but   Proposition \ref{prop:existence_thermal_step} needs to be adapted.

\begin{proposition}[Thermal step with regularization]\label{prop: thermal with reg}
  For any $M > 0$ there exists $\lp_0 > 0$ such that if $\lp \in (0, \lp_0)$, if the minimizer $\yst{k}$ given in Proposition \ref{prop:existence_mechanical_step} exists, and if $\mechen_\lp(\yst{k-1}) \leq M$ and $\mechen_\lp(\yst{k}) \leq M$ the minimization problem \eqref{thermal_step_reg} attains a unique solution $\tst{k}$  satisfying \eqref{el_thermal_step} \CCC for all $\vphi \in H^1(\Omega)$  \BBB with  $\drate$ replaced by $\rdrate$ .
\end{proposition}

\begin{proof}
  As $\drate \ge \rdrate$, the existence and uniqueness of $\tst{k}$ follows by the same reasoning as in  Steps 1--2 of \BBB the proof of Proposition~\ref{prop:existence_thermal_step}. Since $\rdrate \geq 0$, the nonnegativity of the temperature follows by Remark \ref{eq: temp-rem} for $\lp_0$ sufficiently small, where we use $\mechen(\yst{k-1}) \leq M\lp^2$ and $\mechen(\yst{k}) \leq M\lp^2$. \BBB  
\end{proof}

Our next goal is to adapt Proposition \ref{lem:initial_toten_bound} to the present setting.   As a preparation,  supposing that for $k \in \setof{0, \ldots,  T / \tau}$ the steps $\yst{k}$ and $\tst{k}$ exist, \NNN we define \BBB
\begin{equation}\label{def_Fl_lin}
  \mathcal{F}_\lp^{(k)} \defas \toten_\lp(\yst k, \tst k) - \lp^{-1} \sprod{\ell(k\tau)}{\yst k - \id},
\end{equation}
\NNN where $\ell$ is defined in \eqref{ell2}.  By repeating the proof of  Lemma \ref{lem:EF} we find \BBB
  \begin{equation}\label{EF_lin}
    \lp^{-1} |\langle \ell(k\tau), \yst k - \id \rangle|
    \le \min \lbrace \mathcal{F}_\lp^{(k)}, \toten_\lp(\yst k, \tst k) \rbrace
      + C_T C_{f,g}^2
  \end{equation}
\NNN   for  $k \in \setof{0, \ldots,  T / \tau}$, for a constant $C_T > 0$ only depending on $T$ and $C_{f, g}$ as in \eqref{def_Cfg}. \BBB

\begin{lemma}[Inductive bound on the rescaled total energy]\label{lem:initial_toten_bound_lin}
  There exists $\tau_0 \in (0, 1]$ and, given $M > 0$, $\lp_0 \in (0, 1]$ such that the following holds true: \CCC s\BBB uppose that for $\tau \in (0, \tau_0)$, $\lp \in (0, \lp_0)$, and $k \in \setof{1, \ldots,  T / \tau}$ the steps $\yst{0}, \ldots,  \yst{k}$ and $\tst{0}, \ldots,  \tst{k}$   exist such that $\mathcal{F}_\lp^{(l)} \leq M$ for all $l \in \setof{0, \ldots,  k-1}$. 
  Then, for a   a universal constant $C$ and a constant $C_T$ \CCC possibly \BBB depending on $T$ it holds that \BBB
  \begin{align*}
    \mathcal{F}_\lp^{(k)}
    &\leq C \Big(
       \mathcal{F}^{(0)}_\lp  \BBB
        + \sum_{l=0}^k
          \mathcal{F}_\lp^{(l)} \int_{(l-1)\tau}^{l\tau}
            \big(   1 + \Vert \dot{\ell}(t) \Vert_{H^{-1}}
              + \Vert \dot{\ell}(t+\tau) \Vert_{H^{-1}}  \big)\di t
        + \NNN \kappa \BBB \int_0^{k \tau} \int_\Gamma \theta_\flat^2 \di \haus^{d-1} \di t
      \Big)\notag \\
    &\phantom{\leq}\quad + C_T (1 + C_{f,g}^3).
  \end{align*}
\end{lemma}

\begin{proof}
 As a preliminary step, we show that the assumption $\mathcal{F}_\lp^{(l)} \leq M$ for all $l \in \setof{0, \ldots,  k-1}$ implies bounds on the  rescaled \NNN mechanical \BBB energy for all $l \in \setof{0, \ldots,  k}$.   First, \BBB by \eqref{EF_lin}  we get for $l \in \setof{0, \ldots,  k-1}$ that 
  \begin{equation}\label{toten_lp_l_bound}
    \toten_\lp(\yst{l},\tst{l}) \NNN = \BBB \mathcal{F}_{\lp}^{(l)} + \lp^{-1} \sprod{\ell(k\tau)}{\yst{k} - \id}  \leq 2\mathcal{F}_{\lp}^{(l)} + C_T C_{f, g}^2 \BBB \leq 2M + C_T C_{f, g}^2.
  \end{equation}
  Consequently, we can choose $\lp_0$ sufficiently small such that $\mathcal{M}(\yst{l}) \le 1$  for $l \in \setof{0, \ldots,  k-1}$. Then, we  apply \NNN \eqref{bad_mechen_bound} \BBB for $M=1$ \NNN  to \BBB get $\tau_0$ such that for $\tau \in (0,\tau_0]$ it holds that  
  \begin{align*}
    \mechen(\yst{k})
    &\leq (1+C_1\tau) \mechen(\yst{k-1}) + C_1\tau \big( \norm{\tst{k-1} \wedge 1}^2_{L^2(\Omega)}
    + \lp^2\norm{\fst k}_{L^2(\Omega)}^2 + \lp^2 \norm{\gst k}_{L^2(\Gamma)}^2\big),
  \end{align*}
  where $C_1$ is \CCC a universal constant\BBB.   By  \eqref{inten_lipschitz_bounds}  and  the fact that $1\wedge t \le t^{1/\alpha}$ for $t \ge 0$ we find
  \begin{align}\label{eq: severaltim}
  \norm{\tst{\NNN k-1} \wedge 1}^2_{L^2(\Omega)} \le C_0 \norm{\wst{k-1}}_{L^{\CCC\frac{2}{\alpha}\BBB}(\Omega)}^{\CCC\frac{2}{\alpha}\BBB},
  \end{align}
  such that, dividing the above estimate by $\lp^2$ and recalling \eqref{eq: rescaled toten} as well as Remark \ref{rem:bad_mechen_bound} we get
  \begin{equation*}
  \NNN  \mechen_\eps \BBB (\yst{k})
    \leq  (1+C_1\tau)  \toten_\lp
(\yst{k-1},\tst{k-1}) + C_1\big(\norm{f}_{L^2(I \times \Omega)}^2
        + \norm{g}_{L^2(I \times \Gamma_N)}^2\big).
  \end{equation*}
This along with \eqref{toten_lp_l_bound} shows that, possibly \CCC decreasing \BBB $\lp_0$, we have $\mechen_\lp(\yst{l}) \le 1$ for all $l \CCC \in \{0, \ldots, k \}\BBB$. This induces that in the following proof the constants coming from Lemmas \ref{lem:pos_det}, \ref{lem:bad_mechen_bound},  and \ref{lem:Vk_bound} are universal and denoted by $C_1$.

 As in the proof of Proposition \ref{lem:initial_toten_bound}, the strategy relies on a suitable \NNN test of the \BBB  mechanical and the thermal equation, \CCC see also \BBB \eqref{el_mech_test_step_l}--\eqref{el_temp_test_step_l}.  In contrast, however, the resulting equations cannot be summed up, but have to be treated separately. This will allow us to show the estimates 
  \begin{equation}\label{inductive_mechen_bound}
  \begin{aligned}
    &\mechen_\lp(\yst{k}) - \lp^{-1}\sprod{\ell(k\tau)}{\yst{k} - \id}
    \NNN + \BBB \frac{\tau}{\lp^2} \sum_{l = 1}^k \int_\Omega
      \drate(\nabla \yst{l-1}, \ddif \nabla\yst{l}, \tst{l-1}) \di x \\
    &\quad\leq C \mechen_\lp(\yst{0})
      + C_T (1 + C_{f, g}^3)
      + C \sum_{l = 0}^k \mathcal{F}_\lp^{(l)}
          \int_{(l-1)\tau}^{l\tau}
          \big(  1
            + \norm{\dot{\ell}(t)}_{H^{-1}}
            + \norm{\dot{\ell}(t+\tau)}_{H^{-1}}\big) \di t,
  \end{aligned}
  \end{equation}
  and 
    \begin{equation}\label{inductive_therm_bound}
  \begin{aligned}
    &\frac{\alpha}{2\lp^2} \int_\Omega (\wst{k})^{\frac{2}{\alpha}} \di x
      \NNN - \BBB \frac{\tau}{\lp^2} \sum_{l = 1}^{k} \int_\Omega
        \drate(\nabla \yst{l-1}, \ddif \nabla \yst{l}, \tst{l-1}) \di x \\
    &\quad\leq \NNN \frac{\alpha}{\lp^2} \BBB \int_\Omega (\wst{0})^{\frac{2}{\alpha}} \di x + C \mechen_\lp(\yst{0})
      +  \CCC C_T (1 + C_{f, g}^2) \BBB
      + C    \NNN \tau \sum_{l = 0}^{k} 
      \mathcal{F}_\lp^{(l)}  \BBB  +  \NNN \kappa \BBB \int_0^{k \tau} \int_\Gamma \theta_\flat^2 \di \haus^{d-1} \di t,
  \end{aligned}
  \end{equation}
  where $C$ is a universal constant and $C_T$ possibly depends on $T$. Then, in view of \eqref{eq: rescaled toten}, \eqref{def_Fl_lin}, and \eqref{toten_lp_l_bound} for $l=0$, the result follows by summing up the two estimates. We now treat \eqref{inductive_mechen_bound} and \eqref{inductive_therm_bound} separately. \BBB

  \textit{Step 1 (Inductive bound on the mechanical energy):}
   The first part is achieved by  bounds \NNN similar to the ones \BBB obtained in the proof of Proposition \ref{lem:initial_toten_bound}, and we therefore refer to estimates therein. Testing (\ref{mechanical_step_single})  for $l$ in place of $k$ with $z = \ddif \yst{l}$ we get \eqref{el_mech_test_step_l}. Then, multiplying both sides  by $\frac{\tau}{\lp^2}$, summing over $l= 1, \ldots,  k$, and using $W = W^{\rm el} + W^{\rm  cpl}$, \eqref{hypot_sum_lower_bound},   as well as \eqref{elpot_sum_lower_bound}, \BBB \CCC by possibly increasing \BBB  $C_1$ \CCC we derive that\BBB
  \begin{align}\label{mechen_diff_bound}
    &\mechen_\lp(\yst{k}) - \mechen_\lp(y_{0,\lp})
    \NNN +  \BBB \frac{\tau}{\lp^2} \sum_{l=1}^k \int_\Omega \big(
        \pl_F \cplpot(\nabla \yst{l}, \tst{l-1}) : \ddif \nabla \yst{l}
        + \drate(\nabla \yst{l-1}, \ddif \nabla \yst{l}, \tst{l-1})
      \big) \di x \notag\\
    &\quad\leq  C_1 \BBB \frac{\tau^2}{\lp^2} \sum_{l=1}^k \int_\Omega \abs{\ddif \nabla^2 \yst{l}}^2 \di x
      + \frac{\tau}{\lp} \sum_{l=1}^k \sprod{\lst l}{\ddif \yst{l}}.
  \end{align}
 Here, we also used the definition of $V_k$ in \eqref{def_Vk}, and the fact that the initial value is \NNN given \BBB by $y_{0,\lp}$. 
  By  \eqref{C_locally_lipschitz}, \eqref{eq: severaltim}, and  Young's inequality \CCC it follows that \BBB
  \begin{align}
    & \frac{1}{\lp^2} \BBB \Big|\int_\Omega
      \pl_F \cplpot(\nabla \yst{l}, \tst{l-1}) : \ddif \nabla \yst{l} \di x \Big|
    \leq \frac{C}{\lp^2} \int_\Omega (\tst{l-1} \wedge 1) (1 + \abs{\nabla \yst{l} - \Id}) \abs{\ddif \nabla \yst{l}} \di x \nonumber \\
    & \quad  \le  \frac{C}{\lp^2} \int_\Omega \big((\wst{l-1})^{\frac{2}{\alpha}} + \abs{\nabla \yst{l} - \Id}^2\big) \di x + \frac{C}{\lp^2} \int_\Omega    \abs{\ddif \nabla \yst{l}}^2 \di x. \label{first_cpl_term_bound}
  \end{align}
     By Lemma \ref{lem:Vk_bound} \NNN and \BBB \eqref{eq: severaltim}  we get 
  \begin{equation*}
\sum_{l = 1}^k \tau \int_\Omega \abs{\ddif \nabla \yst l}^2 \di x
    \le  \CCC \lp^2 \NNN C_1 \BBB \mechen_\lp(y_{0,\lp})  + \CCC \lp^2 \BBB C_1 C_T  C_{f, g}^2 + C_1\tau \sum_{l = 0}^{k-1} \big( \NNN  \mechen(\yst l)  + \BBB  \Vert(\wst{l-1})^{\frac{1}{\alpha}} \Vert^2_{L^2(\Omega)}\big).  
  \end{equation*} 
  Using the definition of the \CCC total \BBB energy in \eqref{eq: rescaled toten},  the definition of $\elen_\lp$, and \ref{W_lower_bound_spec}, we insert this in \eqref{first_cpl_term_bound} \CCC to \BBB obtain  
\begin{align}\label{first_cpl_term_bound2}  
  &\frac{\tau}{\lp^2}  \sum_{l=1}^k\Big|\int_\Omega
      \pl_F \cplpot(\nabla \yst{l}, \tst{l-1}) : \ddif \nabla \yst{l} \di x \Big| +  \frac{\tau}{\lp^2} \sum_{l=1}^k \int_\Omega \abs{\ddif \nabla^2 \yst{l}}^2 \di x \\
      &\quad\le   C\tau \sum_{l=0}^k    \toten_\lp(\yst{l}, \tst{l}) 
          + C\mechen_\lp(y_{0,\lp})  + \NNN C \BBB C_T     C_{f, g}^2 .           
   \end{align}
    Next, \NNN by repeating the argument in \eqref{rufanew}--\eqref{rufanew2} we \BBB  find
  \begin{align}\label{rufanew3}
 \frac{\tau}{\lp} \sum_{l = 1}^k \langle \lst l, \ddif \yst l \rangle
 &\le \lp^{-1}\langle \ell (k\tau), {\hat{y}_\tau} (k \tau) - \id \rangle
    - \lp^{-1}\langle \ell (0), {\hat{y}_\tau} (0)  - \id\rangle
  \notag \\  &\phantom{\leq} \quad +
   \NNN C\sum_{l=0}^k \Big(
        \mathcal{F}^{(l)}_\eps \int_{(l-1)\tau}^{l\tau}\hspace{-0.1cm}
          \big(
            \Vert \dot{\ell}(t) \Vert_{H^{-1}}
            + \Vert \dot{\ell}(t+\tau) \Vert_{H^{-1}}
          \big) \di t
      \Big) +   C_T  (1  + C_{f,g}^3). \BBB 
  \end{align}
\BBB   Employing \eqref{first_cpl_term_bound2} and \eqref{rufanew3} in \eqref{mechen_diff_bound}, and using again \eqref{EF_lin} we arrive at \eqref{inductive_mechen_bound}.

  \textit{Step 2 (Inductive bound on the temperature):}
 For $\alpha \in [1,2)$, \BBB let $\chi(t) \defas \frac{\alpha}{2} (\lp^\alpha + t)^{\frac{2}{\alpha}}$   for $t \geq 0$. The convexity of $\chi$ implies 
  \begin{equation*}
    \int_\Omega (\wst{l} - \wst{l-1}) \chi'(\wst{l}) \di x
    \geq \int_\Omega \chi(\wst{l}) \di x - \int_\Omega \chi(\wst{l-1}) \di x.
  \end{equation*} 
Summation of this estimate \CCC over \BBB $l = 1, \ldots,  k$ \CCC leads to \BBB 
\begin{align}\label{eq: too divide}
&\frac{\alpha}{2} \int_\Omega (\wst{k})^{\frac{2}{\alpha}} \di x \le \int_\Omega \chi(\wst{\NNN k}) \di x \le \int_\Omega \chi(\wst{0}) \di x + \sum_{l=1}^k \int_\Omega (\wst{l} - \wst{l-1}) \chi'(\wst{l}) \di x.
\end{align}
This suggests to test \eqref{el_thermal_step} (for $l$ in place of $k$, $\rdrate$ in place of $\drate$, and $\lp^{\alpha} \btst{l}$ in place of $\NNN \lp^{2} \BBB \btst{l}$) \BBB  with $\varphi = \chi'(\wst{l})$ which yields
  \begin{align}\label{thermal_test_reg}
    0 &= \int_\Omega \Big(\ddif \wst{l} - \partial_F \cplpot(\nabla \yst{l-1}, \tst{l-1}) : \ddif \nabla \yst{l} \NNN - \BBB \rdrate(\nabla \yst{l-1}, \ddif \nabla \yst{l}, \tst{l-1})\Big) \chi'(\wst{l}) \di x \notag \\
    &\phantom{=}\quad
      + \int_\Omega \hcm(\nabla \yst{l-1}, \tst{l-1}) \nabla \tst{l} \cdot \nabla (\chi'(\wst{l})) \di x
      + \NNN \kappa \BBB \int_{\Gamma} (\tst{l} - \lp^{\alpha} \btst{l}) \chi'(\wst{l}) \di \haus^{d-1}.
  \end{align}
   We now estimate the various terms separately. First, we \BBB employ \eqref{C_locally_lipschitz}, \eqref{inten_lipschitz_bounds},  \eqref{pos_det}, \BBB and Young's inequality with powers $2/\alpha$ and $2/(2 - \alpha)$ to obtain
  \begin{align}
    &\Big|\int_\Omega \Big[\partial_F \cplpot(\nabla \yst{l-1}, \tst{l-1}) : \ddif \nabla \yst{l} \Big] \chi'(\wst{l}) \di x \Big| \nonumber \\
    &\quad\leq  2C_0 \BBB \int_\Omega (\tst{l-1} \wedge 1) (1 +  \abs{\nabla \yst{l-1} } \BBB ) \abs{\ddif \nabla \yst{l}} (\lp^\alpha + \wst{l})^{\frac{2}{\alpha} - 1} \di x \nonumber \\
    &\quad\leq C (1 +  C_1 \BBB )
      \int_\Omega \Big( (\wst{l-1} \wedge 1)^{\frac{2}{\alpha}} \abs{\ddif \nabla \yst{l}}^{\frac{2}{\alpha}} + \big( \lp^2 + (\wst{l})^{\frac{2}{\alpha}} \big) \Big) \di x. \label{second_cpl_term_bound_incompl}
  \end{align}
  If $\alpha \in (1, 2)$, we use in the last estimate another Young's inequality, now with powers $\alpha/(\alpha - 1)$ and $\alpha$, as well as $t \wedge 1 \leq t^{(\alpha-1)/\alpha}$ for all $t \geq 0$ to show
  \begin{align}
    &\Big|\int_\Omega \Big[\partial_F \cplpot(\nabla \yst{l-1}, \tst{l-1}) : \ddif \nabla \yst{l} \Big] \chi'(\wst{l}) \di x \Big|
    \leq C \int_\Omega
     \big( \lp^2 + (\wst{l-1})^{\frac{2}{\alpha}}
      + (\wst{l})^{\frac{2}{\alpha}}
      + \abs{\ddif \nabla \yst{l}}^2 \big) \di x \nonumber \\
    &\quad\leq C \lp^2 \Big(
        1 + \toten_\lp(\yst{l-1}, \tst{l-1}) + \toten_\lp(\yst{l}, \tst{l})
        + \frac{1}{\lp^2} \int_\Omega \abs{\ddif \nabla \yst{l}}^2 \di x
      \Big). \label{second_cpl_term_bound}
  \end{align}
  Notice that for $\alpha = 1$ the above bound follows \CCC directly \BBB from \eqref{second_cpl_term_bound_incompl}, simply using $\wst{l-1} \wedge 1 \leq 1$.
  
 Next, we estimate \NNN the \BBB $\rdrate$-term. \BBB   From the definition of $\rdrate$  in \eqref{def_rdrate}, \BBB we have that $\rdrate \leq \drate^{\frac{\alpha}{2}}$.
  Hence, by Young's inequality with power $2/\alpha$ and $2/(2-\alpha)$,  and by a similar reasoning as before, \BBB it follows that
  \begin{align}\label{eq: xireg}
    &\int_\Omega \rdrate(\nabla \yst{l-1}, \ddif \nabla \yst{l}, \tst{l-1}) \chi'(\wst{l}) \di x \leq \int_\Omega \drate(\nabla \yst{l-1}, \ddif \nabla \yst{l}, \tst{l-1}) \di x
      + C \lp^2 (1 + \toten_\lp(\yst{l}, \tst{l})).
  \end{align}
  We continue by investigating the $\mathcal{K}$\CCC -term\BBB.  By  \eqref{Wint} and  \BBB  the chain rule we have
  \begin{align*}
    \nabla (\chi'(\wst{l})) = \frac{2-\alpha}{\alpha} (\lp^\alpha + \wst{l})^{\frac{2}{\alpha}-2} \Big[
      \Big(
        \partial_F \cplpot(\nabla \yst{l}, \tst{l}) 
        - \tst{l} \partial_{F \theta} \cplpot(\nabla \yst{l}, \tst{l})
      \Big) : \nabla^2 \yst{l} & \\
      - \tst{l} \partial_\theta^2 \cplpot(\nabla \yst{l}, \tst{l}) \nabla \tst{l} &
    \Big].
  \end{align*}
  This \BBB combined with \eqref{C_locally_lipschitz}, the second and third bound in \ref{C_bounds},  \eqref{inten_lipschitz_bounds}, \BBB \eqref{W1infty_dist_to_id}, and \eqref{bound_hcm} \CCC leads to\BBB
  \begin{equation}\label{lower_bound_mixed_term}
  \NNN \Kst{l-1} \BBB \nabla \tst{l} \cdot \nabla (\chi'(\wst{l}))
    \geq \frac{2-\alpha}{\alpha} (\lp^\alpha + \wst{l})^{\frac{2}{\alpha} - 2} \Big(
       c \abs{\nabla \tst{l}}^2
      -C (\wst{l} \wedge 1) \abs{\nabla^2 \yst{l}} \abs{\nabla \tst{l}}
    \Big)
  \end{equation}
  \NNN for some $c>0$,  where we set   $ \Kst{l-1} \defas \hcm(\nabla \yst {l-1}, \tst {l-1})$ for brevity. \BBB
 (In the definition of $\chi$, the addend $\lp^\alpha$ appears to ensure that  \BBB $(\lp^\alpha + \wst{l})^{\frac{2}{\alpha} - 2}$ is well-defined \CCC for $\alpha > 1$\BBB.)   By $t \wedge 1 \leq t^{1 - 2/(p\alpha)}$ for all $t \geq 0$, Young's inequality twice (firstly with power 2 and constant $\lambda \in (0, 1)$, secondly with powers $p/(p-2)$ and $p/2$) we derive that
  \begin{align*}
    (\wst{l} \wedge 1) \abs{\nabla^2 \yst{l}} \abs{\nabla \tst{l}}
    &\leq \lambda \abs{\nabla \tst{l}}^2 + \frac{1}{\lambda} (\wst{l})^{2\frac{p-2}{p}} (\wst{l})^{\frac{4(\alpha-1)}{p\alpha}} \abs{\nabla^2 \yst{l}}^2 \\
    &\leq \lambda \abs{\nabla \tst{l}}^2 + \frac{1}{\lambda} \Big(
        (\wst{l})^2
        + (\wst{l})^{2\frac{\alpha-1}{\alpha}} \abs{\nabla^2 \yst{l}}^p
      \Big).
  \end{align*}
  Choosing $\lambda$ small enough such that $C \lambda <  c/2\BBB$ (with $c$ and $C$ as in \eqref{lower_bound_mixed_term}), we derive with \eqref{lower_bound_mixed_term} that
    \begin{equation}\label{eq: the K term-remark}
    \int_\Omega   \NNN \Kst{l-1} \BBB \nabla \tst{l} \cdot \nabla (\chi'(\wst{l})) \di x
    \geq \frac{c}{2} \NNN  \frac{2-\alpha}{\alpha} \BBB \CCC (\lp^\alpha + \wst{l})^{\frac{2}{\alpha} - 2} \BBB  \abs{\nabla \tst{l}}^2- C \lp^2  \NNN  \toten_\lp(\yst{l}, \tst{l}). \BBB
  \end{equation} \BBB
    Lastly, for the boundary term, we use \eqref{inten_lipschitz_bounds} as well as Young's inequality with powers $2/\alpha$ and $2/(2-\alpha)$ and constant $\lambda \in (0, 1)$ to arrive at
  \begin{align*}
    &\int_{\Gamma} (\tst{l} - \lp^{\alpha} \btst{l})\chi'(\wst{l}) \di \haus^{d-1}  \ge \int_{\Gamma} (C_0^{-1}\wst{l} - \lp^{\alpha} \btst{l})(\lp^\alpha + \wst{l})^{\frac{2}{\alpha} - 1} \di \haus^{d-1} \BBB \\
    &\quad\geq  \frac{1}{C_0} \BBB \int_\Gamma (\wst{l})^{\frac{2}{\alpha}} \di \haus^{d-1}
    -\frac{ \lp^2\BBB}{\lambda} \int_\Gamma (\btst{l})^{\frac{2}{\alpha}} \di\haus^{d-1}
    -\lambda \int_\Gamma ( \lp^\alpha  + \BBB \wst{l})^{\frac{2}{\alpha}} \di \haus^{d-1}.
  \end{align*}
   Therefore, choosing $\lambda$ sufficiently small with respect to $1/C_0$, we get 
  \begin{align}\label{eq: lastbdy} 
  \int_{\Gamma} (\tst{l} - \lp^{\alpha} \btst{l})\chi'(\wst{l}) \di \haus^{d-1} \ge -C\lp^2\CCC\Big(\BBB1 + \int_\Gamma (\btst{l})^{\frac{2}{\alpha}} \di\haus^{d-1}\CCC\Big)\BBB. 
  \end{align}
    We \CCC then \BBB \NNN divide \eqref{eq: too divide} by \BBB  $\lp^2$, insert \eqref{thermal_test_reg} multiplied by $\tau$ in this inequality, and use \eqref{second_cpl_term_bound}, \eqref{eq: xireg},  \eqref{eq: the K term-remark}, and \eqref{eq: lastbdy} to estimate the various terms. This together with the bounds from \eqref{EF_lin} and   \eqref{first_cpl_term_bound2},  and the fact that 
    $$\tau\sum_{l=1}^k\Vert (\btst{l})^{\CCC\frac{2}{\alpha}\BBB} \Vert_{L^1(\Gamma)} \le C\tau\sum_{l=1}^k \big(1+ \Vert \btst{l} \Vert^2_{L^2(\Gamma)}\big) \le C_T + C\Vert \bt \Vert^2_{L^2(\NNN [0,k\tau] \BBB\times \Gamma)} $$
    by  H\"older's inequality yields \eqref{inductive_therm_bound}. This concludes the proof. 
\end{proof}

\begin{theorem}[Well-definedness of the scheme]\label{thm:toten_velo_bound_linearized}
  For any $T > 0$ there  exist  a constant $\bar{C}_T$,  corresponding constants
  \begin{equation*}
    M' \defas
    2 e^{ \bar{C}_T(1  + C_{f,g})} \Big(
    \NNN \bar{C}_T \BBB  \mathcal{F}^{(0)}
      + \bar{C}_T ( 1 + C_{f,g}^3 )
      + \kappa \int_0^T \int_\Gamma \theta_\flat^2 \di \haus^{d-1} \di t
    \Big), \quad  M\defas 2M' + \bar{C}_T   C_{f,g}^2,
  \end{equation*}
  as well as constants  $\lp_0 \CCC,\, \tau_0 \in (0, 1]\BBB$ depending also on $M$  
  such that  the following holds true:  \CCC f\BBB or each $\lp \in (0,\lp_0)$ and $\tau \in (0, \tau_0)$ such that $T / \tau \in \N$ the sequences $\yst{0}, \ldots,  \yst{T/\tau}$ and $\tst{0}, \ldots,  \tst{T/\tau}$ \NNN exist, \BBB  and  for all $k \CCC\in \{ 0, \ldots,  T/\tau \}\BBB$ we  have that 
  \begin{equation*}
    \toten_\lp(\yst k, \tst k)\leq M,
  \end{equation*}
  \begin{equation}\label{velocity_apriori_bound-new}
  \sum_{k=1}^{T/\tau} \frac{\tau}{\lp^2} \int_\Omega \abs{\ddif \nabla \yst l}^2 \di x
    \leq \NNN  \bar{C}_T M (1+T)  +  \bar{C}_T     C_{f, g}^2. \BBB
  \end{equation}
\end{theorem}

\begin{proof}
  The theorem is a consequence of Lemma \ref{lem:initial_toten_bound_lin} and Lemma \ref{lem:Vk_bound}. The argument is similar to the one of Theorem \ref{thm:apriori_toten_velo_bound} and we therefore omit the details. Let us just mention that the energy bound follows in the same way by induction, up to using different values $\beta$, $a_l$, and $b_l$ in \eqref{eq: bbeta}, and by employing \eqref{EF_lin} in place of Lemma \ref{lem:EF}. Based on the uniform energy bound, Proposition \ref{prop: thermal with reg} indeed shows that the scheme is well\CCC -\BBB defined, provided that $\lp_0$ is chosen sufficiently small. Eventually, the bound on the strain rates follows from Lemma \ref{lem:Vk_bound}, see particularly \eqref{first_cpl_term_bound2} in the previous proof. 
\end{proof}

\BBB
\begin{remark}\label{rem: case1}
{\normalfont
Due to our regularization of the dissipation rate, in the case $\alpha \in [1, 2)$ we obtain the  additional control
\begin{equation}\label{weighted_nabla_temp_bound_reg}
  \int_0^T \int_\Omega \frac{\abs{\nabla \overline{\mu}_{\lp, \tau}}^2}{(1 + \overline{\mu}_{\lp, \tau})^{2(1 - \frac{1}{\alpha})}} \di x \di t \leq C_\alpha < \infty
\end{equation}
for a constant $C_\alpha$  depending on $\alpha$, but \BBB independent of $\lp$ and $\tau$, where we shortly wrote $\overline{\mu}_{\lp, \tau} \defas \lp^{-\alpha} \nt$ (see also \eqref{y_interpolations} for the definition of $\nt$). This follows by using \NNN the positive term on the right-hand side of  \eqref{eq: the K term-remark}. \BBB
}
\end{remark}

\subsection{A priori bounds}\label{sec: a priori}

  Fix initial values $(y_{0,\lp}, \theta_{0,\lp})$ with $\mathcal{E}_\lp(y_{0,\lp},\theta_{0,\lp}) \le E_0$ for some $E_0>0$. Without further notice, we suppose in this subsection that the sequences $\yst{0}, \ldots,  \yst{T/\tau}$ and $\tst{0}, \ldots,  \tst{T/\tau}$ exist by Theorem \ref{thm:van_tau}(i) or Proposition \ref{cor:van_tau_reg}(i), respectively, for $\lp \in (0, \lp_0)$ for some $\lp_0$ depending only on  $\alpha$, $E_0$, $f$, $g$,  $\bt$, and $T$. (In \CCC the \BBB case $\alpha = 2$, we can set $\lp =1$). \BBB We derive a priori bounds on the rescaled  displacements \BBB $\lp^{-1} (\yst l - \id)$  and the rescaled temperatures $\lp^{-\alpha} \tst l$ for $l \CCC \in \{1, \ldots,  T/\tau\}\BBB$. To this end, for small $\lp$, we will \NNN again \BBB assume  \ref{W_lower_bound_spec}. Recall the definition of the interpolations in \eqref{y_interpolations}.
In a similar way, we write $\nw = \inten(\ny,\nt)$, and similarly for the other interpolations.
The next lemma is a direct  consequence of Theorem \ref{thm:apriori_toten_velo_bound} and \BBB Theorem \ref{thm:toten_velo_bound_linearized}.

\begin{lemma}[First a priori bounds]\label{lemma: first a prioiro}
  Let $E_0>0$ such that    $\mathcal{E}_\lp(y_{0,\lp},\theta_{0,\lp}) \le E_0$. \BBB   Then, there exists a constant $C > 0$   depending on \NNN $\alpha$, \BBB $E_0$, $f$, $g$,  $\bt$,  and $T$ \BBB such that  $\mathcal{E}_\lp(\yst{k},\tst{k}) \le C$ for all $k \CCC \in \{1, \ldots,  T/\tau \} \BBB$, and \BBB the interpolants constructed from the discrete solutions satisfy
  \begin{subequations}\label{a priori}
  \begin{align}
    &\Vert \ny - \id \Vert_{L^\infty(I; W^{1, \infty}(\Omega;\R^d))} + \Vert \nabla^2 \ny  \Vert_{L^\infty(I; L^{p}(\Omega; \R^d))}
      \le C \lp^{2/p}, \label{a_priori_Linfty_W2p_lin} \\
    &\Vert \ny - \id \Vert_{L^\infty(I; H^{1}(\Omega; \R^d))}
      \le C \lp, \label{a_priori_Linfty_H1_lin} \\
    &\Vert \nabla \dotay \Vert_{L^2(I \times \Omega; \R^{d\times d})}
      \le C \lp, \label{a_priori_strain_rates_lin} \\
    &\Vert \nt \Vert_{L^\infty(I; L^1(\Omega))}
        + \Vert \nw \Vert_{L^\infty(I;L^1(\Omega))}
      \le C   \lp^\alpha \BBB \label{a_priori_temp_L1_lin}
  \end{align}
  \end{subequations}
  Estimates \eqref{a_priori_Linfty_W2p_lin}--\eqref{a_priori_Linfty_H1_lin} also hold for $\py$, and \eqref{a_priori_temp_L1_lin} holds for $\pt$, $\at$, $\pw$, and $\aw$, as well.
\end{lemma}

\begin{proof}
 Let us first suppose that \ref{W_lower_bound_spec} holds. \NNN The energy bound on $\mathcal{E}_\eps$ for $\alpha = 2$ and $\alpha \in [1,2)$ follows directly from Theorem \ref{thm:apriori_toten_velo_bound} and   Theorem \ref{thm:toten_velo_bound_linearized}, respectively.  \BBB
  The first two estimates \CCC can be shown \BBB from the uniform bound on the energy,  \ref{W_lower_bound_spec}, \eqref{H1_dist_to_id}, \eqref{H_upper_bound_spec}, and Poincar\'e's inequality. \BBB   In a similar way, the bound on $\nw$ in \eqref{a_priori_temp_L1_lin} follows from  \NNN the bound on the total rescaled energy, \BBB  \eqref{eq: rescaled toten}, and H\"older's inequality. \BBB Then, the proof of \eqref{a_priori_temp_L1_lin} is concluded by \eqref{inten_lipschitz_bounds}.
  Finally, \eqref{a_priori_strain_rates_lin} is a direct consequence of  \NNN  \eqref{velocity_apriori_bound} and \eqref{velocity_apriori_bound-new}, respectively. \BBB
  
  \NNN  Eventually, for $\alpha =2$ and $\lp$ near $1$, the result also holds without assuming \ref{W_lower_bound_spec}   as \ref{W_lower_bound}   allows us to derive \eqref{a_priori_Linfty_W2p_lin}--\eqref{a_priori_Linfty_H1_lin} with $C$ in place of $C\eps^{2/p}$ and $C\eps$ on the right-hand side.  
\end{proof}

In order to pass to the limit $\tau \to 0$ in the next section, we need additional a priori bounds for the temperature.
Testing the equation \NNN  \eqref{el_thermal_step} \BBB turns out to be delicate since \NNN   for $\alpha =2$  \BBB  the viscous dissipation $\drate(\nabla \yst {k-1}, \ddif \nabla \yst k , \tst {k-1})$ is only bounded in $L^1(I \times \Omega)$.
Thus, to obtain improved estimates that work \NNN in this case, \BBB we employ special test functions developed by Boccardo and Gallou\"et \cite{BoccardoGallouet89Nonlinear} for parabolic equations with a measure-valued right-hand side, see also \cite{FreireislMalek06OntheNavier}.
We follow here the approach in \cite{MielkeRoubicek20Thermoviscoelasticity}.
However, almost complete proofs are provided since compared to their setting we perform the estimates in the time discrete setting and we derive fine estimates in terms of the small parameter $\lp$. 

\begin{lemma}[Weighted $L^2$-bound]\label{lem:weightet_L2_bound}
 For any $\eta \in (0, 1)$ \BBB there exists a constant $C$   independent of $\lp$, $\tau$, and $\alpha$ such that 
  \begin{equation}\label{weighted_L2_bound}
    \sum_{k = 1}^{ T /\tau \BBB} \tau
      \int_\Omega
        \frac{\eta}{\Big( 1 + \lp^{-\alpha} \wst k \Big)^{1 + \eta}}
        \big| \nabla \wst k \big|^2
      \di x \leq C  \lp^{2\alpha}. \BBB
  \end{equation}
\end{lemma}

\NNN Actually, this statement is needed only for $\alpha =2$ since for $\alpha \in [1,2)$ we have a better estimate by Remark \ref{rem: case1}. Still, we state and prove   the result for any $\alpha$ since the following argument \CCC does not depend on $\alpha$\BBB.  \BBB 
\begin{proof}
  \textit{Step 1:}
  In the following, $C$ will denote a constant independent of $k$, $\lp$, $\tau$, $\alpha$, and $\eta$.   Given $k \CCC \in \{1, \ldots,   T/\tau\}$, we have by (\ref{el_thermal_step}) \NNN (for $\rdrate$ in place of $\drate$) \BBB that for any $\vphi_k \in H^1(\Omega)$
  \begin{equation}\label{thermal_step_k}
  \begin{aligned}
    \int_\Omega \ddif \wst k \vphi_k \di x
    &= \int_\Omega  {h}^k_{\eps,\tau}  \vphi_k \di x
      - \int_\Omega \CCC \Kst{k-1} \BBB \nabla \tst k \cdot \nabla \vphi_k \di x - \kappa \int_\Gamma (\tst k - \lp^\alpha \btst k) \vphi_k \di \haus^{d-1},
  \end{aligned}
  \end{equation}
  where we write
  \begin{equation}\label{h_tau_hcm_short}
  \begin{aligned}
    h_{\eps,\tau}^{(k)} &\defas
      \pl_F \cplpot(\nabla \yst{k-1}, \tst{k-1})
        : \ddif \nabla \yst k
      + \rdrate(\nabla \yst{k-1}, \ddif \nabla \yst k, \tst{k-1}), \\
    \CCC \Kst{k-1} \BBB &\defas \hcm(\nabla \yst{k-1}, \tst{k-1})
  \end{aligned}
  \end{equation}
  for brevity.   Given $\eta \in (0, 1)$, let $\chi_{\eta, \lp} \colon \R \to \R$ be the function uniquely determined by $\chi_{\eta, \lp}(0) = 0$ and $\chi_{\eta, \lp}'(t) = 1 - \frac{1}{(1 + \lp^{-\alpha} t)^\eta}$ for all $t \geq 0$.
  Choosing $\vphi_k \defas \chi_{\eta, \lp}'(\wst k)$ in (\ref{thermal_step_k}), multiplying both sides by $\tau$, and summing over $k = 1, \ldots,  T/\tau$, we arrive at
  \begin{align}\label{discrete_thermal_evol}
    &\sum_{k=1}^{T/\tau} \int_\Omega (\wst k  - \wst{k-1}) \chi_{\eta, \lp}' (\wst k)  \di x   = \sum_{k=1}^{T/\tau} \tau \int_\Omega h_{\eps,\tau}^{(k)} \chi_{\eta, \lp}' (\wst k) \di x
   \notag  \\ & \quad   -\sum_{k=1}^{T/\tau} \tau
        \int_\Omega
          \chi_{\eta, \lp}''(\wst k) \CCC \Kst{k-1} \BBB
          \nabla \tst k  \cdot \nabla \wst k  \di x  -\kappa \sum_{k=1}^{T/\tau} \tau
      \int_{\NNN \Gamma} \BBB  (\tst k - \lp^\alpha \btst k)
      \chi_{\eta, \lp}'(\wst k)  \di \haus^{d-1}.
  \end{align}
  \NNN Our goal is to show \BBB   
  \begin{equation}\label{nablaw_nablatheta_bound}
    \sum_{k=1}^{T/\tau} \tau
      \int_\Omega \chi_{\eta, \lp}''(\wst k)
        \CCC \Kst{k-1} \BBB \nabla \tst k \cdot \nabla \wst k \di x
    \leq C \lp^\alpha.
  \end{equation}
\NNN To this end, we estimate the various terms in \eqref{discrete_thermal_evol}. First, \BBB  notice that by the convexity of $\chi_{\eta, \lp}$ we have for any $k \CCC \in \{1,\, \ldots,\, T/\tau\}\BBB$  that \BBB
  \begin{equation*}
    \chi_{\eta, \lp}(\wst{k-1})
    \geq \chi_{\eta, \lp}(\wst k) + \chi_{\eta, \lp}'(\wst k)(\wst{k-1} - \wst k),
  \end{equation*}
  and therefore
  \begin{align*}
    &\sum_{k=1}^{T/\tau}
      \int_\Omega \big(\wst k - \wst{k-1}\big) \chi_{\eta, \lp}'(\wst k) \di x
    \geq \sum_{k=1}^{T/\tau}
      \int_\Omega \big(
        \chi_{\eta, \lp}(\wst k) - \chi_{\eta, \lp}(\wst{k-1})
      \big) \di x \\
    &\quad= \int_\Omega \chi_{\eta, \lp}(\wst {T/\tau}) \di x
      - \int_\Omega \chi_{\eta, \lp}(\wst 0) \di x
    \geq -\int_\Omega \wst 0 \di x \geq - C \lp^\alpha,
  \end{align*}
  where we used $\chi_{\eta, \lp} \geq 0$ and $\chi_{\eta, \lp}(t) \leq t$ for all $t \ge 0$, and in the last step also \eqref{a_priori_temp_L1_lin}.  Using \eqref{C_locally_lipschitz}, $\rdrate \leq \drate$, \eqref{diss_rate}, \ref{D_bounds},  and \BBB \eqref{a_priori_Linfty_W2p_lin} we see that
  \begin{equation*}
    \sum_{k=1}^{T/\tau} \int_\Omega |h_{\eps,\tau}^{(k)}| \di x
    \leq C \sum_{k=1}^{T/\tau} \int_\Omega \big(
      \sqrt{\tst {k-1}} |\ddif \nabla \yst k|
      + \abs{\ddif \nabla \yst k}^2
    \big) \di x,
  \end{equation*}
  where we used that $t \wedge 1 \le \sqrt{t}$ \NNN for $t \ge 0$. \BBB
  Then, by Young's inequality, $\chi_{\eta, \lp}' \leq 1$,  \eqref{a_priori_strain_rates_lin}, and \eqref{a_priori_temp_L1_lin} we get
  \begin{equation}\label{fk_bound}
    \sum_{k=1}^{T/\tau} \tau \int_\Omega h_{\eps,\tau}^{(k)} \chi_{\eta, \lp}'(\wst k) \di x
    \le \sum_{k=1}^{T/\tau} \tau \int_\Omega |h_{\eps,\tau}^{(k)}| \di x
    \leq C \sum_{k=1}^{T/\tau} \tau
      \int_\Omega \big(\tst {k-1} + \abs{\ddif \nabla \yst k}^2 \big) \di x
    \leq C \lp^\alpha,
  \end{equation}
  where we have used $\alpha \CCC\leq\BBB 2$.
  Lastly, by $\tst k \geq 0$, $\kappa \ge 0$, $\chi_{\eta, \lp}' \in [0, 1]$, and the definition of $\btst k$ it follows that
  \begin{align*}
    - \kappa \sum_{k=1}^{T/\tau} \tau
      \int_\Omega (\tst k - \lp^\alpha \btst k) \chi_{\eta, \lp}'(\wst k) \di \haus^{d-1}
    &\leq \kappa \sum_{k=1}^{T/\tau} \tau
      \int_\Omega \lp^\alpha \btst k \chi_{\eta, \lp}'(\wst k) \di \haus^{d-1} \\
    &\leq \kappa \lp^\alpha \int_0^T \int_\Omega \theta_\flat \di x \di t
    \le C \lp^\alpha,
  \end{align*}
  where $C$ also depends on $\theta_\flat$. \BBB
  Employing all the aforementioned estimates in (\ref{discrete_thermal_evol}) we  \NNN obtain \eqref{nablaw_nablatheta_bound}. \BBB

  \textit{Step 2:}
  We are now ready to show \eqref{weighted_L2_bound}.
  In this regard, first notice the following relation between $\nabla \wst k$ and $\nabla \tst k$: since $\wst k = \inten(\nabla \yst k, \tst k)$, \eqref{Wint} implies
  \begin{align}\label{tempgrad_relation}
    \nabla \wst k
    &= \left[
        \pl_F \cplpot(\nabla \yst k, \tst k)
        - \tst k \pl_{F \theta} \cplpot(\nabla \yst k, \tst k)
      \right] : \nabla^2 \yst k
      - \tst k \pl_\theta^2 \cplpot(\nabla \yst k, \tst k) \nabla \tst k \notag \\
    &=: \tilde{W}^{(k)}_1 : \nabla^2 \yst k + \tilde{W}^{(k)}_2 \nabla \tst k.
  \end{align}
 By \eqref{C_locally_lipschitz}, \ref{C_bounds}, and \eqref{a_priori_Linfty_W2p_lin}, we find that the abbreviations $\tilde{W}^{(k)}_1$ and $\tilde{W}^{(k)}_2$ satisfy $\tilde{W}^{(k)}_1 \le C(\tst k \wedge 1)$ and $\tilde{W}^{(k)}_2 \in [\ac, \aC]$, respectively.
 Then, \NNN using \BBB    \eqref{tempgrad_relation},  Lemma \ref{lem:bound_hcm},  and the energy bound from Lemma \ref{lemma: first a prioiro} \BBB we see that there exists a constant $c > 0$   such that
  \begin{align}\label{nabla2_nablaw_estimate-before}
    \frac{c}{\aC} \chi_{\eta, \lp}''(\wst k) \abs{\nabla \wst k}^2
    &\leq \NNN \big(    \tilde{W}^{(k)}_2\big)^{-1} \BBB
      \chi_{\eta, \lp}''(\wst k)
      \CCC \Kst{k-1} \BBB \nabla \wst k \cdot \nabla \wst k \notag \\
    &\leq \chi_{\eta, \lp}''(\wst k) \CCC \Kst{k-1} \BBB \nabla \tst k \cdot \nabla \wst k
      + C \chi_{\eta, \lp}''(\wst k) (\tst k \wedge 1)
        \abs{\nabla^2 \yst k} \abs{\nabla \wst k}.
  \end{align}
  We now control the second term above.
  By $t \wedge 1 \leq t^{\frac{p-1}{p}}$ for all $t \geq 0$ and Young's inequality with constant $\lambda \in (0, 1)$ (to be chosen later), we estimate by \eqref{inten_lipschitz_bounds}
  \begin{equation}\label{nabla2_nablaw_estimate}
    C \chi_{\eta, \lp}''(\wst k) (\tst k \wedge 1)
      \abs{\nabla^2 \yst k} \abs{\nabla \wst k}
    \le C \chi_{\eta, \lp}''(\wst k) \Big(
      \lambda \abs{\nabla \wst k}^2
      + \frac{1}{\lambda} (\wst k)^{2\frac{p-1}{p}} |\nabla^2 \yst k|^2
    \Big).
  \end{equation}
  Using the elementary fact
  \begin{equation}\label{elementary}
    \chi_{\eta, \lp}''(\wst k)
    = \frac{\eta}{\lp^\alpha \Big(1 + \lp^{-\alpha} \wst k\Big)^{1 + \eta}}
    \le \frac{1}{\lp^\alpha + \wst k},
  \end{equation}
  we derive by Young's inequality with power\CCC s \BBB  \NNN $p/(p-2)$ \BBB and $p/2$ that
  \begin{equation*}
    \chi_{\eta, \lp}''(\wst k) (\wst k)^{2\frac{p-1}{p}} \abs{\nabla^2 \yst k}^2
    \leq (\wst k)^{2\frac{p-1}{p} - 1} \abs{\nabla^2 \yst k}^2
    = (\wst k)^{\frac{p-2}{p}} \abs{\nabla^2 \yst k}^2
    \leq C (\wst k + \abs{\nabla^2 \yst k}^p).
  \end{equation*}
  Let us take $\lambda$ small enough so that $C \lambda \leq c/(2 \aC)$ where $c$ is as in \eqref{nabla2_nablaw_estimate-before} and $C$ is as in \eqref{nabla2_nablaw_estimate}.
  Then, inserting \eqref{nabla2_nablaw_estimate} into \eqref{nabla2_nablaw_estimate-before} we derive that
  \begin{equation*}
    \frac{c}{2 \aC} \chi_{\eta, \lp}''(\wst k) \abs{\nabla \wst k}^2
    \leq C \Big(
      \chi_{\eta, \lp}''(\wst k) \CCC \Kst{k-1} \BBB \nabla \tst k \cdot \nabla \wst k
      + \wst k + \abs{\nabla^2 \yst k}^p
    \Big).
  \end{equation*}
  Integrating the above inequality over $\Omega$, multiplying by $\tau$, and  summing over $k = 1, \ldots,  {T/\tau}$ we derive by  \eqref{nablaw_nablatheta_bound}, \eqref{a_priori_Linfty_W2p_lin}, and \eqref{a_priori_temp_L1_lin} \BBB that
  \begin{equation*}
    \sum_{k=1}^{T/\tau} \tau
      \int_\Omega \chi_{\eta, \lp}''(\wst k) \abs{\nabla \wst k }^2 \di x
     \leq  C(1+T) \lp^\alpha. \BBB
  \end{equation*}
  where in the final step we used $\alpha \leq 2$.   By using the first identity in \eqref{elementary}, we conclude the proof of \eqref{weighted_L2_bound}.
\end{proof}

\begin{theorem}[Further a priori bounds on the temperature]\label{thm:further_apriori_temp_bounds}
  For any $q \in [1,  \frac{d+2}{d}\BBB)$ and $r \in [1, \frac{d+2}{d+1})$ there exist constants $C_q$ and $C_r$, respectively, both independent of $\lp$ and $\tau$ such that
  \begin{align}
    \sum_{k=0}^{T/\tau} \tau
      \int_\Omega \big( \abs{\tst k}^q + \abs{\wst k}^q \big) \di x
      &\leq C_q\lp^{\alpha q}, \label{temp_inten_Lq_bound} \\
    \sum_{k=1}^{T/\tau} \tau
      \int_\Omega \big( \abs{\nabla \tst k}^r + \abs{\nabla \wst k}^r \big) \di x
      &\leq C_r \lp^{\alpha r}. \label{nablatemp_nablainten_Lr_bound}
  \end{align}
  Moreover,  we can find a constant $C$ independent of $\lp$ and $\tau$ such that
  \begin{equation}\label{dot_temp_apriori_bound}
    \sum_{k = 1}^{T/\tau} \tau \norm{\ddif \wst k}_{W^{1, \infty}(\Omega)^*}
      \leq C \lp^\alpha.
  \end{equation}
\end{theorem}

\begin{proof}
  Let $q, \, r$ be as in the statement.
  As $\wst k \in H^1(\Omega)$ (see \eqref{thermal_step}), it follows that $\norm{\nw}_{L^\infty(I; H^1(\Omega))} < \infty$.
  Therefore, by using the a priori estimate $\Vert 1 + \lp^{-\alpha} \nw \Vert_{L^\infty(I;L^1(\Omega))} \le    C +  \mathcal{L}^d(\Omega) \BBB $ (see \eqref{a_priori_temp_L1_lin}) as well as  Lemma \ref{lem:weightet_L2_bound}, we can repeat the argument from the proof of \cite[Proposition~6.3,  equation (6.6)\BBB]{MielkeRoubicek20Thermoviscoelasticity} for $\lp^{-\alpha} \nw$ in place of $w_\lp$,  cf.\ also Remark \ref{rem: next a priori} below. \BBB 
  This gives the existence of constants $C_q, \, C_r$ such that
  \begin{align}\label{apriori_bounds_inten}
    \sum_{k=0}^{T/\tau} \tau \int_\Omega \abs{\wst k}^q \di x
      &\leq C_q \lp^{\alpha q}, &
    \sum_{k=1}^{T/\tau} \tau \int_\Omega \abs{\nabla \wst k}^r \di x
      &\leq C_r \lp^{\alpha p}.
  \end{align}
  By \eqref{inten_lipschitz_bounds} we then directly see that (for a possibly larger $C_q$)
  \begin{equation}\label{Cq}
    \sum_{k=0}^{T/\tau} \tau \int_\Omega \abs{\tst k}^q \di x \leq C_q \lp^{\alpha q}.
  \end{equation}
  To conclude the proof of \eqref{temp_inten_Lq_bound}--\eqref{nablatemp_nablainten_Lr_bound}, it remains to control the gradient of the temperature.
  Employing the relation between $\nabla \wst k$ and $\nabla \tst k$ in (\ref{tempgrad_relation}),   by \BBB \eqref{C_locally_lipschitz}  and \ref{C_bounds} we see that
  \begin{equation*}
    \abs{\nabla\tst k}
    \leq C \big(
      \abs{\nabla \wst k}
      + (\tst k \wedge 1) \abs{\nabla^2 \yst k}
    \big).
  \end{equation*}
  Consequently, using $t \wedge 1 \leq t^{\frac{p-1}{p}}$ for all $t \geq 0$ and Young's inequality with power\CCC s \BBB $p/(p-r)$ \CCC and \BBB $p/r$ we derive that
  \begin{align}
    \int_\Omega \abs{\nabla \tst k }^r \di x
    &\leq C \int_\Omega \abs{\nabla \wst k}^r \di x
      + C \lp^{\alpha r} \int_\Omega (\lp^{-\alpha} \tst k)^{r\frac{p-1}{p}}
        \abs{\lp^{-\frac{ \alpha \BBB}{p}} \nabla^2 \yst k}^r \di x \nonumber \\
    &\leq C \int_\Omega \abs{\nabla \wst k}^r \di x
       + C \lp^{\alpha r} \int_\Omega \big( (\lp^{-\alpha} \tst k)^{r\frac{p-1}{p-r}} 
      +  \frac{1}{\lp^\alpha}  \abs{\nabla^2 \yst k}^p \big) \di x. \BBB \label{nablathetar_bound}
  \end{align}
  As $r$ was chosen strictly smaller than $\frac{d+2}{d+1}$, we see by $p \geq 2$ that
  \begin{equation*}
      r\frac{p-1}{p-r}
      < \frac{d+2}{d+1} \frac{p-1}{p-\frac{d+2}{d+1}}
      = \frac{d+2}{d}  \frac{1}{1 + \frac{p-2}{d(p-1)}}
      \leq \frac{d+2}{d}.
  \end{equation*}
  Consequently, multiplying (\ref{nablathetar_bound}) with $\tau$, summing over $k=1, \ldots,  {T/\tau}$, and using \eqref{a_priori_Linfty_W2p_lin}, \eqref{apriori_bounds_inten}, and \eqref{Cq} we conclude the proof of (\ref{nablatemp_nablainten_Lr_bound}).  Here, we again used $\alpha \le 2$. \BBB

  Lastly, we show \eqref{dot_temp_apriori_bound}.
  Testing (\ref{el_thermal_step}) for the $k$-th step with arbitrary $\vphi \in W^{1, \infty}(\Omega)$, and using the shorthand notation for $h_{\eps,\tau}^{(k)}$ and $\CCC \Kst{k-1} \BBB$ from \eqref{h_tau_hcm_short}, we see by (\ref{bound_hcm}) and the continuity of the trace operator in $W^{1, 1}(\Omega)$ that
  \begin{align*}
    \Big| \int_\Omega \ddif \wst k \vphi \di x \Big|
    &= \Big|
      \int_\Omega h_{\eps,\tau}^{(k)} \vphi \di x
      -\int_\Omega \CCC \Kst{k-1} \BBB \nabla \tst k \cdot \nabla \vphi \di x
      -\kappa \int_\Gamma (\tst k - \lp^\alpha \btst k) \vphi \di \haus^{d-1}
    \Big| \\
    &\leq \norm{h_{\eps,\tau}^{(k)}}_{L^1(\Omega)} \norm{\vphi}_{L^\infty(\Omega)}
      + C \norm{\nabla \tst k}_{L^1(\Omega)} \norm{\nabla \vphi}_{L^\infty(\Omega)} \\
    &\phantom{\leq}\quad + \Big(
        C \kappa \norm{\tst k}_{W^{1, 1}(\Omega)}
        + \kappa \lp^\alpha \int_\Gamma \btst k \di \haus^{d-1}
      \Big) \norm{\vphi}_{L^\infty(\Omega)} \\
    &\leq \Big(
        \norm{h_{\eps,\tau}^{(k)}}_{L^1(\Omega)}
        + C \norm{\tst k }_{W^{1, 1}(\Omega)}
        + C \NNN  \lp^\alpha \int_\Gamma \btst k \di \haus^{d-1} \BBB
      \Big) \norm{\vphi}_{W^{1, \infty}(\Omega)}.
  \end{align*}
  By the arbitrariness of $\vphi$ this shows that
  \begin{equation}\label{W-1infty_bound}
    \norm{\ddif \wst k}_{W^{1, \infty}(\Omega)^*}
    \leq \norm{h_\tau^{(k)}}_{L^1(\Omega)} + C \norm{\tst k}_{W^{1, 1}(\Omega)} + C \NNN  \lp^\alpha \int_\Gamma \btst k \di \haus^{d-1}. \BBB
  \end{equation}
  We have already seen in the proof of Lemma \ref{lem:weightet_L2_bound} (see in particular \eqref{fk_bound}) that
  \begin{equation*}
    \sum_{k = 1}^{T/\tau} \tau \norm{h_{\eps,\tau}^{(k)}}_{L^1(\Omega)}
    \leq C \lp^\alpha.
  \end{equation*}
  Consequently, by \eqref{temp_inten_Lq_bound}--\eqref{nablatemp_nablainten_Lr_bound} for $q = r = 1$ \NNN and \eqref{W-1infty_bound} \BBB the desired bound \eqref{dot_temp_apriori_bound} follows.
\end{proof}

\begin{remark}\label{rem: next a priori}
{\normalfont 
 
For $\alpha \in [1, 2)$, by means of Remark \ref{rem: case1} \BBB we obtain a stronger bound on the temperature: given  $q = \frac{2}{\alpha} + \frac{4}{\alpha d}$ and $r = 2\frac{d+2}{\alpha d + 2}$, we can find \CCC a constant $C$ \BBB independent of $\lp$ and $\tau$ such that
\begin{align}\label{temp_bounds_using_reg}
 \NNN  \sum_{k=1}^{T/\tau}  \BBB \tau \int_\Omega \abs{\tst k}^q \di x
      &\leq \CCC C \BBB \lp^{\alpha q}, &
    \sum_{k=1}^{T/\tau} \tau \int_\Omega \abs{\nabla \tst k}^r \di x
      &\leq \CCC C \BBB \lp^{\alpha r}.
\end{align}
This can be seen as follows:   We start with the second bound.
In this regard, by a 
For $\alpha = 1$, this directly follows from \eqref{weighted_nabla_temp_bound_reg}, \NNN where we recall $\overline{\mu}_{\lp, \tau} = \lp^{-\alpha} \nt$.  Let $\alpha \in (1,2)$. Note that $r \in [1, 2)$ and let $m \defas r (1 - \frac{1}{\alpha})$\BBB.
\CCC Employing a standard truncation and approximation argument we can assume, without loss of generality, that $\overline\mu_{\lp, \tau} \in L^\infty(I \times \Omega)$. \BBB
\CCC Then, by \BBB \eqref{weighted_nabla_temp_bound_reg} and Hölder's inequality with powers $\frac{2}{2-r}$ and $\frac{2}{r}$ \CCC we derive that \BBB
\begin{align}\label{nabla_mu_Lr_bound}
  \lVert \nabla \overline{\mu}_{\lp, \tau} \rVert_{L^r(I \times \Omega)}^r
  &= \int_0^T \int_\Omega (1 + \overline{\mu}_{\lp, \tau})^m \frac{\abs{\nabla \overline{\mu}_{\lp, \tau}}^r}{(1 + \overline{\mu}_{\lp, \tau})^m} \di x \di t  \\
  &\leq \lVert 1 + \overline{\mu}_{\lp, \tau} \rVert_{L^{\frac{2m}{2-r}}(I \times \Omega)}^m
    \left(\int_0^{\NNN T} \int_\Omega \frac{\abs{\nabla \overline{\mu}_{\lp, \tau}}^2}{(1 + \overline{\mu}_{\lp, \tau})^{2(1 - \frac{1}{\alpha})}} \di x \di t \right)^{\frac{r}{2}}
  \leq C \lVert 1 + \overline{\mu}_{\lp, \tau} \rVert_{L^{\frac{2m}{2-r}}(I \times \Omega)}^m.   \notag
\end{align}
With $r  = 2\frac{d+2}{\alpha d + 2}= 2 - 2\frac{\alpha d}{\alpha d + 2} (1 - \frac{1}{\alpha})$  we can use the anisotropic Gagliardo-Nirenberg interpolation inequality (see e.g.~  \cite[Lemma 4.2]{MielkeNaumann}) \BBB with $\theta = \frac{\alpha d}{\alpha d + 2}$, $s = p = \frac{r}{\theta}$,  $s_1 = \infty$, \BBB $s_2 = p_2 = r$, and $p_1 = \frac{2}{\alpha}$ to get
\begin{equation}   \label{eq:  why again}                    
  \lVert 1 + \overline{\mu}_{\lp, \tau} \rVert_{L^{\frac{2m}{2-r}}(I \times \Omega)}
  \leq C \lVert 1 + \overline{\mu}_{\lp, \tau} \rVert^{\frac{2}{\alpha d + 2}}_{L^\infty(I; L^{\CCC\frac{2}{\alpha}\BBB}(\Omega))}
    \left( \lVert 1 + \overline{\mu}_{\lp, \tau} \rVert_{L^\infty(I; L^{\CCC \frac{2}{\alpha}\BBB}(\Omega))} + \lVert \nabla \overline{\mu}_{\lp, \tau} \rVert_{L^r(I \times \Omega)} \right)^{\frac{\alpha d}{\alpha d + 2}},
\end{equation}
\NNN where we use $\frac{r}{\theta} = \frac{2m}{2-r}$. \BBB Notice that by  \eqref{eq: rescaled toten} and the energy bound in Lemma \ref{lemma: first a prioiro} we have \BBB that  $\lVert 1 + \overline{\mu}_{\lp, \tau} \rVert_{L^\infty(I; L^{2/\alpha}(\Omega))}$   is uniformly bounded in $\lp$ and $\tau$.
Hence, with \eqref{nabla_mu_Lr_bound}  and  $m\frac{\alpha d}{\alpha d + 2} = \frac{(\alpha - 1)d}{\alpha d + 2} r$ \BBB we derive that
\begin{equation*}
  \lVert \nabla \overline{\mu}_{\lp, \tau} \rVert_{L^r(I \times \Omega)}^r
  \leq C(1 + \lVert \nabla \overline{\mu}_{\lp, \tau} \rVert_{L^r(I \times \Omega)}^{\frac{(\alpha - 1)d}{\alpha d + 2} r}).
\end{equation*}
As $\frac{(\alpha - 1)d}{\alpha d + 2} < 1$, this shows the second bound in \eqref{temp_bounds_using_reg} for the case $\alpha \in (1, 2)$.
The first estimate in \eqref{temp_bounds_using_reg} then follows from the second one  and \eqref{eq:  why again}, where we use that $\frac{2m}{2-r} \NNN = \frac{r}{\theta} \BBB = q$. \BBB 
}
\end{remark}

\section{Existence of solutions in the nonlinear setting}\label{sec:tau_to_zero_delta_fixed}

In this section we pass from time-discrete to time-continuous solutions by letting $\tau \to 0$ and establish  Proposition \ref{cor:van_tau_reg}(ii). \BBB
Notice that for the special case $\alpha = 2$ and $\lp = 1$ this will lead to  Theorem \ref{thm:van_tau}(ii). \BBB
For the deformation and the momentum balance we can closely follow \cite[Section 5]{MielkeRoubicek20Thermoviscoelasticity}, and therefore proofs are omitted or sketched only.
For the limit passage in the heat equation, however, our arguments are different as we work without regularization terms, cf.~Remark \ref{rem:difference_to_mielke_in_steps}.
We first use the a priori estimates on the interpolants in order to extract convergent subsequences.
Afterwards, we pass to the limit in the discretized weak forms of the momentum balance and the heat equation.
Here, the most delicate term is the dissipation rate $\xi$ which is quadratic in $\dot F$. Therefore,  strong convergence in $L^2(I; H^1(\Omega))$ for the strain rates  is  required.  \BBB 

 As before, we assume for simplicity that $T/\tau \in \N$. Moreover, \BBB without further notice, we suppose from now on that $\tau \in (0, \tau_0)$  and $\lp \in (0,\lp_0\CCC ]\BBB$, where $\tau_0$ and $\lp_0 = \lp_0(\alpha)$ \BBB are chosen such that all statements from Subsections \ref{sec:single_step_well_defined}--\ref{sec: a priori} are satisfied.  In particular, $\lp_0=1$ for $\alpha=2$. The corresponding \BBB time-discrete solutions are denoted by  $\yst 0, \ldots,  \yst{T/\tau} \in \Wid$ and $\tst 0, \ldots,  \tst{T/\tau} \in L^2_+(\Omega)$. We recall the definition of the interpolations in \eqref{y_interpolations} and employ similar notation for $\nt$, $\pt$, and $\at$, as well as $\nw$, $\pw$, and $\aw$.
All generic constants $C > 0$ are always assumed to be independent of $\tau$ and $\lp$.

We start with the convergence of the deformations under vanishing time-discretization.
\begin{lemma}[Convergence of deformations]\label{lem:uniform_def_conv}
  For each $\lp \in (0, \lp_0]$, \BBB we can find $y_\lp  \in L^\infty(I; \Wid) \cap \CCC H^1\BBB(I; H^1(\Omega; \R^d))$ with $y_\lp(0, \cdot) =  y_{0,\lp} \BBB $ such that, up to a subsequence (not relabeled), it holds that
  \begin{subequations}\label{uniform_def_conv}
  \begin{align}
     \ay &\weaklystar y_\lp \text{ weakly* in } L^\infty(I; \Wid) & &\text{and} &
     \ay &\weakly y_\lp \text{ weakly in } \CCC H^1\BBB(I;H^1(\Omega;\R^d)), \label{uniform_def_conv1} \\
    \nabla \ay &\to \nabla y_\lp \text{ in } L^\infty(I; L^\infty(\Omega;\R^{d \times d}))\label{uniform_def_conv2}
  \end{align}
    \end{subequations}
  as $\tau \to 0$.
  In the first convergence of \eqref{uniform_def_conv1}, and in \eqref{uniform_def_conv2}, the same holds true if we replace $\ay$ by $\py$ or $\ny$.
\end{lemma}

\begin{proof}
First, \eqref{uniform_def_conv1} follows from the a priori estimates  \eqref{a_priori_Linfty_W2p_lin}, \eqref{a_priori_strain_rates_lin} \BBB and by Banach's selection principle.
For \eqref{uniform_def_conv2}, one uses the embedding $W^{2,p}(\Omega; \R^d) \subset C^{1, 1 - \frac{d}{p}}(\Omega; \R^d)$ to obtain a H\"older estimate in space and \eqref{a_priori_strain_rates_lin} for a H\"older estimate in time.
Then, by an interpolation estimate one can show that the sequence is bounded in $C^\gamma(I; C^{1, \gamma}(\Omega; \R^d))$ for some $\gamma > 0$, and the uniform convergence of the gradients follows then from the Arzel\`a-Ascoli theorem.
We refer to \cite[Proof of Proposition 5.1, Step 1]{MielkeRoubicek20Thermoviscoelasticity} for more details.
To conclude that the first convergences in \eqref{uniform_def_conv1} and \eqref{uniform_def_conv2} also hold for $\py$ or $\ny$ one again uses {\eqref{a_priori_Linfty_H1_lin}--\eqref{a_priori_strain_rates_lin}} to see $\Vert \nabla \ay - \nabla \ny \Vert_{L^\infty(I; L^2(\Omega))} \le C\tau^{\frac{1}{2}}$.
\end{proof}

We proceed with the convergence of the temperatures.

\begin{lemma}[Convergence of temperatures]\label{lem:pointwise_temp_conv}
  For each $\lp \in (0, \lp_0]$, \BBB there exists $\theta_\lp \in L^1(I; W^{1,1}(\Omega))$ \NNN with $\theta_\lp\ge 0$ a.e.\ \BBB  such that, up to a subsequence (not relabeled), it holds that
  \begin{subequations}\label{pointwise_temp_conv}
    \begin{align}
      \nt &\rightharpoonup \theta_\lp & &\text{and} &
      \nw &\rightharpoonup w_\lp &
      &\text{weakly in } L^r(I; W^{1,r}(\Omega)) \text{ for any } r \in [1, \tfrac{d+2}{d+1}), \label{pointwise_temp_conv1} \\
      \at &\to \theta_\lp & &\text{and} & \aw &\to w_\lp &
      &\text{in } \CCC L^s(I \times \Omega) \BBB \text{ for any } s \in [1,  \tfrac{d+2}{d}), \label{pointwise_temp_conv2}
    \end{align}
  \end{subequations}
  as $\tau \to 0$ where $w_\lp \defas \inten(\nabla y_\lp, \theta_\lp)$ for $y_\lp$ as in Lemma \ref{lem:uniform_def_conv}.
  In \eqref{pointwise_temp_conv2}, the same holds true if we replace $\at$ with $\pt$ or $\nt$ and $\aw$ with $\pw$ or $\nw$, respectively.
\end{lemma}

\begin{proof}
  The existence of the limit and  the  convergences \BBB in \eqref{pointwise_temp_conv1} follow from the a priori bounds in Theorem \ref{thm:further_apriori_temp_bounds} together with Banach's selection principle.

  Let $t_0 \CCC \in (0, T)\BBB$ and $r \in [1, \frac{d+2}{d+1})$. By Theorem \ref{thm:further_apriori_temp_bounds}, $(\aw)_\tau$ is bounded in
  \begin{equation*}
    L^r([t_0,T]; W^{1, r}(\Omega))  \cap W^{1, 1}([t_0,T]; W^{1, \infty}(\Omega)^*).
  \end{equation*}
  Hence, for any $\tilde r < r^* \defas \frac{rd}{d-r}$, due to the compact embedding $W^{1,r}(\Omega) \subset \subset L^{\tilde r}(\Omega)$, the Aubin-Lions' theorem shows that there exists $\hat{w}_\lp \in L^r([t_0, T]; L^{\tilde r}(\Omega))$ such that $(\aw)_\tau \to \hat{w}_\lp$ in $L^r([t_0, T]; L^{\tilde r}(\Omega))$, up to taking a subsequence.
  We observe that $\hat{w}_\lp = w_\lp$.
  Indeed, it is elementary to check that by \eqref{dot_temp_apriori_bound}
  \begin{equation}\label{nmuk_pmuk_Ls_conv3-toproof}
    \norm{\aw - \nw}_{L^1(I; W^{1, \infty}(\Omega)^*)}
    \le \norm{\nw - \pw}_{L^1(I; W^{1, \infty}(\Omega)^*)}
    \leq \tau \norm{\dotaw}_{L^1(I; W^{1, \infty}(\Omega)^*)} \to 0
  \end{equation}
  as $\tau \to 0$.
  Next, we show that the convergence $\CCC \aw \to w_\lp\BBB$ in $L^r([t_0,T]; L^{\tilde r}(\Omega))$ \CCC as $\tau \to 0$ \BBB can be improved to convergence in $L^s([t_0, T]; L^s(\Omega))$ for any exponent $s \in [1,  \frac{d+2}{d})$.
  To this end, we will interpolate with the bound
  \begin{equation}\label{Linfty_L1_wdelta_bound}
    \norm{w_\lp}_{L^\infty(I; L^1(\Omega))}
    \leq \sup_{\tau > 0} \, \norm{\nw}_{L^\infty(I; L^1(\Omega))} < \infty,
  \end{equation}
  which follows from \eqref{a_priori_temp_L1_lin}.
  Fix $s \in (1,  \frac{d+2}{d})$ and consider $r \in (1, \frac{d+2}{d+1}), \, \tilde r \in (1, r^*)$, both to be specified later.
  Now,  as \BBB $\lim_{r \to \frac{d+2}{d+1}} \frac{rd}{d - r} \NNN \ge \BBB  \frac{d+2}{d} > s$,  notice that \BBB for $r, \, \tilde r$ large enough it holds that $\lambda \defas \frac{\tilde r - s}{s(\tilde r - 1)} \in (0, 1)$.
  Writing $v_\tau \defas \aw - w_\lp$ for shorthand and using Hölder's inequality in the integral over $\Omega$ with power\CCC s \BBB $q_1 = \frac{\tilde r - 1}{\tilde r - s}$ and $q_1' = \frac{\tilde r - 1}{s - 1}$, we derive that
  \begin{align}
    \norm{v_\tau}^s_{L^s([t_0,T]; L^s(\Omega))}
    = \int_{t_0}^T \int_\Omega \abs{v_\tau}^{\lambda s} \abs{v_\tau}^{(1 - \lambda)s} \di x \di t
    = \int_{t_0}^T
      \CCC\Bigg(\BBB
        \int_\Omega \abs{v_\tau} \di x
      \CCC\Bigg)\BBB^{\frac{1}{q_1}}
      \CCC\Bigg(\BBB
        \int_\Omega \abs{v_\tau}^{\tilde r} \di x
      \CCC\Bigg)\BBB^{\frac{1}{q_1'}} \di t, \label{first_hoelder}
  \end{align}
  where we have used $\lambda s q_1 = 1$ and $(1 - \lambda) s q_1' = \tilde r$.
  Let $q_2 \defas \frac{r(\tilde r - 1)}{\tilde r (s - 1)}$ and notice that
  \begin{equation*}
    \lim_{r \to \frac{d+2}{d+1}\BBB} \lim_{\tilde r \to r^*} q_2
    = \lim_{r \to  \frac{d+2}{d+1} \BBB } \frac{r(d+1) - d}{d(s - 1)}
    = \frac{2}{d(s - 1)} > 1
  \end{equation*}
  where the last inequality is due to $s <  1 + \BBB \frac{2}{d}$.
  Hence, by possibly increasing $r$ and $\tilde r$ we can assure that $q_2 > 1$.
  We denote by $q_2'$ the conjugate of $q_2$.
  Consequently,  by \CCC$\aw \to \hat{w}_\lp$ \BBB in $L^r([t_0, T]; L^{\tilde r}(\Omega))$ \CCC as $\tau \to 0$\BBB, by \eqref{Linfty_L1_wdelta_bound}, and \BBB by H\"older's inequality in the integral in (\ref{first_hoelder}) over \NNN $[t_0,T]$ \BBB with powers $q_2'$ and $q_2$ we get
  \begin{align}\label{second_hoelder}
    \norm{v_\tau}^s_{L^s([t_0,T]; L^s(\Omega))}
    &\leq
    \Bigg(
      \int_{t_0}^T
        \Big(
          \int_\Omega \abs{v_\tau} \di x
        \Big)^{\CCC\frac{q_2'}{q_1}\BBB} \di t
    \Bigg)^{\frac{1}{q_2'}}
    \Bigg(
      \int_{t_0}^T
        \Big(
          \int_\Omega \abs{v_\tau}^{\tilde r} \di x
        \Big)^{\CCC\frac{r}{\tilde r}\BBB} \di t
    \Bigg)^{\frac{1}{q_2}} \notag \\
    &\leq \big(
      2\sup \nolimits_{\tau > 0} \norm{\nw}_{L^\infty(I; L^1(\Omega))}
    \big)^{\frac{1}{q_1}} \, \norm{\aw - w_\lp}^{\frac{r}{q_2}}_{L^r([t_0,T]; L^{\tilde r}(\Omega))}
    \to 0 \text{ as } \tau \to 0.
  \end{align}
 Sending $t_0 \to 0$ and using \eqref{temp_inten_Lq_bound}, \BBB this shows \eqref{pointwise_temp_conv2} for the sequence $(\aw)_\tau$.
  To obtain the same convergence for $\nw$ and $\pw$, we use a more general version of Aubin-Lions for time-derivatives as measures, see Corollary~7.9 in \cite{Roubicek13Nonlinear}.
  To this end it suffices to see that $\nw$ and $\pw$ are bounded in $L^r([t_0,T]; W^{1, r}(\Omega)) \cap BV([t_0,T]; W^{1, \infty}(\Omega)^*)$,  and then by repeating \eqref{first_hoelder}--\eqref{second_hoelder} we get \eqref{pointwise_temp_conv2} for $\nw$ and $\pw$, up to taking a subsequence.

  It remains to show \eqref{pointwise_temp_conv2} for the three different interpolations of the temperatures.
  In view of \eqref{sec_deriv}, for any $F \in \NNN GL^+(d) \BBB$, the map $\inten(F, \cdot)$ is invertible with $\CCC \frac{\di}{\di \theta} \BBB (\inten(F, \cdot)^{-1}) \leq \frac{1}{\ac}$.
  Thus, from the definition $\nw = \inten(\ny,\nt)$ we get $\nt = \inten(\nabla \ny, \cdot)^{-1}(\nw)$.
  \CCC Setting \BBB $\theta_\lp \defas \inten(\nabla  y_\lp, \cdot)^{-1}(w_\lp)$,   by \eqref{uniform_def_conv2} for $\ny$ \BBB and by $\nw \to w_\lp$ in \CCC $L^s(I \times \Omega)$ \BBB (see \eqref{pointwise_temp_conv2}), we get
  \begin{equation*}
  \nt =  \inten(\nabla \ny, \cdot)^{-1}(\nw) \to \inten(\nabla  y_\lp, \cdot)^{-1}(w_\lp) = \theta_\lp \quad \text{in } \CCC L^s(I \times \Omega)\BBB.
  \end{equation*}
  The convergence for $(\pt)_\tau$ follows in a similar fashion.
  Lastly, combining the convergence of $(\nt)_\tau$ and $(\pt)_\tau$ we obtain \eqref{pointwise_temp_conv2} also for $\at$.
\end{proof}

\begin{remark}\label{rem: after temp}   (i) \BBB   Note that \eqref{pointwise_temp_conv1} does not holds in general for $\at$, $\pt$, $\aw$, and $\pw$ as we did not assume Sobolev regularity for the initial datum $ \theta_{0,\lp} \BBB  \in L^2_+(\Omega)$. Yet, the statement could be obtained on any subinterval $I' \subset I$ with $0 \notin I'$. \\
(ii) The result only relies on the a priori bounds in \NNN Theorem \BBB \ref{thm:further_apriori_temp_bounds}.
Consequently, the same convergence result holds true for the \emph{rescaled temperature and rescaled internal energy}, namely along (interpolations of) the sequences $(\lp^{-\alpha}_k \theta^{(k)}_{\lp_k\CCC,\BBB \tau_k})_k$ and $(\lp^{-\alpha}_k w^{(k)}_{\lp_k\CCC,\BBB \tau_k})_k$ for sequences $(\lp_k\CCC,\BBB \tau_k)_k$ with $\lp_k \to 0$ as $k \to \infty$.
Namely, the proof of $\lp^{-\alpha}_k \overline{w}_{\lp_k\CCC,\BBB \tau_k} \to \tilde{w}$ in \CCC $L^s(I \times \Omega)$ \BBB for some $\tilde w$ is the same, taking the a priori bounds in \eqref{a_priori_temp_L1_lin} and Theorem \ref{thm:further_apriori_temp_bounds} into account.
In view of \ref{C_heatcap_cont}, \BBB $\bar c_V = c_V(\Id, 0)$ exists and by the third estimate in \ref{C_bounds}  we have \BBB $\bar c_V \geq \ac$.
Hence, we can define $\tilde \theta \defas  \tilde w / \bar c_V\BBB$.
Further\CCC more\BBB, by $\inten(F, 0) = 0$ for all $F \in GL^+(d)$, $c_V = \pl_\theta \inten$  (see \eqref{sec_deriv}), \BBB and the Fundamental Theorem of Calculus  we find \BBB
\begin{equation}\label{nmuk_pmuk_Ls_conv3-toproof2NNN}
  \overline{\theta}_{\lp_k\CCC,\BBB \tau_k}
    = \BBB \inten(\nabla \overline{y}_{\lp_k\CCC,\BBB \tau_k}, \cdot)^{-1} (\overline{w}_{\lp_k\CCC,\BBB \tau_k})
  = \int_0^{\overline{w}_{\lp_k\CCC,\BBB \tau_k}}
    c_V(\nabla \overline{y}_{\lp_k\CCC,\BBB \tau_k}, s)^{-1} \di s
  = \lp_k^\alpha \int_0^{\lp_k^{-\alpha} \overline{w}_{\lp_k\CCC,\BBB \tau_k}}
    c_V(\nabla \overline{y}_{\lp_k\CCC,\BBB \tau_k},  \lp_k^\alpha \BBB s)^{-1} \di s,
\end{equation}
where we changed coordinates in the last \NNN identity. \BBB
\CCC Consequently, using the third inequality in \ref{C_bounds} we can derive the following bound
\begin{equation*}
  |\lp_k^{-\alpha} \overline \theta_{\lp_k, \tau_k} - \bar c_V^{-1} \tilde w| = \Big|  \int_0^{\lp_k^{-\alpha} \overline{w}_{\lp_k\CCC,\BBB \tau_k}}
    c_V(\nabla \overline{y}_{\lp_k\CCC,\BBB \tau_k},  \lp_k^\alpha \BBB s)^{-1} \di s - \int_0^{\tilde w}  \bar c_V^{-1}\di s \Big|
  \leq \frac{1}{c_0} |\lp_k^{-\alpha} \overline w_{\lp_k, \tau_k} - \tilde w|
    + f_k,
\end{equation*}
where
\begin{equation*}
  f_k \defas \int_0^{\tilde w}
    |c_V(\nabla \overline{y}_{\lp_k, \tau_k},\eps_k^\alpha s)^{-1} - \bar c_V^{-1}| \di s.
\end{equation*}
It remains to show that $f_k \to 0$ in $L^s(I \times \Omega)$.
By the third bound in \ref{C_bounds} we see that $|f_k| \leq \frac{2}{c_0} \tilde w \in L^s(I \times \Omega)$.
Then, by \ref{C_heatcap_cont} and the definition of $\bar c_V$, it follows that $f_k \to 0$ a.e.~in $I \times \Omega$.
Dominated Convergence yields the desired result. \BBB
% Taking into account that $ c_V(\nabla \overline{y}_{\lp_k\CCC,\BBB \tau_k}, \lp_k^\alpha s) \BBB \to \bar c_V$ \NNN pointwise \BBB as $k \to \infty$  by \ref{C_heatcap_cont}, we get $\lp_k^{-\alpha} \overline{\theta}_{\lp_k\CCC,\BBB \tau_k} \to \tilde{\theta}$ in $L^s(I; L^s(\Omega))$.
The same argument holds for the other interpolations. \\
(iii) In the case $\alpha =1$, the convergence can be improved to $\lp_k^{-1} \underline{\theta}_{\lp_k\CCC,\BBB \tau_k} \to \tilde{\theta}$ in $L^2(I;L^2(\Omega))$.
Indeed, by Remark \ref{rem: next a priori}  and $\theta_{0,\lp} \in L^2_+(\Omega)$ we get 
$$\Vert \overline{\theta}_{\lp_k\CCC,\BBB \tau_k} \Vert_{L^2(I \times \Omega)} + \Vert \underline{\theta}_{\lp_k\CCC,\BBB \tau_k} \Vert_{L^2(I \times \Omega)} + \Vert \nabla \overline{\theta}_{\lp_k\CCC,\BBB \tau_k} \Vert_{L^2(I \times \Omega)} \le C\lp. $$
Then, the convergence in $L^2(I;L^2(\Omega))$ follows by repeating the argument above via Aubin-Lions' theorem, simply using \CCC the compact embedding \BBB \NNN $  H^1(\Omega) \subset \subset L^2(\Omega)$. \BBB
\end{remark}

We are ready to pass to the limit in the time-discrete mechanical evolution.
\begin{proposition}[Convergence of the mechanical equation]\label{thm:vanishing_tau_mech_nonlinear}
  Let $y_\lp$ be as in Lemma \ref{lem:uniform_def_conv} and $\theta_\lp$ as in Lemma \ref{lem:pointwise_temp_conv}.
  Then, for any test-function $z \in C^\infty(I \times \overline{\Omega})$ with $z = 0$ on $I \times \Gamma_D$ we have that \eqref{weak_limit_mechanical_equation} holds.
\end{proposition}

\begin{proof}
  The statement is proved in \cite[Proof of Proposition 5.1, Step 2]{MielkeRoubicek20Thermoviscoelasticity} and we include a sketch for the reader's convenience.
  For $y \in \Wid$ we define a functional on $X \defas W^{2,p}(\Omega; \R^d)$ by
  \begin{equation*}
    \langle \mathbf{H}(y), z\rangle = \intQ \pl_G \hypot(\nabla^2 y) \cdddot \nabla^2 z.
  \end{equation*}
  Note that $\mathbf{H}$ is a hemicontinuous and monotone operator as $\hypot$ is convex.
  We further choose $b_{\lp\tau}, \, b_\lp \in X^*$ such that  \eqref{mechanical_step_single} \BBB can be written as
  \begin{equation}\label{minty0}
    \langle \mathbf{H}(\ny),z\rangle = \langle b_{\lp\tau}, z \rangle
  \end{equation}
  for all $z \in \mathcal{Y}_0$ and \eqref{weak_limit_mechanical_equation} can be written as
  \begin{equation}\label{minty}
    \langle \mathbf{H}(y_\lp),z\rangle = \langle b_\lp,z\rangle
  \end{equation}
  for all $z \in \mathcal{Y}_0$.
  Note that \eqref{minty0} holds by Proposition \ref{prop:existence_mechanical_step}, and that our goal is to confirm \eqref{minty}.

  First, $b_{\lp\tau} \weaklystar b_\lp$ \CCC weakly* \BBB in $X^*$ for $\tau \to 0$ as in each of the three terms of $b_{\lp\tau}$ (i.e., $\partial_F W$, $\partial_{\dot{F}} R$, and  $\ell^{(k)}_\tau$, respectively, see  \eqref{mechanical_step_single}) \BBB one can pass to the limit by using weak convergence of $(\nabla \dotay)_\tau$ in $L^2(I; H^1(\Omega; \R^d))$ (see \eqref{uniform_def_conv1}), uniform convergence of $(\nabla \ny)_\tau, (\nabla \py)_\tau$ \CCC on \BBB $I \times \Omega$ (see \eqref{uniform_def_conv2}), and pointwise a.e.~convergence of $(\pt)_\tau$ \CCC on \BBB $I \times \Omega$ (up to a subsequence, see \eqref{pointwise_temp_conv2}).
  At this point, we use in particular that $\pl_{\dot F} \disspot$ is linear in $\nabla \dotay$ and that $\partial_F W(\ny,\pt)$ is bounded  due to \ref{W_regularity}, \eqref{C_locally_lipschitz}, and \eqref{a_priori_Linfty_W2p_lin}. \BBB   Moreover, due to uniform convergence of the gradients we also have $\langle b_{\lp \tau}, \ny \rangle \to \langle b_\lp, y_\lp \rangle $.
  We now use Minty's trick for the monotone operator $\mathbf{H}$: \BBB identity  \eqref{minty0} and the convergences $\ny \rightharpoonup y_\lp$ \CCC weakly \BBB in $X$, $b_{\lp\tau} \weaklystar  b_\lp$ \CCC weakly* \BBB in $X^*$, and $\langle b_{\lp\tau}, \ny \rangle \to \langle b_\lp, y_\lp \rangle $ imply $\mathbf{H}(y_\lp) = b_\lp$ as elements of $X^*$, i.e., \eqref{minty} holds.
\end{proof}

For the limit passage in the time-discrete heat equation, we will need the strong convergence of the strain rates $(\nabla \dotay)_\tau$ in $L^2(I; L^2(\Omega;\R^{d \times d}))$ since the dissipation rate $\drate(\nabla \py, \nabla \dotay, \pt)$ is quadratic in $\nabla \dotay$.
Note that our a priori bounds currently only guarantee weak convergence.
The next lemma improves this convergence:
\begin{lemma}[Strong convergence of the strain rates]\label{lem:strong_strain_rates_conv}
  For $y_\lp$ as in Lemma \ref{lem:uniform_def_conv}, we have that, up to taking a subsequence,
  \begin{equation}\label{strong_strain_rates_conv}
    \dotay \to \dot y_\lp \text{ strongly in } L^2(I; H^1(\Omega; \R^d)) \text{ as } \tau \to 0.
  \end{equation}
\end{lemma}

\begin{proof}
  The proof follows essentially by combining Steps 4 in the proof of \cite[Proposition~5.1, Proposition~6.4]{MielkeRoubicek20Thermoviscoelasticity}.
  We give the main steps here in our setting because we   work completely without regularization. \BBB
  First, in the time-continuous setting, one derives the energy balance
  \begin{equation}\label{cont_energy_balance}
  \begin{aligned}
    \mechen(y_\lp(T))
      + 2 \int_0^T \hspace{-0.2cm}  \diss \BBB (y_\lp, \dot y_\lp, \theta_\lp) \di t
    &=  \mechen \BBB ( y_{0,\lp} \BBB )
      + \lp \int_0^T \hspace{-0.2cm} \langle \ell(t), \dot y_\lp \rangle \di t  -\int_0^T \hspace{-0.15cm} \int_\Omega
        \pl_F \cplpot(\nabla y_\lp, \theta_\lp) : \nabla \dot y_\lp \di x \di t,
  \end{aligned}
  \end{equation}
  where we recall the notation in \NNN \eqref{mechanical}, \BBB \eqref{dissipation}, and \eqref{ell2}.
  This follows by testing the momentum balance \eqref{weak_limit_mechanical_equation} derived in Proposition \ref{thm:vanishing_tau_mech_nonlinear} with $\dot{y}_\lp \in L^2(I;H^1(\Omega))$,  employing \eqref{diss_rate}, \BBB and using a chain rule for the $\Lambda$-convex functional \NNN $\mathcal{M}$, \BBB see \cite[Proposition 3.6]{MielkeRoubicek20Thermoviscoelasticity}.
  Our next goal is to show a similar balance in the time-discrete setting.
  To this end, we test the Euler-Lagrange equation  \eqref{mechanical_step_single} \BBB of the $k$-th mechanical step with $\yst k - \yst{k-1}$ to get
  \begin{align}\label{energy_balance_discrete_stepk}
    &2 \tau \diss(\yst{k-1}, \ddif \yst k, \tst{k-1})  = \tau \lp \langle \lst k, \ddif \yst k \rangle
      - \tau \int_\Omega \pl_F \cplpot(\nabla \yst k, \tst{k-1})
        : \ddif \nabla \yst k \di x \notag \\
    &\phantom{\quad =}\quad -\int_\Omega \pl_G \hypot(\nabla^2 \yst k)
      \cdddot (\nabla^2 \yst k - \nabla^2 \yst{k-1})
      -\pl_F \elpot(\nabla \yst k) : (\nabla \yst k - \nabla \yst{k-1}) \di x.
  \end{align}
  By the $\Lambda$-convexity of  $\mathcal{M}$ \BBB derived in \cite[Proposition 3.2]{MielkeRoubicek20Thermoviscoelasticity}, we can find $\Lambda > 0$ depending on the energy bound  in Lemma \ref{lemma: first a prioiro} and the bound in \eqref{pos_det}  \BBB  but independent of $\lp$, $\tau$, and $k$ such that
  \begin{align*}
    \mechen(\yst {k-1})
    &\ge \mechen(\yst k)
      - \Lambda \norm{\nabla \yst{k-1} - \nabla \yst k}_{L^2(\Omega)}^2
      + \int_\Omega \pl_G \hypot(\nabla^2 \yst k)
        \cdddot (\nabla^2 \yst {k-1} - \nabla^2 \yst k) \di x \\
    &\phantom{\ge}\quad + \int_\Omega \pl_F \elpot(\nabla \yst k)
      : (\nabla \yst {k-1} - \nabla \yst k) \di x.
  \end{align*}
  Using this bound in \eqref{energy_balance_discrete_stepk} \CCC then \BBB leads to
  \begin{align*}
    &\mechen(\yst k) - \mechen(\yst{k-1})
      + 2\tau \diss(\yst{k-1}, \ddif \yst k, \tst{k-1})
      - \Lambda \tau^2 \norm{\ddif \nabla \yst k}^2_{L^2(\Omega)} \\
    &\quad \leq \tau \lp \langle \lst k, \ddif \yst k \rangle
      - \tau \int_\Omega \pl_F \cplpot(\nabla \yst k, \tst{k-1}) : \ddif \nabla \yst k \di x.
  \end{align*}
  Summing the above inequality over $k \in \setof{1, \ldots,  T/\tau}$ we arrive at a discrete analog of \eqref{cont_energy_balance}, namely,
  \begin{align}\label{disc_energy_balance}
    &\mechen(\ny(T))
      +2 \int_0^T \diss(\py, \dotay, \pt) \di t
      -\Lambda \tau \int_0^T \int_\Omega
        \abs{\nabla \dotay}^2 \di x \di t \notag \\
    &\quad \leq  \mechen \BBB ( y_{0,\lp} \BBB )
      +\lp \int_0^{T} \langle \ell(t), \dotay \rangle \di t
      - \int_{0}^{T} \int_\Omega \pl_F \cplpot(\nabla \ny, \pt) : \nabla \dotay \di x \di t,
  \end{align}
  where in the integral for the \CCC force \BBB terms we used the definition in \eqref{forces_mech_step}.
  Up to selecting a further subsequence, we can suppose that the convergences in Lemma \ref{lem:uniform_def_conv} and Lemma \ref{lem:pointwise_temp_conv} hold true, and that $\pt \to \theta_\lp$ pointwise a.e.~in $I \times \Omega$, $\dotay \weakly \dot y_\lp$ weakly in $L^2(I; H^1(\Omega; \CCC \R^d \BBB))$, and $\ny(T) \weakly y_\lp(T)$ weakly in $W^{2, p}(\Omega)$ as $\tau \to 0$.  
  This shows
  \begin{align}\label{two ini0}
    I_\lp^{(1)} &\defas \lim_{\tau \to 0} \Bigg(
        \lp \int_0^{T} \langle \ell(t), \dotay \rangle \di t
        - \int_{0}^{T} \int_\Omega \pl_F \cplpot(\nabla \ny, \pt)
          : \nabla \dotay \di x \di t
      \Bigg) \notag \\
    &= \lp \int_0^T \langle \ell(t), \dot y_\lp \rangle \di t
      - \intQ \pl_F \cplpot(\nabla y_\lp, \theta_\lp)
        : \nabla \dot y_\lp \di x \di t.
  \end{align}
  \CCC Setting
  \begin{align*}
    \dot C_{\lp \tau} &\defas (\nabla \dotay)^T \nabla \py
      + (\nabla \py)^T \nabla \dotay, &
    \dot C_\lp &\defas (\nabla \dot y_\lp)^T \nabla y_\lp
      + (\nabla y_\lp)^T \nabla \dot y_\lp
  \end{align*}
  we see by \eqref{uniform_def_conv} that $\dot C_{\lp \tau} \weakly \dot C_\lp$ weakly in $L^2(I \times \Omega; \R^{d \times d})$. Consequently, \BBB
  by the convexity of $\hypot$ and the fact that  $\mathcal{R}$ \BBB is convex in \CCC $\dot C =\dot F^T F + F^T \dot F$\BBB, standard lower semicontinuity arguments (see also \cite[Theorem 7.5]{FonsecaLeoni07Modern}) imply
  \begin{equation}\label{two ini}
  \begin{aligned}
    I^{(2)}_{\lp} &\defas \liminf_{\tau \to 0} \mechen(\ny(T))
      \geq \mechen(y_\lp(T)), \\
    I^{(3)}_{\lp} &\defas \liminf_{\tau \to 0} \int_0^T
      \diss(\py, \dotay, \pt) \di t
      \geq \int_0^T \diss(y_\lp, \dot y_\lp, \theta_\lp) \di t.
  \end{aligned}
  \end{equation}
  Combining \eqref{cont_energy_balance}, \eqref{disc_energy_balance}, \eqref{two ini0}, and \eqref{two ini}, and using that $\lim_{\tau \to 0} \tau \int_{0}^{T} \int_\Omega \abs{\nabla \dotay}^2 \di x \di t = 0$ we get
  \begin{equation*}
    \mechen(y_\lp(T))
      +2 \int_0^T \hspace{-0.2cm} \diss(y_\lp, \dot y_\lp, \theta_\lp) \di t
    =  \mechen \BBB ( y_{0,\lp} \BBB ) + I_{\lp}^{(1)}
    \ge I_{\lp}^{(2)} + 2I_{\lp}^{(3)}
    \ge \mechen(y_\lp(T)) + 2 \int_0^T \hspace{-0.2cm}
      \diss(y_\lp, \dot y_\lp, \theta_\lp) \di t,
  \end{equation*}
  and thus both inequalities in \eqref{two ini} are actually equalities.  \CCC Consequently\BBB, \BBB we get by \eqref{dissipation} and \ref{D_quadratic} that
  \begin{equation}\label{diss_convergence}
    \intQ D(C_{\lp \tau}, \pt) \, \dot C_{\lp \tau}
      : \dot C_{\lp \tau} \di x \di t
    \to \intQ D(C_{\lp }, \theta_\lp) \, \dot C_\lp : \dot C_\lp \di x \di t\CCC,\BBB
  \end{equation}
  where we shortly write \CCC $C_{\lp \tau} \defas (\nabla \py)^T \nabla \py$ and $C_\lp \defas (\nabla y_\lp)^T \nabla y_\lp$\BBB.
  Based on this, we show the strong convergence of the strain rates.
  By \ref{D_bounds} \CCC it follows that \BBB
  \begin{align*}
    \ac \intQ \abs{\dot C_{\lp \tau} - \dot C_\lp}^2 \di x \di t
    &\leq \intQ D(C_{\lp \tau}, \pt) (\dot C_{\lp \tau} - \dot C_\lp)
      : (\dot C_{\lp \tau} - \dot C_\lp) \di x \di t \\
    &= \intQ D(C_{\lp \tau}, \pt) \, \dot C_{\lp \tau}
      : \dot C_{\lp \tau} \di x \di t
      -2 \intQ D(C_{\lp \tau}, \pt) \, \dot C_\lp : \dot C_{\lp \tau} \di x \di t \\
    &\phantom{=}\quad +\intQ D(C_{\lp \tau}, \pt) \, \dot C_\lp
      : \dot C_\lp \di x \di t.
  \end{align*}
  By a weak-strong convergence argument \CCC and \BBB \eqref{uniform_def_conv} we get that $\dot C_{\lp \tau} \rightharpoonup \dot C_\lp$ weakly in $L^2(I; L^2(\Omega; \R^{d \times d} \BBB ))$.
  Moreover, \CCC by \BBB \ref{D_bounds}, $D(C_{\lp \tau}, \pt)$ is uniformly bounded and  $D(C_{\lp \tau}, \pt) \dot C_\lp$ converges to $D(C_{\lp}, \theta_\lp) \, \dot C_\lp$ strongly in $L^2(\Omega;\R^{d \times d})$.
  Thus, \eqref{diss_convergence} and Dominated Convergence imply that
  \begin{equation}\label{simmi}
    \lim_{\tau \to 0} \Vert \dot C_{\lp \tau} - \dot C_\lp \Vert_{L^2(I \times \Omega)} = 0.
  \end{equation}
  It remains to show that $\nabla \dotay \to \nabla \dot y_\lp$ strongly in $L^2(I; L^2(\Omega; \R^{ d \times d}))$ as then \eqref{strong_strain_rates_conv} follows from Poincar\'e's inequality.
  By the uniform bound on the energy  in Lemma \ref{lemma: first a prioiro}, \BBB we can apply the generalized Korn's inequality stated in Lemma \ref{lem:gen_korn} for a constant $c$ depending only on the initial data and $f$, $g$, $\bt$,  and $T$. \BBB
  This shows 
  \begin{align*}
    c \Vert \nabla \dotay - \nabla \dot y_\lp \Vert_{L^2( I \times \Omega)}
    &\le \Vert (\nabla \dotay - \nabla \dot y_\lp)^T \nabla y_\lp
      + (\nabla y_\lp)^T (\nabla \dotay - \nabla \dot y_\lp) \Vert_{L^2( I \times \Omega)} \\
    &\le \Vert
      (\nabla \dotay)^T \nabla \py
      + (\nabla \py)^T \nabla \dotay
      - (\nabla \dot{y}_\lp)^T  \nabla y_\lp
      - (\nabla y_\lp)^T \nabla \dot{y}_\lp
    \Vert_{L^2(I \times \Omega)} \\
    &\phantom{\leq}\quad + 2\Vert \nabla \dotay \Vert_{L^2( I \times \Omega)}
      \Vert \nabla \py - \nabla y_\lp \Vert_{L^\infty( I \times \Omega)}.
  \end{align*}
  Now, \eqref{uniform_def_conv2}, \eqref{simmi}, and   $\sup_{\tau > 0} \Vert \nabla \dotay \Vert_{L^2(I \times \Omega)} < +\infty$ by \eqref{uniform_def_conv1} show $\Vert \nabla \dotay - \nabla \dot y_\lp \Vert_{L^2(I \times \Omega)} \to 0$ as $\tau \to 0$.
  This concludes the proof.
\end{proof}

The last step in the proof of  Theorem \ref{thm:van_tau}(ii) and Proposition \ref{cor:van_tau_reg}(ii) \BBB consists in passing to the limit of the thermal evolution.

\begin{proposition}[Convergence of the heat-transfer equation]\label{thm:convergence_heat_vanishing_tau}
  Let $y_\lp$ be as in Lemma \ref{lem:uniform_def_conv} and $\theta_\lp$ as in Lemma \ref{lem:pointwise_temp_conv}.
  Then, for any test-function $\varphi \in C^\infty(I \times \overline \Omega)$ with $\varphi(T) = 0$,  we have that $(y_\lp, \theta_\lp)$ satisfies \eqref{weak_limit_heat_equation}   with $\rdrate$ in place of $\drate$. \BBB
\end{proposition}

\begin{proof}
  Suppose that we have already selected a subsequence such that Lemma \ref{lem:uniform_def_conv} and Lemma \ref{pointwise_temp_conv} apply.
  By possibly taking a further subsequence we can also assume that $\pt \to \theta_\lp$ pointwise a.e.~in $I \times \Omega$.
  Further\CCC more\BBB, let $\vphi$ as in the \NNN statement. \BBB
  Summing the Euler-Lagrange equation    \eqref{el_thermal_step} (for $\rdrate$ in place of $\drate$) \BBB for each step and integrating by parts we get
  \begin{align}\label{ea: just copy}
    &\int_0^{T} \int_\Omega \hcm(\nabla \py, \pt) \nabla \nt
      \cdot \nabla \vphi \di x \di t
      +\kappa \int_0^{T} \int_{\Gamma} \nt \varphi \di \haus^{d-1} \di t \notag \\
    &\quad -\int_0^{T} \int_\Omega \CCC\Big(\BBB
        \rdrate(\nabla \py, \nabla \dotay, \pt)
        + \pl_F \cplpot( \nabla \py, \BBB \pt) : \nabla \dotay
      \CCC\Big)\BBB \vphi \di x \di t
      -\int_0^{T} \int_\Omega \aw \dot \vphi \di x \di t \notag \\
    &= \kappa \lp^\alpha \int_0^{T} \int_{\Gamma} \overline \theta_{\flat, \tau} \vphi \di \haus^{d-1} \di t
    + \int_\Omega \inten(\nabla  y_{0,\lp} \BBB ,  \theta_{0,\lp} \BBB ) \vphi(0) \di x,
  \end{align}
  where $\overline \theta_{\flat, \tau}(t) \defas  \theta_{\flat, \tau}^{(k)}\BBB$ for $t \in ((k-1)\tau, k\tau]$ and $k \in \setof{1, \ldots,  T/\tau}$.   As $\theta_\flat \in W^{1, 1}(I; L^2(\Gamma))$ we see $\norm{\overline \theta_{\flat, \tau} - \theta_\flat}_{L^1( I \BBB ; L^1(\Gamma))} \leq \tau \norm{\dot \theta_\flat}_{L^1(I; L^2(\Gamma))}$.   Consequently,
  \begin{equation}\label{ea: just copy-after}
    \int_0^{T} \int_{\Gamma} \overline \theta_{\flat, \tau} \vphi \di \haus^{d-1} \di t
    \to \int_0^T \int_{\Gamma}  \theta_\flat \BBB \vphi \di \haus^{d-1} \di t \text{ as } \tau \to 0.
      \end{equation}
      It  thus \BBB remains to show that the left-hand side of the above equality converges towards the left-hand side of   \eqref{weak_limit_heat_equation} (with $\rdrate$ \CCC in place of $\drate$\BBB) \BBB as $\tau \to 0$.
  By Lemma \ref{lem:bound_hcm} and our choice of $\vphi$ we have $\abs{\hcm(\nabla \py, \pt) \nabla \vphi} \leq C \abs{\nabla \vphi}$ a.e.~in $I \times \Omega$.
  Consequently, by the weak convergence of $(\nt)_\tau$ in $L^r(I; W^{1, r}(\Omega))$, \NNN see  \eqref{pointwise_temp_conv}, \BBB it follows \CCC that \BBB
  \begin{align*}
    &\int_0^{T} \int_\Omega \hcm(\nabla \py, \pt) \nabla \nt
      \cdot \nabla \vphi \di x
    + \kappa \int_0^{T}   \int_{\Gamma}
      \nt \vphi \di \haus^{d-1} \di t \\
    &\quad\to \intQ \hcm(\nabla y_\lp, \theta_\lp) \nabla \theta_\lp \cdot \nabla \vphi \di x
    + \kappa \int_0^T \int_{\Gamma} \theta_\lp \vphi \di \haus^{d-1} \di t.
  \end{align*}
  The strong convergence of $(\aw)_\tau$  in \CCC $L^s(I \times \Omega)$ \BBB for some $s \in (1,  \frac{d+2}{d})$, see \eqref{pointwise_temp_conv2}, leads to
  \begin{equation*}
    - \int_0^{T} \int_\Omega \aw \dot \vphi \di x \di t \to - \intQ w_\lp \dot \vphi \di x \di t = - \intQ \inten (\nabla y_\lp,\theta_\lp)
      \dot \vphi \di x \di t.
  \end{equation*}
  As in the proof of Lemma \ref{lem:strong_strain_rates_conv}, see \eqref{two ini0}, we obtain
  \begin{equation*}
    \int_0^{T} \int_\Omega\pl_F \cplpot(\nabla \NNN \py, \BBB \pt) : \nabla \dotay \vphi \di x \di t
    \to \intQ \pl_F \cplpot(\nabla y_\lp, \theta_\lp) : \nabla \dot y_\lp \vphi \di x \di t.
  \end{equation*}
  Note that by \ref{D_bounds},  \eqref{diss_rate}, \BBB and by $\rdrate \leq \drate$ we have
  \begin{equation*}
    \rdrate(\nabla \py, \nabla \dotay, \pt)
    \leq  2 \BBB \aC \big|
      (\nabla \dotay)^T \nabla \py
      + (\nabla \py)^T \nabla \dotay
    \big|^2.
  \end{equation*}
 \CCC By Lemma \ref{lem:strong_strain_rates_conv} and \eqref{uniform_def_conv2} \BBB $\CCC(\BBB(\nabla \dotay)^T \nabla \py + (\nabla \py)^T \nabla \dotay\CCC)_\tau\BBB$ converges strongly in $L^2(I; L^2(\Omega; \NNN \R^{d\times d} \BBB ))$.
  Consequently,  we get that $\big( \NNN \rdrate \BBB (\nabla \py, \nabla \dotay, \pt) \big)_\tau$ is equi-integrable. 
  Using the pointwise convergence of $(\nabla \py)_\tau$ and $(\pt)_\tau$ as well as the continuity of $\rdrate$,  we can also pass to the limit in the $\rdrate$-term \BBB by an application of Vitali's convergence theorem.  As we passed to the limit in each term, the proof is concluded. \BBB
\end{proof}

\section{Passage to the linearized system}\label{sec: linearization}
This section is devoted to the proofs of  Theorems \ref{thm:linearization_right_diag}--\ref{thm:linearization_left_bottom}. \BBB
In the following, let $(\lp_k)_k$ and $(\tau_k)_k$ be sequences with $\lp_k \to 0$ and either $\tau_k =\tau$ constant or $\tau_k \to 0$.
Suppose that initial data $(y_{0, \lp_k}, \theta_{0, \lp_k})$ as in \eqref{initial_cond} are given.
For brevity, we denote the corresponding time-discrete interpolations by $\nyk \defas \overline y_{\lp_k\CCC,\BBB \tau_k}$, $\pyk \defas \underline y_{\lp_k\CCC,\BBB \tau_k}$, and $\ayk \defas \hat y_{\lp_k\CCC,\BBB \tau_k}$, see \eqref{y_interpolations}.
A similar \CCC shorthand \BBB notation is also used for the interpolation of the temperatures as well as \NNN the \BBB internal energies.
Recall that the objects exist by Proposition \ref{cor:van_tau_reg}(i).
In a similar way, we denote the time-continuous solutions  obtained  \BBB in Proposition \ref{cor:van_tau_reg}(ii) by $(y_{\lp_k}, \theta_{\lp_k})$.
It will be useful to use a similar notation for the rescaled quantities: for time-discrete solutions we define
\begin{align*}
  \nuk &\defas \frac{\nyk - \id}{\lp_k}, &
  \puk &\defas \frac{\pyk - \id}{\lp_k}, &
  \auk &\defas \frac{\ayk - \id}{\lp_k}, &
  \nmuk &\defas \frac{\ntk}{\lp_k^\alpha}, &
  \pmuk &\defas \frac{\ptk}{\lp_k^\alpha},
\end{align*}
and for time-continuous solutions we let
\begin{equation*}
  u_{\lp_k} \defas \frac{y_{\lp_k} - \id}{\lp_k}, \qquad
  \mu_{\lp_k} \defas \frac{\theta_{\lp_k}}{\lp_k^\alpha}.
\end{equation*}
For any $v \in L^2(I; H^1(\Omega; \R^d))$ we denote the \emph{symmetrized gradient} by $e(v) \defas \frac{1}{2} (\nabla v + \nabla v^T)$.
Finally, all constants we encounter in this \NNN section \BBB are implicitly assumed to be independent of $k$.

We start with compactness results for the rescaled quantities which directly follow from the a priori estimates for the nonlinear system. \NNN Recall the definition of   $H^1_{\Gamma_D}$ in \eqref{eq: H1 strange}.  \BBB

\begin{lemma}[Compactness for the rescaled displacements]\label{lem:comp_uk}
  There exist $u, \, \tilde{u} \in \CCC H^1\BBB(I; H^1_{\Gamma_D}(\Omega; \R^d))$ with $u(0) = \tilde{u}(0)= u_0$ such that, up to possibly taking a subsequence, it holds that
  \begin{align}
    \auk &\to u \text{ in } L^\infty(I; L^2(\Omega; \R^d)), &
    \auk &\weakly u \text{ weakly in } H^1(I; H^1(\Omega; \R^d)),  \label{hoelder_conv_nuk_puk} \\
    u_{\lp_k} &\to \tilde{u} \text{ in } L^\infty(I; L^2(\Omega; \R^d)), &
    u_{\lp_k} &\weakly \tilde{u} \text{ weakly in } H^1(I; H^1(\Omega; \R^d)).  \label{hoelder_conv_nuk_puk2}
  \end{align}
  Moreover, if $\tau_k \to 0$, we also have
  \begin{equation}\label{hoelder_conv_nuk_puk3}
    \nuk, \, \puk \CCC \weakly \BBB u \NNN \text{ weakly in } H^1(I; H^1(\Omega; \R^d)). \BBB
  \end{equation}
\end{lemma}

Later, by uniqueness of the solution to the linear system, we will see that actually $u = \tilde u$.

\begin{proof}
  By the definition of $\nuk$ and \eqref{a_priori_Linfty_H1_lin} we derive for any $t \in I$ that $\norm{\nuk(t)}_{H^1(\Omega)} = \lp_k^{-1} \norm{\nyk - \id}_{H^1(\Omega)} \leq C$.
  For the other interpolations, we proceed in a similar fashion and get for all $t \in I$ that
  \begin{equation}\label{up and down}
    \norm{\auk(t)}_{H^1(\Omega)} \leq C.
  \end{equation}
  Moreover, using Poincaré's inequality, \eqref{a_priori_strain_rates_lin}, and the definition of $\auk$ we have that
  \begin{equation}\label{up and down2}
    \norm{\dotauk}_{L^2(I; H^1(\Omega))}
    \leq C \norm{\nabla \dotauk}_{L^2(I; L^2(\Omega))}
    = \frac{1}{\lp_k}  \Vert \nabla \dot{\hat{y}}_k \Vert_{L^2(I \times \Omega)} \leq C.
  \end{equation}
  Combining \eqref{up and down}--\eqref{up and down2} we discover that $(\auk)_k$ is bounded  in $L^\infty(I; H^1(\Omega;\R^d)) \cap H^1(I; H^1(\Omega;\R^d))$ and thus $(\auk)_k$ is compact in $C(I; L^2(\Omega;\MMM \R^d\BBB))$ by the Aubin-Lions' theorem.
  This \NNN together with Banach's selection principle  \BBB shows (\ref{hoelder_conv_nuk_puk}).
  Moreover, \eqref{hoelder_conv_nuk_puk3} follows from \eqref{up and down2} and the definition of the interpolations.
  Finally, due to (\ref{hoelder_conv_nuk_puk}) and the fact that $\hat{u}_k \in \CCC H^1\BBB(I; H^1_{\Gamma_D}(\Omega; \R^d))$ with $\hat{u}_k(0) = u_0$  (see \eqref{eq: randwerde} and \eqref{initial_cond}), \BBB it directly follows that $u \in \CCC H^1\BBB(I; H^1_{\Gamma_D}(\Omega; \R^d))$ with $u(0)= u_0$.

  We now show \eqref{hoelder_conv_nuk_puk2}.
  To this end, suppose that for each $k \in \N$ the solution $(y_{\lp_k}, \theta_{\lp_k})$ is obtained as the limit of time discrete solutions $(\hat{y}_{\lp_k\tau_l},\hat{\theta}_{\lp_k\tau_l})$ for a sequence $(\tau_l)_l$ converging to zero.
  Repeating \eqref{up and down}--\eqref{up and down2} the corresponding rescaled quantities satisfy $\norm{\hat{u}_{\lp_k\tau_l}}_{L^\infty(I; H^1(\Omega))} \leq C$ and $\norm{\dot{\hat{u}}_{\lp_k\tau_l}}_{L^2(I; H^1(\Omega))} \le C$ for a constant $C$ independent of $l$.
  Then, using \eqref{van_tau_y_conv} we get
  \begin{equation*}
    \norm{{u}_{\lp_k}}_{L^\infty(I;H^1(\Omega))} \leq C \quad \text{ and } \quad    \norm{\dot{u}_{\lp_k}}_{L^2(I; H^1(\Omega))} \le C.
  \end{equation*}
  Now, \eqref{hoelder_conv_nuk_puk2} and the other properties of $\tilde u$ again follow by the Aubin-Lions' theorem.
\end{proof}

\begin{lemma}[Compactness for the rescaled temperatures]\label{lem:comp_muk}
  There exist $\mu, \, \tilde{\mu} \in L^1(I; W^{1,1}(\Omega))$ with $\mu, \, \tilde{\mu} \ge 0$ such that, up to possibly taking a subsequence, for any $s \in [1,  \frac{d+2}{d})$ and $r \in [1, \frac{d+2}{d+1})$ it holds that
  \begin{align}
    \nmuk &\to \mu \text{ in } \CCC L^s(I \times \Omega)\BBB, &
    \nmuk &\weakly \mu \text{ weakly in } L^r(I; W^{1, r}(\Omega)), \label{nmuk_pmuk_Ls_conv} \\
    \mu_{\lp_k} &\to \tilde\mu \text{ in } \CCC L^s(I \times \Omega)\BBB, &
    \mu_{\lp_k} &\weakly \tilde \mu \text{ weakly in } L^r(I; W^{1, r}(\Omega)). \label{nmuk_pmuk_Ls_conv2}
  \end{align}
  Moreover, if $\tau_k \to 0$, we also have
  \begin{equation}\label{nmuk_pmuk_Ls_conv3}
    \pmuk, \, \amuk \to \mu \text{ in } \CCC L^s(I \times \Omega)\BBB.
  \end{equation}
\end{lemma}

 Later, by uniqueness of the solution to the linear system, we will see that actually $\mu = \tilde \mu$. \BBB

\begin{proof}
  Let $r$ and $s$ be as in the statement.
  The proof of \eqref{nmuk_pmuk_Ls_conv} relies on the a priori bounds on the internal energy in \NNN Theorem \BBB \ref{thm:further_apriori_temp_bounds}, i.e.,
  \begin{align}\label{nmuk_pmuk_Ls_conv-proof}
    \Vert \overline{\theta}_{k} \Vert_{L^{r}(I; W^{1,r}(\Omega))}
    + \Vert \overline{w}_{k} \Vert_{\CCC L^s(I \times \Omega)\BBB}
    + \Vert \overline{w}_{k} \Vert_{L^{r}(I; W^{1,r}(\Omega))}
    + \Vert \dot{\hat{w}}_{k} \Vert_{L^1(I; W^{1,\infty}(\Omega)^*)} \le C \lp^\alpha_k.
  \end{align}
  In fact, we can follow closely the lines of the proof of Lemma \ref{lem:pointwise_temp_conv}, see Remark \ref{rem: after temp}(ii).
  In particular, one first shows the convergence of the internal energies and then  by \eqref{nmuk_pmuk_Ls_conv3-toproof2NNN}  the convergence of the temperatures.
  Here, we also see that for $\tau_k \to 0$ property \eqref{nmuk_pmuk_Ls_conv3-toproof} implies \eqref{nmuk_pmuk_Ls_conv3}.

  To see \eqref{nmuk_pmuk_Ls_conv2}, we suppose that for each $k \in \N$ the solution $(y_{\lp_k}, \theta_{\lp_k})$ is obtained as the limit of time discrete solutions $(\hat{y}_{\lp_k\tau_l}, \hat{\theta}_{\lp_k\tau_l})$ for a sequence $(\tau_l)_l$ converging to zero.
  By the above reasoning we obtain \eqref{nmuk_pmuk_Ls_conv-proof} for  $\overline{\theta}_{\lp_k\tau_l}$ in place of $ \overline{\theta}_{k} $   and $\overline{w}_{\lp_k\tau_l} \defas  \inten(\nabla \overline{y}_{\lp_k\tau_l}, \overline{\theta}_{\lp_k\tau_l})$ in place of $\overline{w}_{k}$. \BBB
  Then by \eqref{pointwise_temp_conv} and the lower semicontinuity of the norms we get
  \begin{equation*}
    \Vert \theta_{\lp_k} \Vert_{L^{r}(I; W^{1,r}(\Omega))}
    + \Vert w_{\lp_k} \Vert_{\CCC L^s(I \times \Omega) \BBB}
    + \Vert w_{\lp_k} \Vert_{L^{r}(I; W^{1,r}(\Omega))} \le C \lp_k^\alpha,
  \end{equation*}
  where $w_{\lp_k} \defas \inten(\nabla y_{\lp_k}, \theta_{\lp_k})$.
  It now suffices to check that also
  \begin{align}\label{nmuk_pmuk_Ls_conv-proof2}
    \Vert \dot{w}_{\lp_k} \Vert_{L^1(I; W^{1,\infty}(\Omega)^*)} \le C \lp_k^\alpha
  \end{align}
  holds as then the statement follows by repeating the proof of Lemma~\ref{lem:pointwise_temp_conv}, see again Remark \ref{rem: after temp}(ii).
  To derive \eqref{nmuk_pmuk_Ls_conv-proof2}, we use \eqref{weak_limit_heat_equation}  (for $\rdrate$ \CCC in place of $\drate$\BBB) \BBB to get that $\dot{w}_{\lp_k}$ coincides in the distributional sense with \NNN $\sigma$ where for each $t \in I$ and each  $\varphi \in C_c^\infty(\Omega)$  we set \BBB
  \begin{align*}
    \langle \sigma(t), \varphi \rangle
    &\defas \kappa \int_{ \Gamma} \big(
         \lp_k^\alpha \BBB \bt - \theta_{\lp_k}
      \big) \vphi \di \haus^{d-1}(x) \\
    &\phantom{\defas} \   - \int_\Omega \Big(
      \hcm(\nabla y_{\lp_k}, \theta_{\lp_k}) \nabla \theta_{\lp_k}
        \cdot \nabla \vphi
      - \big(
           \rdrate \BBB (\nabla y_{\lp_k}, \nabla \dot y_{\lp_k}, \theta_{\lp_k})
          + \pl_F \cplpot(\nabla y_{\lp_k}, \theta_{\lp_k})
            : \nabla \dot y_{\lp_k}
        \big) \vphi
      \Big) \di x
  \end{align*}
  \NNN where all functions on the right-hand side are evaluated at $t \in I$. \BBB   By passing to the limit $\tau \to 0$ in \eqref{a_priori_Linfty_W2p_lin}--\eqref{a_priori_strain_rates_lin} and \eqref{temp_inten_Lq_bound}--\eqref{nablatemp_nablainten_Lr_bound} we obtain the a priori bounds \NNN  $\Vert y_{\lp_k} - \id \Vert_{L^\infty(I; W^{1,\infty}(\Omega))} \le C \lp_k^{2/p}$, \BBB $\Vert y_{\lp_k} - \id \Vert_{H^1(I; H^1(\Omega))} \le C \lp_k$\CCC, \BBB and $\Vert \theta_{\lp_k} \Vert_{L^1(I; W^{1,1}(\Omega))} \le C  \lp_k^\alpha \BBB$.
  This along with   \ref{D_bounds}, \NNN $\CCC\rdrate \BBB \le \drate$,  \BBB \eqref{bound_hcm}, \eqref{C_locally_lipschitz}, \NNN $\alpha \le 2$, \BBB and the trace estimate shows that $t \mapsto \Vert \sigma(t) \Vert_{W^{1,\infty}(\Omega)^*}$ lies in $L^1(I)$ with $\Vert \sigma \Vert_{L^1(I; W^{1,\infty}(\Omega)^*)} \le C \lp_k^\alpha$.
  This concludes the proof of \eqref{nmuk_pmuk_Ls_conv-proof2}.
\end{proof}

We now proceed with the proofs of Theorems \ref{thm:linearization_right_diag} and \ref{thm:linearization_left_bottom} which we split into two subsections.

\subsection{Proof of Theorem \ref{thm:linearization_right_diag}}\label{sec: 5.1}

We will only prove Theorem \ref{thm:linearization_right_diag}(iii) as item (ii) of the statement can be obtained along similar lines by performing the linearization directly in the weak formulation \eqref{weak_limit_mechanical_equation}--\eqref{weak_limit_heat_equation} in place of the Euler-Lagrange equations \eqref{mechanical_step_single} and \eqref{el_thermal_step}.
Note that the proof of Theorem~\ref{thm:linearization_right_diag}(iii) will also imply the existence statement in Theorem~\ref{thm:linearization_right_diag}(i).
In this subsection, we also address the uniqueness of the solutions to the linearized system.

\begin{proposition}[Linearization of the mechanical equation]\label{prop:linearization_mech}
  Let $u$ and $\mu$ be given as in Lemmas~\ref{lem:comp_uk}--\ref{lem:comp_muk}.
  Then, for any $z \in C^\infty(I \times \overline{\Omega}; \R^d)$ with $z = 0$ on $I \times \Gamma_D$ we have that \eqref{linear_evol_mech} holds.
\end{proposition}

\begin{proof}
  Let $z$ be as in the statement.   As $z \in \mathcal{Y}_0$, we can  multiply  \eqref{mechanical_step_single} with $\tau_k/\lp_k$  and sum over  all steps  $1, \ldots,  T/\tau$ to get \BBB
  \begin{equation}\label{discri-mech}
    \frac{1}{\lp_k} \int_0^{T} \int_\Omega \Big(
        \pl_F \felpot(\nabla \nyk, \ptk)
        + \pl_{\dot F} \disspot(\nabla \pyk, \nabla \dotayk, \ptk)
      \Big) : \nabla z
      + \pl_G \hypot(\nabla^2 \nyk) \cdddot \nabla^2 z \di x \di t
    = \int_0^{T} \langle  \overline\ell_{\tau_k}(t), \BBB z \rangle \di t,
  \end{equation}
     where $\overline\ell_{\tau_k}(t) \defas   \ell^{(l)}_{\tau_k} $ for $t \in ((l-1)\tau, l\tau]$ and $l \in \setof{1, \ldots,  T/\tau}$.  \BBB   Our goal now is to show that \eqref{linear_evol_mech} arises as the limit of the above equation as $k \to \infty$.
  First, recalling \eqref{forces_mech_step} we can easily check that  
  \begin{equation}\label{linmech1}
    \int_0^{T} \langle \overline\ell_{\tau_k}(t), z(t) \rangle \di t
      \to \intQ f \cdot z + \intSN g \cdot z \di x \di t
  \end{equation}
  as $k \to \infty$.
  By \ref{H_bounds} for $\pl_G \hypot$, \eqref{a_priori_Linfty_W2p_lin}, and H\"older's inequality with powers $\frac{p}{p-1}$ and $p$ we derive that
  \begin{align}\label{linmech2}
    \frac{1}{\lp_k} \Bigg|
      \int_0^{T} \int_\Omega \pl_G \hypot(\nabla^2 \nyk) \cdddot \nabla^2 z \di x \di t
    \Bigg|
    &\leq \frac{C_0}{\lp_k} \intQ \abs{\nabla^2 \nyk}^{p-1}
      \abs{\nabla^2 z} \di x \di t \\
    &\leq  \frac{C_0}{\lp_k} \int_{0}^T \Vert \nabla^2 \nyk \Vert^{p-1}_{L^p(\Omega)}
      \norm{\nabla^2 z}_{L^p(\Omega)}\di t  \BBB
      \leq C \lp_k^{\frac{2(p-1)}{p} - 1} = C \lp_k^{1-\frac{2}{p}} \to 0, \notag
  \end{align}
  as $p > d \ge 2$.   We now address the coupling term.
  In view of $\partial_F \cplpot(\Id, 0) = 0$,  \eqref{a_priori_Linfty_W2p_lin}, \BBB and \eqref{C_locally_lipschitz}, a Taylor expansion implies
  \begin{equation}\label{all not remove0}
  \Big|
    \pl_F \cplpot(\nabla \nyk, \ptk)
    -\Big(
      \partial_F^2 \cplpot(\Id, 0) \lp_k \nabla \nuk
      + \partial_{F\theta} \cplpot(\Id, 0)  (\lp_k^\alpha \pmuk \wedge 1 )
 \BBB    \Big)
  \Big|
  \le C |\lp_k \nabla \nuk|^2 + C \big( |\lp_k^{\alpha}  \pmuk|^2 \BBB \wedge 1 \big)
  \end{equation}
  pointwise a.e.~in $I \times \Omega$.
  Thus, by \eqref{hoelder_conv_nuk_puk} and \NNN \eqref{nmuk_pmuk_Ls_conv3}, \BBB along with $t^2 \wedge 1 \le t^s$ for $t \ge 0$ for some fixed $s \in (1,  \frac{d+2}{d})$ \CCC it follows that \BBB
  \begin{align*}
    &\lim_{k \to \infty} \frac{1}{\lp_k} \intQ
      \pl_F \cplpot(\nabla \nyk, \ptk) : \nabla z \di x \di t \\
    &\quad = \lim_{k \to \infty} \intQ \big(
      \partial_F^2 \cplpot(\Id, 0) \nabla \nuk
      +   \lp_k^{  -1}  \BBB \partial_{F\theta} \cplpot(\Id, 0)  ( \lp_k^{\alpha}\pmuk\wedge 1) \BBB
    \big) : \nabla z \di x \di t.
  \end{align*}
  Recalling \CCC $\partial_F^2 \cplpot(\Id, 0) = 0$  (cf.~\ref{C_zero_temperature}) and the definition of  $\mathbb{B}^{(\alpha)}$ in \eqref{alpha_dep} we find \BBB
  \begin{equation}\label{all not remove}
    \lim_{k\to \infty} \frac{1}{\lp_k} \intQ
      \pl_F W^{\rm cpl}(\nabla \nyk, \ptk) : \nabla z \di x \di t 
      = \intQ \mathbb{B}^{(\alpha)} \mu : \nabla z  \di x \di t.
  \end{equation}
  By a Taylor expansion,   \eqref{a_priori_Linfty_W2p_lin}, \BBB and the fact that $\elpot$ is $C^3$ we have
  \begin{equation*}
    \Big|
      \lp_k^{-1} \pl_F \elpot(\nabla \nyk)
      - \pl^2_F \elpot(\Id) \nabla \nuk
    \Big|
    \leq \frac{C}{\lp_k} \abs{\nabla \nyk - \Id}^2.
  \end{equation*}
  Integrating the above inequality over $I \times \Omega$ and using \eqref{a_priori_Linfty_H1_lin} we get
  \begin{align}   \label{linmech3}
    \Bigg|
      \int_0^{T} \int_\Omega \Big(
        \lp_k^{-1} \pl_F \elpot(\nabla \nyk)
        - \partial^2_F \elpot(\Id) \nabla \nuk
      \Big) : \nabla z \di x \di t
    \Bigg|
    \leq C \lp_k^{-1} \Vert \nabla \nyk - \Id \Vert_{L^2(I \times \Omega)}^2
    \leq C T \lp_k \to 0.
  \end{align}
  \CCC By \BBB \eqref{chain_rule_Fderiv}
  \begin{align}\label{eq: auch noch}
    \pl_{\dot F} \disspot(\nabla \pyk, \nabla \dotayk, \ptk) : \nabla z
    = 2 \nabla \pyk (D(C_k, \ptk) \lp_k \dot C_k) : \nabla z
    = \lp_k \dot C_k : D(C_k, \ptk)
      (\nabla z^T \nabla \pyk + \nabla \pyk^T \nabla z),
  \end{align}
  where
  \begin{align}
    C_k &\defas \nabla \pyk^T \nabla \pyk, &
    \dot C_k &\defas \nabla \dotauk^T \nabla \pyk
      + \nabla \pyk^T \nabla \dotauk. \label{def_Ck_dotCk}
  \end{align}
  Note that the second identity is obtained by an elementary computation using the symmetries of $D$ stated in \ref{D_quadratic}.
  By   \eqref{a_priori_Linfty_W2p_lin}  \BBB and \eqref{hoelder_conv_nuk_puk} we \CCC then \BBB see that
  \begin{align}\label{ffflater}
  \dot C_k \weakly 2 e(\dot u) \quad \text{ weakly in } L^2(I\times \Omega;\R^{d \times d}_{\rm \NNN sym}).
  \end{align}
  Using \ref{D_bounds} we also \CCC have \BBB that
  \begin{equation*}
    \abs{D(C_k, \ptk) (\nabla z^T \nabla \pyk + \nabla \pyk^T \nabla z)}
    \leq 2 \aC \norm{\nabla z}_{L^\infty(\Omega)} \norm{\nabla \pyk}_{L^\infty(\Omega)}.
  \end{equation*}
  Up to taking a subsequence (not relabeled), we can suppose that $\nabla \pyk \to \Id$ and $\ptk \to 0$ a.e.~in $I \times \Omega$.
  Thus, Dominated Convergence implies
  \begin{equation*}
    D(C_k, \ptk) (\nabla z^T \nabla \pyk
      + \nabla \pyk^T \nabla z)
    \to D(\Id, 0) (\nabla z + \nabla z^T)
    = 2 D(\Id, 0) \nabla z
  \end{equation*}
  strongly in $L^2(I \times \Omega;\R^{d \times d})$.
  This along with  \NNN \eqref{eq: auch noch} and \BBB  \eqref{ffflater} \BBB leads to
  \begin{align}   \label{linmech4}
    \lp_k^{-1} \int_0^{T} \int_\Omega \pl_{\dot F} \disspot(\nabla \pyk, \nabla \dotayk, \ptk)
      : \nabla z \di x \di t
    \to \intQ 4 D(\Id, 0) e(\dot u) : \nabla z \di x \di t.
  \end{align}
  Recalling the definition of $\CD$ and $\C_W$ in \eqref{def_WD_tensors}, as well as \NNN collecting \eqref{linmech1}, \eqref{linmech2}, \eqref{all not remove}, \eqref{linmech3}, and \eqref{linmech4} \BBB we conclude the proof.
\end{proof}

Similarly as in Section \ref{sec:tau_to_zero_delta_fixed}, for the limit passage in the heat-transfer equation,  we will need  the strong convergence of the strain rates $(\nabla \dot{\hat{u}}_k)_k$ in $L^2(I; L^2(\Omega;\R^{d \times d}))$ since the dissipation rate is quadratic in $\nabla \dot{\hat{u}}_k$.
We now improve the compactness in Lemma \ref{lem:comp_uk} as follows. \NNN At this state, we need the additional assumption \ref{H_bounds2} which \BBB combined with the bound on $\pl_G H(G)$ from \ref{H_bounds} leads to  
\begin{equation}\label{H_upper_bound_spec}
  \abs{H(G)} \leq  \aC  \abs{G}^p \quad \text{ for all } G \in \R^{d \times d \times d}.
\end{equation}

\begin{lemma}[Strong convergence of the rescaled strains and strain rates]\label{lemma: strong ratistrain}
  With $u$ as in Lemma \ref{lem:comp_uk}, up to possibly taking a subsequence, we have
  \begin{align}\label{strong_rescaled_strain_rate_comp}
    \hat{u}_k(t) &\to u(t) \text{ strongly in } H^1(\Omega; \R^d)
      \text{ for all } t \in I, &
    \nabla \dotauk &\to \nabla \dot u \text{ strongly in } L^2(I; L^2(\Omega; \NNN \R^{d \times d \BBB} )).
  \end{align}
  The first convergence also holds with $\overline{u}_k$ or $\underline{u}_k$ in place of $\hat{u}_k$. \BBB
  \end{lemma}

\begin{proof}
  \textit{Step 1 (Lower bounds for elastic energy and dissipation):}
  Suppose we have already selected a subsequence so that the convergences of Lemma \ref{lem:comp_uk} as well as Lemma \ref{lem:comp_muk} hold true.
 Recall the definition of $\mathcal{M}_{\lp_k}$ before \eqref{eq: rescaled toten}. \BBB   For convenience, for any $v \in H^1(\Omega;\R^d)$, we define
  \begin{equation*}
    \mechenl(v) \defas \frac{1}{2} \int_\Omega \CW  e(v) \BBB : e(v) \di x,
  \end{equation*}
  where $\CW = \partial^2_F \elpot(\Id)$ is as in \eqref{def_WD_tensors}.
  Let us fix an arbitrary $t \in I$.
  By the non-negativity of $\hypot$, a Taylor expansion, and   \eqref{a_priori_Linfty_W2p_lin}  \BBB we derive that
  \begin{align}\label{taylor_mechen_lpk}
    \mechen_{\lp_k}(\nyk(t))
    &\geq \lp_k^{-2} \int_\Omega \elpot(\nabla \nyk(t)) \di x \notag \\
    &\geq \frac{1}{2} \int_\Omega \partial^2_F \elpot(\Id)\nabla \nuk(t)
      : \nabla \nuk(t)
      - C \int_\Omega \abs{\nyk(t) - \Id} \abs{\nabla \nuk(t)}^2 \di x \notag \\
    &\geq \frac{1}{2} \int_\Omega \partial^2_F \elpot(\Id) \nabla \nuk(t)
      : \nabla \nuk(t)
      - C \lp_k^{2/p} \int_\Omega \abs{\nabla \nuk(t)}^2 \di x.
  \end{align}
  Consequently, using \eqref{hoelder_conv_nuk_puk3} it follows that
  \begin{equation}\label{liminf_rescaled_en}
    I_1 \defas \liminf_{k \to \infty} \mechen_{\lp_k}(\nyk(t))
    \ge \liminf_{k \to \infty} \mechenl(\nuk(t))
    \geq \mechenl(u(t)).
  \end{equation}
  Let $C_k$ and $\dot C_k$ be as in \eqref{def_Ck_dotCk}.
  In \eqref{ffflater} we have  seen that $\dot C_k \weakly 2 e(\dot u)$ weakly in $L^2(I\times \Omega;\R^{d \times d})$.
  This along with the definition in \eqref{diss_rate}, $\CD = 4 D(\Id, 0)$, the pointwise convergences of $(\nabla \pyk)_k$ and $(\ptk)_k$\CCC, and standard lower semicontinuity arguments (see also \cite[Theorem 7.5]{FonsecaLeoni07Modern}) show\BBB
  \begin{align}
    I_2 \defas \liminf_{k \to \infty} \lp_k^{-2}
      \int_{0}^t \int_\Omega \drate(\nabla \pyk, \nabla \dotayk, \ptk) \di x \di s
    &= \liminf_{k \to \infty} \int_0^t \int_\Omega D(C_k, \ptk) \dot C_k
      : \dot C_k \di x \di s \nonumber \\
    &\geq \int_0^t \int_\Omega \CD e(\dot u) : e(\dot u) \di x \di s. \label{liminf_rescaled_dissrate}
  \end{align}
  \textit{Step 2 (Convergence of elastic energies and dissipations):}
  Our next goal is to show the reverse inequalities for the $\limsup$.
  To this end, we draw ideas from the proof of Lemma \ref{lem:strong_strain_rates_conv} and compare an energy balance on the nonlinear time-discrete level with a time-continuous energy balance in the linearized setting.
  First, recall from \eqref{disc_energy_balance} that for $K \in \N$ with $K  \tau_k \BBB \in \NNN [t, t + \tau_k)\BBB$   it holds that
  \begin{align}\label{disc_energy_balance-new}
    &\mechen_{\lp_k}(\nyk(K\tau_k))
      + \lp_k^{-2} \int_{0}^{K\tau_k} \int_\Omega
        \drate(\nabla \pyk, \nabla \dotayk, \ptk) \di x \di s
      - \tau_k \Lambda \int_{0}^{K\tau_k} \int_\Omega
        \abs{\nabla \dotauk}^2 \di x \di s  \notag \\
    &\quad \leq \mechen_{\lp_k}( y_{0,\lp_k} \BBB)
      + \frac{1}{\lp_k} \int_0^{K\tau_k}
        \langle \ell(s), \dot{\hat{y}}_k(s) \rangle \di s
      - \int_{0}^{K\tau_k} \int_\Omega
        \lp_k^{-1} \pl_F \cplpot(\nabla \nyk, \ptk)
          : \nabla \dotauk \di x \di s,
  \end{align}
  where $\Lambda > 0$ does not depend on $k$.
  Here, we also used \eqref{diss_rate} to replace $R$ by $\xi$.
  Now, in a similar fashion, testing \eqref{linear_evol_mech} \CCC with \BBB $z = \dot u$ we see that
  \begin{equation}\label{lim_energy_balance-new}
    \mechenl(u(t)) - \mechenl(u_0)
    + \int_0^t \int_\Omega \big( \CD e(\dot u) : e(\dot u) +  \mu \mathbb{B}^{(\alpha)}  : \nabla \dot u  \big) \BBB \di x \di s 
    = \int_0^t \langle \ell(s), \dot{u}(s) \rangle \di s.
  \end{equation}
  We now address the convergence of the various terms.
  First of all, by \eqref{hoelder_conv_nuk_puk} we clearly have
  \begin{equation}\label{lindiss3}
    \frac{1}{\lp_k} \int_0^{K \tau_k}
      \langle \ell(s), \dot{\hat{y}}_k(s) \rangle \di s
    = \int_0^{K \tau_k} \langle \ell(s), \dot{\hat{u}}_k(s) \rangle \di s
    \to \int_0^t \langle \ell(s), \dot{u}(s) \rangle \di s.
  \end{equation}
  For $\alpha = 1$, by arguing similarly as in \eqref{all not remove0}--\eqref{all not remove},  and using \eqref{a_priori_Linfty_W2p_lin} as well as  $\underline{\mu}_k  \to \mu $ strongly in $L^2(I \times \Omega)$ by Remark~\ref{rem: after temp}(iii) we find
 \begin{equation}\label{lindiss2XXX}
\NNN  I_3 \coloneqq  \lim_{k\to\infty} \BBB   \frac{1}{\lp_k}
      \int_{0}^{K \tau_k} \int_\Omega
        \pl_F \cplpot(\nabla \nyk, \ptk) : \nabla \dot{\hat{u}}_k \di x \di s
         \NNN = \BBB  \int_0^t \int_\Omega  \mathbb{B}^{(\alpha)} \mu : \nabla \dot u  \di x \di t,
  \end{equation} 
  where we also used the definition of $\mathbb{B}^{(\alpha)}$ in \eqref{alpha_dep}. For $\alpha \in (1,2]$,  \eqref{lindiss2XXX} also holds (with $\mathbb{B}^{(\alpha)} = 0$), since by Remark \ref{rem: next a priori} we find that $ \underline{\mu}_k$ is bounded in $L^q(I;L^q(\Omega))$ for some \CCC$q \in( 2/\alpha, 2]$\BBB, and therefore using $t \wedge 1 \le t^{q/2} $ for $t \ge 0$ and Young's inequality with constant $\lp_k^{\alpha q/2}$ we get   
      \begin{align}\label{eq: do it small!}
 \lp_k^{  -1}\intQ    \big| \CCC(\BBB\lp_k^{\alpha}\pmuk\CCC)\BBB\wedge 1 \big| | \nabla \dot{\hat{u}}_k| \di x \di t  \le  \lp_k^{  -1} \Big( \lp_k^{-\alpha q/2} \Vert  \lp_k^{\alpha}\pmuk \Vert_{L^q(I\times \Omega)}^q   +     \lp_k^{\alpha q/2}\Vert \nabla \dot{\hat{u}}_k \Vert_{L^2(I\times \Omega)}^2 \Big) \to 0.
  \end{align}
\BBB    Eventually, we get
  \begin{equation}\label{lindiss1}
    \lim_{k \to \infty} \mechen_{\lp_k}(y_{0, \lp_k})
    = \lim_{k \to \infty} \mechen_{\lp_k}(\id + \lp_k u_0) = \mechenl(u_0).
  \end{equation}
  In fact, for the convergence of the elastic energy we repeat the Taylor expansion in \eqref{taylor_mechen_lpk} (with equality), and for the \NNN second-gradient \BBB term we get by \eqref{H_upper_bound_spec}, $u_0 \in W^{2, p}(\Omega; \R^d)$, and $p > 2$ that
  \begin{equation*}
    \lp_k^{-2} \Big|
      \int_\Omega \hypot(\lp_k \nabla^2 u_0) \di x
    \Big|
    \leq C \lp_k^{p-2} \int_\Omega \abs{\nabla^2 u_0}^p \di x
    \leq C \lp_k^{p-2} \to 0.
  \end{equation*}
  Combining \eqref{disc_energy_balance-new}--\eqref{lim_energy_balance-new}, $K\tau_k \ge t$, the convergences \eqref{liminf_rescaled_en}, \eqref{liminf_rescaled_dissrate}, \eqref{lindiss3}, \NNN  \eqref{lindiss2XXX}, \BBB and \eqref{lindiss1}, as well as using that $\tau_k \int_0^{K \tau_k} \int_\Omega \abs{\nabla \dotauk}^2 \di x \di s \to 0$ as $\tau_k \to 0$ we get 
  \begin{equation*}
  \begin{aligned}
    \mechenl(u(t))
      + \int_0^t \int_\Omega &  \big( \CD e(\dot u) : e(\dot u)  +  \mu \mathbb{B}^{(\alpha)}  : \nabla \dot u  \big) \BBB \di x \di s
    = \mechenl(u_0)
      + \int_0^t \langle \ell(s), \dot{u}(s) \rangle \di s \\
    &\ge I_1 + I_2 \NNN + I_3 \BBB
    \ge \mechenl(u(t))
       \NNN  + \int_0^t \int_\Omega  \big( \CD e(\dot u) : e(\dot u)  +  \mu \mathbb{B}^{(\alpha)}  : \nabla \dot u  \big)  \di x \di s. \BBB
  \end{aligned}
  \end{equation*}
  Thus, all inequalities in \eqref{liminf_rescaled_en} and \eqref{liminf_rescaled_dissrate} are equalities.
  In particular, we derive
  \begin{align}
    \lim_{k \to \infty} \frac{1}{2} \int_\Omega
      \CW e(\nuk(t)) : e(\nuk(t)) \di x
    &= \frac{1}{2} \int_\Omega \CW e(u(t)) : e(u(t)) \di x, \label{rescaled_diss_conv1} \\
    \lim_{k \to \infty} \NNN \frac{1}{\eps_k^2} \BBB \int_{0}^t \int_\Omega
      \drate(\nabla \pyk, \nabla \dotayk, \ptk) \di x \di s
    &= \int_0^t \int_\Omega 4 D(\Id,0) e(\dot u) : e(\dot u) \di x \di s, \label{rescaled_diss_conv2}
  \end{align}
  where we also used the definition of $\CD$ in \eqref{def_WD_tensors}. \\
  \textit{Step 3 (Strong convergence):}
  Strong convergence for $\overline{u}_k$ in $H^1(\Omega;\R^d)$, i.e.,  the first part of \eqref{strong_rescaled_strain_rate_comp}, \BBB follows directly from \eqref{rescaled_diss_conv1}, Korn's and Poincar\'e's inequality, and the fact that $\CW$ is positive definite on $\R^{d \times d}_{\rm sym}$.
In the same way we obtain convergence of  $\underline{u}_k$ by employing $\pyk(t)$ in place of $\nyk(t)$ in \eqref{taylor_mechen_lpk}.
  Hence, the statement  also holds for  $\hat{u}_k$. \BBB

  \NNN For the second part of \BBB \eqref{strong_rescaled_strain_rate_comp}, we will first show strong convergence of $(\dot C_k)_k$ defined in \eqref{def_Ck_dotCk}:
  by \ref{D_bounds} we estimate
  \begin{align*}
    &\ac \intQ \abs{\dot C_k - 2 e(\dot u)}^2 \di x \di t
    \leq \intQ D(C_k, \ptk) (\dot C_k - 2e(\dot u))
      : (\dot C_k - 2 e(\dot u)) \di x \di t \\
    &\quad= \NNN \eps_k^{-2} \BBB \intQ \drate(\nabla \pyk, \nabla \dotayk, \ptk) \di x \di t
      - 2 \intQ 2 D(C_k, \ptk) e(\dot u) : \dot C_k \di x \di t \\
    &\phantom{\quad=}\quad + \intQ 4D(C_k, \ptk) e(\dot u) : e(\dot u) \di x \di t.
  \end{align*}
  By \eqref{rescaled_diss_conv2} for $t = T$, the pointwise convergence of $(\nabla \pyk)_k$ and $(\ptk)_k$ to $\Id$ and $0$, respectively (see \eqref{hoelder_conv_nuk_puk}--\eqref{hoelder_conv_nuk_puk3} and \eqref{nmuk_pmuk_Ls_conv}), and the already shown weak convergence of $\dot C_k$ towards $2 e(\dot u)$  (cf.~\eqref{ffflater}) \BBB we see that the above derived upper bound converges to $0$ as $k \to \infty$.
  Then, the desired strong convergence of $(\nabla \dotauk)_k$ is derived as follows: by using Poincaré's and Korn's inequality,  \eqref{hoelder_conv_nuk_puk}, and  \eqref{a_priori_Linfty_W2p_lin} \BBB we get
  \begin{align*}
    &\intQ \abs{\nabla \dotauk - \nabla \dot u}^2 \di x \di t
    \leq C \intQ \abs{\sym(\nabla \dotauk - \nabla \dot u)}^2 \di x \di t \\
    &\quad\leq C \intQ \abs{\dot C_k - 2 e(\dot u)}^2 \di x \di t
      + C \intQ \abs{\nabla \pyk - \Id}^2 \abs{\nabla \dotauk}^2 \di x \di t \\
    &\quad\leq C \intQ \abs{\dot C_k - 2 e(\dot u)}^2 \di x \di t
      + C \lp_k^{4/p} \intQ \abs{\nabla \dotauk}^2 \di x \di t \to 0.
  \end{align*}
  This concludes the proof.
\end{proof}

\begin{proposition}[Linearization of the heat-transfer equation]\label{prop:linearization_therm}
  Let $u$ be as in Lemma \ref{lem:comp_uk} and $\mu$ as in Lemma \ref{lem:comp_muk}.
  Then, for any $\vphi \in C^\infty(I \times \overline \Omega)$ with $\vphi(T) = 0$ we have that \eqref{linear_evol_temp} holds.
\end{proposition}

\begin{proof}
  Similarly to the proof of Proposition \ref{thm:convergence_heat_vanishing_tau},  see \eqref{ea: just copy}, \BBB we can show that
  \begin{align}\label{discr-heat}
    &\int_0^T \int_\Omega
    \hcm(\nabla \pyk, \ptk) \nabla \nmuk \cdot \nabla \vphi \di x \di t
    - \int_0^{T} \int_\Omega   \lp_k^{-\alpha} \BBB \awk \dot \vphi \di x \di t
    + \kappa \int_0^{T} \int_{ \Gamma} \nmuk \vphi \di \haus^{d-1} \di t \notag \\
    &\quad- \int_0^{\NNN T} \int_\Omega \Big(
      \lp_k^{-\alpha} \rdrate(\nabla \pyk, \nabla \dotayk, \ptk)
      + \lp_k^{1-\alpha} \pl_F \cplpot(\nabla \NNN \pyk,  \ptk \BBB ) : \nabla \dotauk
    \Big) \vphi \di x \di t \notag \\
    &= \kappa \int_0^{T} \int_{ \Gamma}
     \NNN \overline \theta_{\flat, \tau}  \BBB \vphi \di \haus^{d-1} \di t
      + \lp_k^{-\alpha} \int_\Omega \inten(\nabla  y_{0,\lp} \BBB ,  \theta_{0,\lp} \BBB ) \vphi(0) \di x,
  \end{align}
  where $\awk \defas \hat w_{\lp_k\CCC,\BBB \tau_k}$, $\nabla  y_{0,\lp} \BBB  = \Id + \lp_k \nabla u_0$, and $ \theta_{0,\lp} \BBB  = \lp_k^\alpha \mu_0$.
  Note that in contrast to  \eqref{ea: just copy}, \BBB we rescaled both sides with $\lp_k^{-\alpha}$.
  We will now pass to the limit in each integral above as $k \to \infty$.
  Recall that $c_V(F, \theta) \defas -\theta \pl_\theta^2 \cplpot(F, \theta)$ for any $F \in \NNN GL^+(d) \BBB$ and $\theta \geq 0$.
  % By our assumption  \ref{C_heatcap_cont}, \BBB $c_V$ is a continuous map on $\NNN GL^+(d) \BBB \times \R_+$.
  % Consequently, by the definition of $\bar c_V$ in \eqref{lin_heat_cond_cap}, \eqref{sec_deriv}, \NNN $W^{\rm in}(F,0) = 0$ for all $F \in GL^+(d)$, \BBB and a change of variables we get \NNN by the Fundamental Theorem of Calculus that \BBB
  % \begin{align*}
  %   \Big|
  %     \lp_k^{-\alpha} \int_\Omega \inten(\nabla  y_{0,\lp} \BBB ,  \theta_{0,\lp} \BBB ) \vphi(0) \di x
  %     - \hspace{-0.1cm} \int_\Omega \bar c_V \mu_0 \varphi(0) \di x
  %   \Big|
  %   &\leq \int_\Omega \Big|
  %     \lp_k^{-\alpha} \int_0^{ \theta_{0,\lp} \BBB } c_V(\nabla  y_{0,\lp} \BBB , s) \di s
  %     - \bar c_V \NNN \mu_0 \BBB
  %   \Big| \abs{\vphi(0)} \di x \\
  %   &\leq \int_\Omega \int_0^{\mu_0} \abs{c_V(\nabla  y_{0,\lp} \BBB ,  \lp_k^{\alpha} \BBB s)
  %     - c_V(\Id, 0)} \di s \, \abs{\vphi(0)} \di x.
  % \end{align*}
  % For brevity let us define \error{Streamline the argument with Remark \ref{rem: after temp}(ii).}
  % \begin{equation*}
  %   f_k(x) \defas \int_0^{\mu_0(x)} \abs{c_V(\nabla  y_{0,\lp} \BBB , \lp_k^\alpha s) - c_V(\Id, 0)} \di s.
  % \end{equation*}
  % By \ref{C_bounds} we see that $f_k \leq 2 \aC \mu_0 \in L^1(\Omega)$.
  % By \ref{C_bounds} and the continuity of $c_V$ we also discover that $\lim_{k \to \infty} f_k(x) = 0$ for a.e.~$x \in \Omega$.
  % Consequently, by Dominated Convergence it follows that
  \CCC Using \ref{C_heatcap_cont} and Dominated Convergence we can show in a similar fashion as in Remark \ref{rem: after temp} that \BBB
  \begin{equation*}
    \lp_k^{-\alpha} \int_\Omega \inten(\nabla  y_{0,\lp} \BBB ,  \theta_{0,\lp} \BBB ) \vphi(0) \di x
    \to \int_\Omega \bar c_V \mu_0 \vphi(0) \di x.
  \end{equation*}
By Lemma \ref{lem:bound_hcm} we have that $\abs{\hcm(\nabla \pyk, \ptk)}$  is uniformly bounded. \BBB
  Consequently, from the pointwise convergence of $\nabla \pyk$ and $\ptk$ to $\Id$ and $0$, respectively, see \eqref{hoelder_conv_nuk_puk}--\eqref{hoelder_conv_nuk_puk3} and \eqref{nmuk_pmuk_Ls_conv}, we derive that
  \begin{align*}
    &\int_{0}^{T} \hspace{-0.1cm}
      \int_\Omega \hcm(\nabla \pyk, \ptk) \nabla \nmuk \cdot \nabla \vphi \di x \di t
    + \kappa\int_0^{T} \hspace{-0.1cm} \int_{\NNN \Gamma} \nmuk \vphi \di \haus^{d-1} \di t  \to \int_0^T \hspace{-0.1cm} \int_\Omega \NNN \hc_0 \BBB \nabla \mu \cdot \nabla \vphi \di x \di t
      + \kappa \NNN \int_\Gamma  \BBB \mu \vphi \di \haus^{d-1} \di t,
  \end{align*}
  where \NNN $\hc_0$  \BBB is defined in \eqref{lin_heat_cond_cap}.
By a change of variables and Dominated Convergence we find
  \begin{align}\label{eq: cof}
    \lp_k^{-\alpha} \inten(\nabla \nyk, \ntk)
    = \int_0^{\nmuk} c_V(\nabla \nyk, \lp_k^{\alpha} s) \di s
    = \int_0^{\mu} c_V(\nabla \nyk, \lp_k^{\alpha} s) \di s
      + {\rm O}(|\nmuk - \mu|) \to c_V(\Id, 0) \, \mu
  \end{align}
  pointwise a.e.~in $I \times \Omega$, where we again used that by \ref{C_bounds} the function $c_V$ is bounded, the pointwise convergence of $(\nabla \nyk)_k$ to $\Id$, and the pointwise convergence $\nmuk \to \mu$ (see \eqref{nmuk_pmuk_Ls_conv}, up to a subsequence).
   By \BBB Dominated Convergence this convergence also holds in $L^1(I \times \Omega)$.
  The same holds true for $ \pyk$, $\ptk$ in place of $\nyk$, $\ntk$.
  Thus, recalling the definition of $\awk$, we have shown that
  \begin{equation}\label{eq: XXXX}
    \int_{0}^{T} \int_\Omega \lp_k^{-\alpha} \awk \dot \vphi \di x \di t
    \to \intQ c_V(\Id,0) \mu \dot \vphi \di x \di t
    = \intQ \bar c_V \mu \dot \vphi \di x \di t.
  \end{equation}
  We now  prove \BBB that the contribution of the coupling potential \NNN vanishes \BBB in the limit.
  Indeed, by \eqref{C_locally_lipschitz},  \eqref{a_priori_Linfty_W2p_lin}, \BBB \eqref{temp_inten_Lq_bound}, \NNN \eqref{hoelder_conv_nuk_puk}, \BBB  and $t \wedge 1 \le t^{s/2}$ for some $s > \frac{2(\alpha-1)}{\alpha}$ with $s \in (1,  \frac{d+2}{d})$, the Cauchy-Schwarz and H\"older's inequality we see that   \BBB
  \begin{align*}
    \Big|
      \int_{0}^{T} & \int_\Omega\lp_k^{1-\alpha} \cplpot(\nabla  \NNN \pyk, \BBB \ptk)
        : \nabla \dotauk \vphi \di x \di t
    \Big|
    \leq \lp_k^{1-\alpha} \int_0^{T} \int_\Omega
      C (\ptk \wedge 1) (1 + \abs{\nabla \NNN \pyk \BBB - \Id})
      \abs{\nabla \dotauk} |\varphi| \di x \di t \nonumber \\
    &\leq C \lp_k^{1-\alpha} \Vert \ptk^{\frac{s}{2}} \Vert_{L^2(\Omega)}
      \Vert \nabla \dotauk \Vert_{L^2(\Omega)} \Vert \varphi \Vert_{L^\infty(\Omega)}
    \leq C \lp_k^{1 - \alpha + \alpha s / 2} \Vert \pmuk \Vert^{\frac{s}{2}}_{L^s(\Omega)} \NNN  \Vert \nabla \dotauk \Vert_{L^2(\Omega)} \Vert \varphi \Vert_{L^\infty(\Omega)} \BBB \to 0.
  \end{align*}
  Lastly, \NNN by \eqref{def_Ck_dotCk}, \BBB by the second convergence in \eqref{strong_rescaled_strain_rate_comp},  \eqref{diss_rate}, \BBB  and the continuity of $D$ one can show for $\alpha = 2$ that
  \begin{equation*}
      \int_{0}^{T} \int_\Omega
      \lp_k^{-\alpha} \rdrate(\nabla \pyk, \nabla \dotayk, \ptk) \vphi = \BBB \int_{0}^{T} \int_\Omega
      \lp_k^{-\alpha} \drate(\nabla \pyk, \nabla \dotayk, \ptk) \vphi
    \to \intQ \CD e(\dot u) : e(\dot u) \vphi \di x \di t.
  \end{equation*}
  For $\alpha < 2$ \NNN instead, \BBB it is easy to check \CCC using $\rdrate \leq \drate$ \BBB that the term vanishes as $k \to \infty$.
  Collecting all convergences and recalling the definition of  $\CD^{(\alpha)}$  \BBB in \eqref{alpha_dep}, we get that \eqref{linear_evol_temp} holds true, \NNN where for the external temperature we use \eqref{ea: just copy-after}. \BBB
\end{proof}

\begin{lemma}[Uniqueness of the linearized system]\label{lem:unique_mu}
  There exists at most one solution in the sense of Definition \ref{def:weak_form_linear_evol}.
\end{lemma}

\begin{proof}
  We start with $\alpha \in (1,2]$.
  In this case, \eqref{linear_evol_mech} is independent of the variable $\mu$.
  We show uniqueness of $u$.
  To this end, we suppose that there exist two solutions $u_1, \, u_2$, and set $u \defas u_1 - u_2$. Then $u = 0$ on $I \times \Gamma_D$ and $u(0) = 0$.
  Subtracting the weak formulations \eqref{linear_evol_mech} for both $u_1$ and $u_2$, we see that for any $z \in C^\infty(I\times \overline{\Omega}; \R^d)$ with $z = 0$ on $I \times \Gamma_D$ it holds that
  \begin{equation}\label{linear_evol_udif}
    \intQ \Big(\CW e(u) + \CD e(\dot u)\Big) : \nabla z \di x \di t = 0.
  \end{equation}
  Let us now define
  \begin{equation*}
    a(t) \defas \frac{1}{2} \int_\Omega \CD e(u(t)) : e(u(t)) \di x \quad \text{ for } t \in I.
  \end{equation*}
  Note that $a \in W^{1, 1}(I)$ with
  \begin{equation*}
    \dot a(t) = \int_\Omega \CD e(\dot u(t)) : e(u(t)) \di x
    = \int_\Omega \CD e(\dot u(t)) : \nabla u(t) \di x
  \end{equation*}
  for a.e.~$t \in I$.
  Let $\tilde \vphi \in C^\infty(I)$. Testing (\ref{linear_evol_udif}) with a sequence of smooth maps $(z_h)_h$ vanishing on $I \times \Gamma_D$ and converging  to $\tilde \vphi u$ in $L^2(I; H^1(\Omega))$ we derive that
  \begin{equation*}
    \int_0^T \tilde \vphi \int_\Omega \Big( \CW e(u) + \CD e(\dot u) \Big) : \nabla u \di x \di t = 0.
  \end{equation*}
  By the arbitrariness of $\tilde \vphi$ it then follows \CCC for almost all \BBB $t \in I$ \NNN that \BBB
  \begin{align*}
    \int_\Omega \big( \CD e(\dot u(t)) + \CW e(u(t)) \big) : \nabla u(t) \di x
    =0.
  \end{align*}
  This shows
  $$\dot a(t) = \int_\Omega   \CD e(\dot u(t)) : \nabla u(t) \di x =  - \int_\Omega \CW e(u(t)) : e(u(t)) \di x \leq 0. $$
   \MMM As $a(0) = 0$, it follows that $a=0$, and therefore $u=0$. \BBB

  Now, given a unique $u \in H^1(I;H^1(\Omega))$, we see that \eqref{linear_evol_temp} is an equation in the variable $\mu$ only.
  More precisely, it corresponds to the weak formulation of a heat equation with $L^1$-data.
  Uniqueness has been provided in \cite[Proposition 1]{Roubicek98Nonlinear}. This finishes the proof in the case $\alpha \in (1,2]$.

  We now briefly give the argument for $\alpha =1$.
  In this case, \eqref{linear_evol_temp} does not depend on $u$ and uniqueness follows again from  \cite[Proposition 1]{Roubicek98Nonlinear}.
  Then, the term $\intQ \mathbb{B}^\alpha \mu : \nabla z  \di x \di t$ in \eqref{linear_evol_mech} is only a datum, and uniqueness of $u$ follows by repeating the argument starting with \eqref{linear_evol_udif}.
\end{proof}

\MMM We are now ready to prove Theorem \ref{thm:linearization_right_diag}. \BBB

\begin{proof}[Proof of Theorem \ref{thm:linearization_right_diag}]
  We start with the proof of Theorem \ref{thm:linearization_right_diag}(iii).
  First, by Lemmas \ref{lem:comp_uk}--\ref{lem:comp_muk}, we obtain limits $u \in \CCC H^1\BBB(I; H^1_{\Gamma_D}(\Omega; \R^d))$ and $\mu \in L^1(I; W^{1,1}(\Omega))$.
  In view of \eqref{hoelder_conv_nuk_puk}, \eqref{nmuk_pmuk_Ls_conv}, and Lemma~\ref{lemma: strong ratistrain}, the convergences stated in  the statement \BBB hold, up to selecting a subsequence.
  In particular, \eqref{hoelder_conv_nuk_puk3} and \eqref{nmuk_pmuk_Ls_conv3} show that the convergence holds for all three different interpolations.    By Propositions \ref{prop:linearization_mech} and   \ref{prop:linearization_therm} we see that $(u, \mu)$ is a weak solution in the sense of Definition~\ref{def:weak_form_linear_evol}.
  As the weak solution is unique by Lemma~\ref{lem:unique_mu}, Urysohn's subsequence principle implies that the convergence holds for the whole sequence.
  This concludes the proof of Theorem \ref{thm:linearization_right_diag}(i),(iii).

  We briefly describe the adaptions for Theorem \ref{thm:linearization_right_diag}(ii).
  First, in the compactness result we replace \eqref{hoelder_conv_nuk_puk} and \eqref{nmuk_pmuk_Ls_conv} by \eqref{hoelder_conv_nuk_puk2} and \eqref{nmuk_pmuk_Ls_conv2}, respectively.
  The linearization of the mechanical equation and the heat-transfer equation in Propositions \ref{prop:linearization_mech} and \ref{prop:linearization_therm}, respectively, can be derived along similar lines, by replacing the time discrete equations \eqref{discri-mech} and \eqref{discr-heat} with their time-continuous analogs in \eqref{weak_limit_mechanical_equation} and \eqref{weak_limit_heat_equation}, respectively.
  In a similar fashion, for the proof Lemma \ref{lemma: strong ratistrain}, we use the time-continuous energy balance \eqref{cont_energy_balance} in place of \eqref{disc_energy_balance-new}.
  The rest of the argument remains unchanged.
\end{proof}

\subsection{Proof of Theorem \ref{thm:linearization_left_bottom}}\label{sec: 5.2}
\NNN We start with a $\Gamma$-convergence result. \BBB With the notation from \NNN Subsections~\ref{sec:setting}--\ref{sec:nonlinear_scheme} \BBB we define for $k \in \{1, \ldots, T/\tau \}$ the functional $E_\lp^{(k)} \colon H^1_{\Gamma_D}(\Omega; \R^d) \to \R$ through $E_\lp^{(k)}(u) = +\infty$ if $u \notin W^{2, p}(\Omega; \R^d)$ and
\begin{equation}\label{def_E_lp}
\begin{aligned}
  E_\lp^{(k)}(u) &\defas
   \frac{1}{\lp^2} \mechen(\id + \lp u)
    +  \frac{1}{\lp^2} \cplen(\id + \lp u, \tst{k-1})
    + \frac{1}{\tau\lp^2} \diss(\yst{k-1},  \id + \lp u \BBB - \yst{k-1}, \tst{k-1}) \\
  &\phantom{\defas}\quad
    - \langle \lst{k}, u \rangle
    - \frac{1}{\lp^2} \int_\Omega
      \tst{k-1} \pl_\theta \cplpot(\nabla \yst{k-1}, \tst{k-1}) \di x
\end{aligned}
\end{equation}
if $u \in W^{2, p}(\Omega; \R^d)$.
Although the last term in \eqref{def_E_lp} does not influence the minimizers of $E_\lp^{(k)}$ for fixed $k$, it is needed to ensure the boundedness of $(|E_\lp^{(k)}|)_\lp$ as $\lp \to 0$ along sequences of minimizers.  Recall also $\mathcal{E}_\lp$ from \eqref{eq: rescaled toten}. \BBB

\begin{proposition}\label{prop:gamma_conv}
  Suppose that  $\sup_{\lp>0}\toten_\lp(\yst{k-1}, \tst{k-1}) <+\infty $ \BBB and $ \ust{k-1}_{\lp,\tau}  \coloneqq  \BBB \lp^{-1}(\yst{k-1} - \id) \to u_\tau^{(k-1)}$ \NNN strongly \BBB in $H^1(\Omega; \R^d)$ \NNN as $\eps \to 0$. \BBB Suppose that $\lp^{-\alpha} \BBB \theta_{\lp, \tau}^{(k-1)} \to \mu_\tau^{(k-1)}$ in $L^1(\Omega)$ and that the convergence holds in  $L^2(\Omega)$ if $\alpha=1$. \BBB     Then, the sequence $(E_\lp^{(k)})_\lp$, defined in \eqref{def_E_lp},  $\Gamma$-converges in the weak \CCC$H^1$\BBB-topology  to \BBB $\bar E_0^{(k)} \colon H^1_{\Gamma_D}(\Omega; \R^d) \to \R$ given by
  \begin{equation*}
    \bar E_0^{(k)}(u) \defas
     \int_\Omega  \Big(  \frac{1}{2} \CW e(u) : e(u) \di x
      + \frac{1}{2 \tau}  \CD e(\tilde{u}) : e(\tilde{u}) 
         + \bar c_V \mu_\tau^{(k-1)}   +\mu_\tau^{(k-1)} \mathbb{B}^{(\alpha)} \colon \nabla \tilde{u} \Big)  \di x    - \sprod{\lst k}{u},
  \end{equation*}
  where $\tilde{u} \coloneqq  u - u_\tau^{(k-1)}$, \BBB  $\CW, \, \CD$  as in \eqref{def_WD_tensors}, $\bar c_V$ as in \eqref{lin_heat_cond_cap},  and $\mathbb{B}^{(\alpha)} $ as in \eqref{alpha_dep}. \BBB
\end{proposition}

\begin{proof}
  All constants we encounter in this proof are implicitly assumed to be independent of $\lp$.
  We will work with the equivalent representation
  \begin{align}\label{def_E_lp_2}
    E_\lp^{(k)}(u) &=
       \frac{1}{\lp^2} \mechen(\id + \lp u)
           + \frac{1}{\tau\lp^2} \diss(\yst{k-1}, \id + \lp u - \yst{k-1}, \tst{k-1})
      + \frac{1}{\lp^2} \int_\Omega    \inten_\lp(\Id + \lp \nabla u, \tst{k-1}) \di x \BBB \notag  \\
    &\phantom{=}\quad
      + \frac{1}{\lp^2} \int_\Omega       \tst{k-1} \big(\pl_\theta \cplpot(\Id + \lp \nabla u, \tst{k-1}) - \pl_\theta \cplpot(\nabla \yst{k-1}, \tst{k-1}) \big)  \di x  - \langle \lst{k}, u \rangle , 
  \end{align}
  which can be derived from \eqref{def_E_lp} by adding and subtracting
  \begin{equation*}
    \frac{1}{\lp^2} \int_\Omega \tst{k-1} \pl_\theta \cplpot(\Id + \lp \nabla u, \tst{k-1})
  \end{equation*}
  and using the definition of $\inten$  in \eqref{Wint}. \\
  \textit{Step 1 (Mechanical energy bound):}
  Let $(u_\lp)_\lp \subset W^{2, p}_{\Gamma_D}(\Omega; \R^d)$ be a sequence such that $\sup_{\lp > 0} E_\lp^{(k)}(u_\lp) < \infty$.  We will show that then also $\sup_{\lp > 0} \NNN \eps^{-2}  \mechen \BBB (y_\lp) < \infty$, where we shortly wrote $y_\lp \defas \id + \lp u_\lp$.
  By the nonnegativity of $\inten$ and $\disspot$ we derive that  
  \begin{equation}\label{E_eps_initial_lower_bound}
  \begin{aligned}
    E_\lp^{(k)}(u_\lp)
    &\geq  \frac{1}{\lp^2} \mechen(y_\lp)
      + \frac{1}{\lp^2} \int_\Omega \tst{k-1}\big(
          \pl_\theta \cplpot(\nabla y_\lp, \tst{k-1})
          - \pl_\theta \cplpot(\nabla \yst{k-1}, \tst{k-1})
        \big) \di x   - \langle \lst k, u_\lp \rangle.
  \end{aligned}
  \end{equation}
  By the second bound in \ref{C_bounds}, Young's inequality with constant $\lambda$,  and $1\wedge t \le \sqrt{t}$ \NNN for $t \ge 0$ \BBB it follows that
  \begin{align*}
    &\tst{k-1} \big|
      \pl_\theta \cplpot(\nabla y_\lp, \tst{k-1})
      - \pl_\theta \cplpot(\nabla \yst{k-1}, \tst{k-1})
    \big| \\
    &\quad\leq C (\tst{k-1} \wedge 1)
      (1 + |\nabla y_\lp - \Id| + |\nabla \yst{k-1} - \Id|) \\
    &\quad\leq \frac{C}{\lambda} \tst{k-1}
      + C \lambda |\nabla y_\lp - \Id|^2
      + C \NNN \lambda \BBB |\nabla \yst{k-1} - \Id|^2.
  \end{align*}
  Integrating over $\Omega$ and using \eqref{H1_dist_to_id} \NNN as well as \BBB \ref{W_lower_bound_spec}  \MMM we get \BBB that
  \begin{align}
    &\frac{1}{\lp^2} \Big|
      \int_\Omega \tst{k-1}\big(
          \pl_\theta \cplpot(\nabla y_\lp, \tst{k-1})
          - \pl_\theta \cplpot(\nabla \yst{k-1}, \tst{k-1})
        \big) \di x
    \Big| \nonumber \\
    &\quad\leq \frac{C}{\lambda \lp^2} \int_\Omega \tst{k-1} \di x
      +   \frac{C \NNN \lambda \BBB}{\lp^2} \elen(\yst{k-1})
      + \frac{ C \lambda}{\lp^2}  \elen(y_\lp). \label{pl_theta_diff_bound}
  \end{align}
  Again by \eqref{H1_dist_to_id}, Poincar\'e's inequality, and Young's inequality with constant  $\lambda / \lp$ \BBB we see that
  \begin{equation*}
    |\langle \lst k, u_\lp \rangle|  =  \lp^{-1} |\langle \lst k, y_\lp - \id \rangle| \BBB
    \leq \frac{C}{\lambda } \norm{\lst k}_{H^{-1}}^2
      + C \frac{\lambda}{\lp^2} \elen(y_\lp).
  \end{equation*}
  Hence, combining the above estimate with \eqref{pl_theta_diff_bound} and \eqref{E_eps_initial_lower_bound},  and using H\"older's inequality \BBB we arrive at
  \begin{equation*}
    E_\lp^{(k)}(y_\lp) \geq (1 - C \lambda)\lp^{-2}\mechen(y_\lp)
      - \frac{C}{\lambda} \big(
        \toten_\lp(\yst{k-1}, \tst{k-1})
        +  \norm{\lst k}_{H^{-1}}^2
      \big).
  \end{equation*}
  Choosing $\lambda$ sufficiently small such that $1 - C\lambda \geq 1/2$ this leads  the desired bound. Consequently, in the sequel, we can assume that  \eqref{a_priori_Linfty_W2p_lin}--\eqref{a_priori_Linfty_H1_lin} holds for both $y_\lp$ and $\yst{k-1}$. \BBB 
  
  \textit{Step 2 ($\Gamma$-$\liminf$):} Let $(u_\lp)_\lp \subset H^1_{\Gamma_D}(\Omega; \R^d)$ be such that $u_\lp \weakly u$ weakly in $H^1(\Omega; \R^d)$.
  Without loss of generality we can assume that $\sup_{\lp > 0} E_\lp^{(k)}(y_\lp) < \infty$ and $\liminf_{\lp \to 0} E_\lp^{(k)}(u_\lp) = \lim_{\lp \to 0} E_\lp^{(k)}(u_\lp)$.
  In particular, we can select a subsequence (without relabeling) such that $\NNN \theta_{\lp, \tau}^{(k-1)} \BBB \to 0$ a.e.~in $\Omega$.
  We are now ready to compute the $\liminf$ of the various terms of $E_\lp^{(k)}(u_\lp)$.
  By  \eqref{a_priori_Linfty_W2p_lin} \BBB we see that $\nabla y_\lp \to \Id$ uniformly.
  Hence, by the weak convergence of $(u_\lp)_\lp$ in $H^1(\Omega; \R^d)$ we can show similarly to the derivation of \eqref{liminf_rescaled_en} that
  \begin{equation}\label{eq: se energy}
    \liminf_{\lp \to 0} \frac{1}{\lp^2}\mechen(y_\lp)
    \geq   \liminf_{\lp \to 0} \frac{1}{2}  \int_\Omega \CW \nabla u_\lp : \nabla u_\lp \di x \ge \BBB   
    \frac{1}{2} \int_\Omega \CW e(u) : e(u) \di x.
  \end{equation}
  As in the proof of \eqref{liminf_rescaled_dissrate}, it follows from the pointwise convergence of \NNN $(\yst{k-1})_\lp$ \BBB and $(\theta_{\lp, \tau}^{(k-1)})_\lp$ that
  \begin{equation*}
    \liminf_{\lp \to 0} \frac{1}{\tau\NNN \eps^2} \diss(\yst{k-1}, y_\lp - \yst{k-1}, \tst{k-1})
    \geq \frac{1}{2 \tau} \int_\Omega \CD e(u - u_\tau^{(k-1)})
      : e(u - u_\tau^{(k-1)}) \di x.
  \end{equation*}
  By the same argument as in  \eqref{eq: XXXX}, \BBB the $L^1$-convergence of $ \lp^{-\alpha}  \BBB \theta_{\lp \tau}^{(k-1)}$  implies that
  \begin{equation}\label{liminf_inten}
    \lim_{\lp \to 0} \frac{1}{\lp^\alpha} \int_\Omega \inten(\nabla y_\lp, \tst{k-1}) \di x = \bar c_V \int_\Omega \mu_\tau^{(k-1)} \di x.
  \end{equation}
  For the  remaining  coupling  term in \eqref{def_E_lp_2}, we  Taylor expand around $(\Id, \tst{k-1})$ and get \BBB  by the second bound in \ref{C_bounds}\CCC, \ref{C_thrid_order}, \BBB and \eqref{a_priori_Linfty_W2p_lin}, applied for  both $y_\lp$ and $\yst{k-1}$, that
  \begin{align*}
 &\tst{k-1}    \big| \big(\pl_\theta \cplpot(\nabla y_\lp, \tst{k-1}) 
      - \pl_\theta \cplpot(\nabla \yst{k-1}, \tst{k-1}) \big) - \partial_{\theta F} W^{\rm cpl}(\Id,  \tst{k-1} ) \colon \nabla (y_\lp - \yst{k-1})    \big| \\ &  \ \leq C(\CCC \tst{k-1} \wedge 1 \BBB) \big(\abs{\nabla y_\lp -\Id }^2 + \abs{\nabla \yst{k-1}-\Id}^2\big) \le C\lp^{1+\CCC \frac{2}{p}}(\CCC \tst{k-1} \wedge 1 \BBB) \big(\abs{\nabla u_\lp  } + \abs{\nabla \ust{k-1}_{\lp,\tau}}\big).
  \end{align*}
  pointwise a.e.~in $\Omega$.    Thus, by   repeating the \NNN argument \BBB in \eqref{eq: do it small!}  we derive 
  \begin{align*}
    &\lim_{\lp \to 0} \frac{1}{\lp^2}   \int_\Omega
      \tst{k-1}\big(
          \pl_\theta \cplpot(\nabla y_\lp, \tst{k-1})
          - \pl_\theta \cplpot(\nabla \yst{k-1}, \tst{k-1})
        \big) \di x \\
    &\quad= \lim_{\lp \to 0}   \int_\Omega \lp^{-1}  \BBB \theta_{\lp, \tau}^{(k-1)}  \partial_{\theta F} W^{\rm cpl}(\Id,  \tst{k-1} ) \colon \nabla (u_\lp - \ust{k-1}_{\lp,\tau}).
  \end{align*}  
  Thus, by the definition of  $\mathbb{B}^{(\alpha)}$  in \eqref{alpha_dep} and by repeating the argument in \eqref{lindiss2XXX}--\eqref{eq: do it small!} we conclude 
  \begin{equation}\label{liminf_rem}
 \frac{1}{\lp^2}  \int_\Omega
      \tst{k-1}\big(
          \pl_\theta \cplpot(\nabla y_\lp, \tst{k-1})
          - \pl_\theta \cplpot(\nabla \yst{k-1}, \tst{k-1})
        \big) \di x
    \to   \int_\Omega \mu_\tau^{(k-1)} \mathbb{B}^{(\alpha)} \colon \nabla( \NNN u \BBB - u^{(k-1)}_\tau)  \di x
   \end{equation}  
 as $\lp \to 0$. Finally, \BBB   notice that the weak convergence also implies $    \lim_{\lp \to 0} \sprod{\lst k}{u_\lp} = \sprod{\lst k}{u}$.  Combining all aforementioned  estimates \BBB we conclude the proof of the  $\Gamma$-$\liminf$.

  \textit{Step 3 ($\Gamma$-$\limsup$):}
  Let $u \in H^1(\Omega; \R^d)$ with $u = 0$ on $\Gamma_D$.
  By a standard approximation argument in Sobolev spaces we can assume without loss of generality that $u \in C^\infty(\Omega; \R^d)$.
   Choose $u_\lp = u$ for all $\lp$. We only need to check the convergence of the energy. \BBB 
  First, notice that by \eqref{H_upper_bound_spec} and $p > 2$
  \begin{equation*}
    \frac{1}{\lp^2} \int_\Omega \hypot(\nabla^2 y_\lp) \di x
    \leq \frac{1}{\lp^2} \int_\Omega \aC |\lp \nabla^2 u|^p \di x
    = \aC \lp^{p-2} \int_\Omega |\nabla^2 u|^p \di x \to 0,
  \end{equation*}
  where $y_\lp \defas \id + \lp u$.
  By a Taylor expansion we also see that
  \begin{equation*}
    \frac{1}{\lp^2} \int_\Omega \elpot(\nabla y_\lp) \di x
    = \frac{1}{2\lp^2} \int_\Omega \CW \lp \nabla u : \lp \nabla u \di x
      + \mathrm{O}\Big( \lp \int_\Omega |\nabla^3 u| \di x \Big)
    \to \frac{1}{2} \int_\Omega \CW e(u) : e(u) \di x.
  \end{equation*}
  \CCC Furthermore, using \BBB \ref{D_quadratic} we can write
  \begin{equation*}
    \NNN \frac{1}{\eps^{2}} \BBB \diss(\yst{k-1}, y_\lp - \yst{k-1}, \tst{k-1})
    = \frac{1}{2} \int_\Omega D(C_\lp, \tst{k-1}) \dot C_\lp : \dot C_\lp,
  \end{equation*}
  where $C_\lp \defas (\nabla \yst{k-1})^T \nabla \yst{k-1}$ and $\dot C_\lp \defas (\nabla u - \nabla u_{\lp, \tau}^{(k-1)})^T \nabla \yst{k-1} + (\nabla \yst{k-1})^T(\nabla u - \nabla u_{\lp, \tau}^{(k-1)})$.
  By the strong convergence of $(u_{\lp, \tau}^{(k-1)})_\lp$ in $H^1(\Omega; \R^d)$ it follows that $\dot C_\lp \to  2 \BBB e(u - u_\tau^{(k-1)})$ strongly in $L^2(\Omega; \R^{d \times d})$.
  Consequently,
  \begin{equation*}
    \frac{1}{\tau \NNN \eps^2} \diss(\yst{k-1}, y_\lp - \yst{k-1}, \tst{k-1})
    \to \frac{1}{2 \tau} \int_\Omega \CD e(u - u_\tau^{(k-1)}) : e(u - u_\tau^{(k-1)}).
  \end{equation*}
 The convergence of the terms  \eqref{liminf_inten} and \eqref{liminf_rem} follows as \CCC in the previous step\BBB. This concludes the proof.
\end{proof}

We  close with the proof of  Theorem \ref{thm:linearization_left_bottom}.

\begin{proof}[Proof of Theorem \ref{thm:linearization_left_bottom}]
  We prove the result by induction on $k$. 
  For the base case $k = 0$, we only need to check the convergences  and the energy convergence. \BBB
  In fact, setting $u_\tau^{(0)} \defas u_0$ and $\mu_\tau^{(0)} \defas \mu_0$, this directly follows from  \eqref{initial_cond} and repeating the argument in the $\Gamma$-$\limsup$ above. \BBB
  
  Suppose now that the statement is true for $k-1$ where $k \in \setof{1, \ldots, T/\tau}$.
  With \eqref{a_priori_Linfty_H1_lin}  we have \BBB $u_{\lp, \tau}^{(k)} = \lp^{-1}(\yst k - \id) \weakly u_\tau^{(k)}$ weakly in $H^1(\Omega; \R^d)$ (up to a subsequence).
  By the induction hypothesis it  holds \CCC that \BBB $u_{\lp, \tau}^{(k-1)} = \lp^{-1}(\yst{k-1} - \id) \to u_\tau^{(k-1)}$ strongly in $H^1(\Omega; \R^d)$ and  $\lp^{-\alpha} \tst{k-1} \weakly \mu_\tau^{(k-1)}$ weakly in $W^{1, r}(\Omega)$ for any $r \in [1, \frac{d+2}{d+1})$. Therefore, we also find \BBB $\lp^{-\alpha} \tst{k-1} \to \mu_\tau^{(k-1)}$ strongly in $L^1(\Omega)$.    If $\alpha =1$, Remark~\ref{rem: after temp}(iii) even yields convergence in $L^2(\Omega)$.  \NNN As also  $\sup_{\lp>0}\toten_\lp(\yst{k-1}, \tst{k-1}) <+\infty $ due to Lemma \ref{lemma: first a prioiro}, we can apply Proposition \ref{prop:gamma_conv}.  \BBB   By the fundamental theorem of $\Gamma$-convergence, $u_\tau^{(k)}$ is a minimizer of $\bar E_0^{(k)}$ and $E_\lp^{(k)}(u_{\lp, \tau}^{(k)}) \to \bar E_0^{(k)}(u_\tau^{(k)})$.
 As  $\bar E_0^{(k)}$ \NNN is \BBB strictly convex, \BBB  $u_\tau^{(k)}$ is the unique minimizer of the corresponding minimization problem.
  In particular,  the weak $H^1$-convergence of $(u_{\lp, \tau}^{(k)})_\lp$ holds true without selecting a subsequence.  Moreover, energy convergence implies that in \eqref{eq: se energy} equality holds. This along with weak convergence, as well as Korn's and Poincar\'e's inequality yields  \BBB $u_{\lp, \tau}^{(k)} \to u_\tau^{(k)}$ strongly in $H^1(\Omega; \R^d)$.  Clearly, $u_\tau^{(k)}$ satisfies \eqref{el_mech_step_lin}. \BBB

  Let $r \in [1, \frac{d+2}{d+1})$ \CCC and $s \in [1, \frac{d+2}{d})$\BBB.
  As $\tau > 0$ was fixed, we see by  \eqref{nmuk_pmuk_Ls_conv} \BBB that, up to selecting a subsequence, $\lp^{-\alpha} \tst k \to \mu_\tau^{(k)}$ weakly in $W^{1, r}(\Omega)$ and strongly in $L^s(\Omega)$.  This limit $\mu_\tau^{(k)}$ solves \eqref{el_thermal_step_lin}. Indeed,  testing \eqref{el_thermal_step}  (with $\rdrate$ in place of $\drate$) \BBB with $\vphi \in C^\infty(\bar \Omega)$ and dividing by $\lp^\alpha$ we can pass to the limit $\lp \to 0$,  and obtain  \eqref{el_thermal_step_lin} \BBB by an argument similar to the one in the proof Proposition \ref{prop:linearization_therm}  neglecting the time dependence. The main difference is that we do not perform integration by parts in time, but by using the argument in \eqref{eq: cof} we pass directly to the limit in the term
$$\frac{1}{\lp^\alpha} \int_\Omega    \tau^{-1} \big(    \wst k - \wst{k-1} \big) \vphi \di x \to \int_\Omega     \bar{c}_V  \tau^{-1} \big( \mu_\tau^{(k)} - \mu_\tau^{(k-1)} \big) \vphi \di x. $$  
 To conclude the induction step,    \BBB  it remains to show the uniqueness of $\mu_\tau^{(k)}$, which in particular will imply that the weak $W^{1, r}$-convergence holds true without selecting a subsequence.
  Suppose that $\tilde \mu_\tau^{(k)}$ also satisfies \eqref{el_thermal_step_lin}.
  Then, for the difference $\mu \defas \mu_\tau^{(k)} - \tilde \mu_\tau^{(k)}$ it holds that
  \begin{equation*}
    \int_\Omega \big( \NNN \bar{c}_V \BBB\frac{\mu}{\tau} \vphi + \mathbb{K}_0 \nabla \mu \cdot \nabla \vphi \big) \di x
      + \kappa \int_\Gamma \mu \vphi \di \haus^{d-1} = 0.
  \end{equation*}
  Taking a   smooth \BBB sequence $(\vphi_h)_h \NNN \subset C_c^\infty(\Omega)\BBB$ converging to $\chi(\mu)$ \CCC in $C^1$, \BBB where \CCC$\chi(t) \defas \arctan(t)$, \BBB this shows with \eqref{spectrum_bound_K}, \CCC $\chi(t) t \geq 0$ for all $t$\BBB,  and \CCC$\chi' \ge 0$ \BBB \BBB that $\int_\Omega \frac{\mu}{\tau} \chi(\mu) \di x = 0$.     As  \CCC $\chi(t) t \geq 0$ for all $t$ \BBB and $\chi(t) = 0$ if  and only if $t = 0$, we have proved $\mu \equiv 0$, and thus uniqueness holds.

(ii)   We only sketch the proof as it follows along the lines of the reasoning in Section \ref{sec:tau_to_zero_delta_fixed}.
  Let $\hat u_\tau$, $\overline{u}_\tau$, $\underline{u}_\tau$  be defined similar to \eqref{y_interpolations}, and use similar notation for $\mu$.  We first observe \CCC that \BBB  $\CCC(\BBB\hat u_\tau\CCC)_\tau\BBB$ is bounded in $H^1(I;H^1(\Omega;\R^d))$ and $\CCC(\BBB\hat{\mu}_\tau\CCC)_\tau\BBB$ is bounded in $L^r(I;W^{1,r}(\Omega))$. This follows from Lemmas \ref{lem:comp_uk}--\ref{lem:comp_muk} and \eqref{eq: again a convergence}. Additional control can be recovered from the estimates stated in Theorem \ref{thm:further_apriori_temp_bounds}. Thus,  \BBB  we can find $u \in H^1(I; H^1_{\Gamma_D}(\Omega; \R^d))$ such that $\nabla \dot{\hat u}_\tau \weakly \nabla \dot u$ and $\nabla \overline u_\tau \weakly \nabla u$ weakly in $L^2(I\times \Omega; \R^{d \times d})$.   Moreover, there exists $\mu \in L^1(I; W^{1,1}(\Omega))$ with $\mu \ge 0$ a.e.~such that the latter two convergences in \eqref{linearized_tau_zero_convs} can be derived (up to a subsequence) using the Aubin-Lions' theorem and by following the reasoning in Lemma~\ref{lem:pointwise_temp_conv}.  \BBB

  Using  \eqref{el_mech_step_lin} for every smooth $z \in L^2(I; H^1_{\Gamma_D}(\Omega; \R^d))$ and summing \CCC over \BBB every $k \in \setof{1, \ldots, T/\tau}$ we derive that
  \begin{equation*}
    \int_0^{T} \int_\Omega \big( \CW e(\overline u_\tau)  +  \underline{\mu}_\tau \mathbb{B}^{(\alpha)} + \CD e(\dot{\hat u}_\tau) \big) \BBB : \nabla z
       \di x \di t
    - \int_0^{T} \langle \overline \ell_\tau(t), z(t) \rangle \di t = 0.
  \end{equation*}
  Consequently, we can then pass to the limit $\tau \to 0$ in \CCC the above equality \BBB  which results in \eqref{linear_evol_mech}. \BBB 
  Using \eqref{el_thermal_step_lin} for every $k \in \setof{1, \ldots, T/\tau}$ we also see that for any $\vphi \in C^\infty(I \times \bar \Omega)$ with $\vphi(T) = 0$ it \CCC holds \BBB
  \begin{align*}
   &\int_0^{T} \int_\Omega
      \big(  \CD^{(\alpha)} e(\dot{\hat u}_\tau) : e(\dot{\hat u}_\tau) \, \vphi
      + \mathbb{K}_0 \nabla \overline \mu_\tau \cdot \nabla \vphi -\bar c_V \hat \mu_\tau \dot \vphi\big) \di x \di t
      + \kappa \int_0^{T} \int_{\Gamma} (\overline \mu_\tau - \overline \theta_{\flat, \tau}) \vphi \di \haus^{d-1} \\
      &\quad= \bar c_V \int_\Omega \mu_0 \vphi(0) \di x,
  \end{align*}
   where as usual we applied integration by parts.   In particular, \NNN as $\tau \to 0$ \BBB by \eqref{linearized_tau_zero_convs} \BBB we see that
  \begin{align}\label{eq. LLL enerty2}
    &\lim_{\tau \to 0} \int_0^{T} \int_\Omega
     ( -\bar c_V \hat \mu_\tau \dot \vphi
      + \mathbb{K}_0 \nabla \overline \mu_\tau \cdot \nabla \vphi) \di x \di t
      + \kappa \int_0^{T} \int_{\Gamma} (\overline \mu_\tau - \overline \theta_{\flat, \tau}) \vphi \di \haus^{d-1} \notag \\
    &\quad= \int_0^T \int_\Omega
      (-\bar c_V \mu \dot \vphi
      + \mathbb{K}_0 \nabla \mu \cdot \nabla \vphi) \di x \di t
      + \kappa \int_0^T \int_{\Gamma} (\mu - \theta_\flat) \vphi \di \haus^{d-1}.
  \end{align}
   We also find \BBB
    \begin{align}\label{eq. LLL enerty}
    \lim_{\tau \to 0}   \frac{1}{2} \int_\Omega \CW e(\bar u(t)) : e(\bar \NNN u(t) \BBB ) \di x
    &=\frac{1}{2} \int_\Omega \CW e( u_\tau(t)) : e( u_\tau(t)) \di x, \notag \\   
    \lim_{\tau \to 0} \int_0^{T} \int_\Omega \CD e(\dot{\hat u}_\tau) : e(\dot{\hat u}_\tau) \, \vphi \di x \di t
    &= \int_0^T \int_\Omega \CD e(\dot u) : e(\dot u) \, \vphi \di x \di t.
  \end{align} 
  Indeed, inequalities follow from weak convergence, and the equalities are recovered by resorting to energy balances in the time-discrete and time-continuous setting, \BBB see Lemma \ref{lem:strong_strain_rates_conv},  in particular \eqref{two ini}--\eqref{diss_convergence}, for details. Let us highlight that at this point for $\alpha = 1$ we exploit $ \int_0^T \int_\Omega \underline{\mu}_\tau \mathbb{B}^{(\alpha)} : \nabla \dot{\hat u}_\tau \di x \to \int_0^T  \int_\Omega\mu \mathbb{B}^{(\alpha)} : \nabla \dot{u}\di x$ since we can assume $\underline{\mu}_\tau \to \mu$ in $L^2(I;L^2(\Omega))$ by \NNN Remark \ref{rem: after temp}(iii). \BBB

The second part of \eqref{eq. LLL enerty} along with \eqref{eq. LLL enerty2} implies \NNN that  \eqref{linear_evol_temp} holds. \BBB  This shows that $(u, \mu)$ is a weak solution of  \eqref{viscoel_small}--\eqref{initial_conds_lin} \BBB in the sense of Definition \ref{def:weak_form_linear_evol}.
   \MMM This \BBB solution is unique  (see Theorem \ref{thm:linearization_right_diag}(i)), \BBB  all aforementioned convergences hold true without selecting a subsequence.  Energy convergence in \eqref{eq. LLL enerty} along with weak convergence implies $\overline u_\tau(t) \to \overline u(t)$ strongly in $H^1(\Omega; \R^d)$ for every $t \in I$. For the other interpolations, one can argue in a similar fashion by replacing $\overline u_\tau(t)$ by $\underline u_\tau(t)$ in \eqref{eq. LLL enerty}. \BBB 
\end{proof}

\noindent \textbf{Acknowledgements} This work was funded by  the DFG project FR 4083/5-1 and  by the Deutsche Forschungsgemeinschaft (DFG, German Research Foundation) under Germany's Excellence Strategy EXC 2044 -390685587, Mathematics M\"unster: Dynamics--Geometry--Structure. The work was further supported by the DAAD project 57600633, and by the project  DAAD-22-03. M.K.~acknowledges support by  GA\v{C}R-FWF project 21-06569K and by the Erwin Schr\"{o}dinger International Institute for Mathematics and Physics  during his stay in Vienna in 2022.

\typeout{References}

\end{document}